\documentclass[10pt,a4paper,pointlessnumbers]{scrartcl}
\usepackage[T1]{fontenc}

\usepackage{amssymb}
\usepackage{amsmath}
\usepackage{amsfonts}
\usepackage{bbm}
\usepackage{amsthm}
\usepackage{lmodern}
\usepackage{mathrsfs}
\usepackage[hidelinks]{hyperref}\usepackage{color}
\usepackage[margin=2.3cm]{geometry}
\usepackage[all,cmtip]{xy}
\usepackage[utf8]{inputenc}
\usepackage{graphicx}
\usepackage{varwidth}
\usepackage{overpic}

\makeatletter
\DeclareRobustCommand{\em}{%
	\@nomath\em \if b\expandafter\@car\f@series\@nil
	\normalfont \else \slshape \fi}
\makeatother
\usepackage{upgreek}	
\usepackage{rotating}
\usepackage{tikz}
\usetikzlibrary{matrix,arrows,decorations.pathmorphing,shapes.geometric}
\usepackage{tikz-cd}
\usepackage{needspace}
\newcommand{\spaceplease}{\needspace{10\baselineskip}}
\usetikzlibrary{decorations.markings}

\usetikzlibrary{backgrounds}

\newcommand{\PrL}{\catf{Pr}^\catf{L}}

\newcommand{\HOM}{\underline{\catf{Hom}}} \newcommand{\END}{\underline{\catf{\End}}}

\newcommand{\RGraphs}{\catf{RGraphs}}

\usepackage{tikz-cd}

\usetikzlibrary{decorations.markings}
\usetikzlibrary{backgrounds}

\usepackage{etex}

\tikzstyle{tikzfig}=[baseline=-0.25em,scale=0.5]

\pgfkeys{/tikz/tikzit fill/.initial=0}
\pgfkeys{/tikz/tikzit draw/.initial=0}
\pgfkeys{/tikz/tikzit shape/.initial=0}
\pgfkeys{/tikz/tikzit category/.initial=0}

\pgfdeclarelayer{edgelayer}
\pgfdeclarelayer{nodelayer}
\pgfsetlayers{background,edgelayer,nodelayer,main}

\tikzstyle{none}=[inner sep=0mm]

\newcommand{\tikzfig}[1]{%
	{\tikzstyle{every picture}=[tikzfig]
		\IfFileExists{#1.tikz}
		{\input{#1.tikz}}
		{%
			\IfFileExists{./figures/#1.tikz}
			{\input{./figures/#1.tikz}}
			{\tikz[baseline=-0.5em]{\node[draw=red,font=\color{red},fill=red!10!white] {\textit{#1}};}}%
	}}%
}

\tikzstyle{every loop}=[]

\usepackage{tikzit}

\tikzstyle{black dot}=[fill=black, draw=black, shape=circle, minimum size=3pt, inner sep=0pt]
\tikzstyle{reddot}=[fill=red, draw=red, shape=circle, minimum size=3pt, inner sep=0pt]
\tikzstyle{black dot small}=[fill=black, draw=black, shape=circle, minimum size=2pt, inner sep=0pt]
\tikzstyle{fblack dot}=[fill=black, draw=red, shape=circle, minimum size=2pt, inner sep=0pt]
\tikzstyle{wbox}=[fill=white, draw=black, shape=rectangle, minimum height=0.5cm, minimum width=0.01cm]
\tikzstyle{bbox}=[fill=white, draw=blue, shape=rectangle, minimum height=0.5cm, minimum width=0.01cm]
\tikzstyle{rbox}=[fill=white, draw=red, shape=rectangle, minimum height=0.5cm, minimum width=0.01cm]
\tikzstyle{bwbox}=[draw=blue, shape=rectangle, minimum width=2cm, minimum height=0.5cm]
\tikzstyle{bbwbox}=[draw=blue, shape=rectangle, minimum width=1cm, minimum height=1cm]
\tikzstyle{big white circle}=[fill=white, draw=black, shape=circle, minimum width=0.75cm]
\tikzstyle{white dot big}=[fill=white, draw=black, shape=circle, inner sep=1pt]
\tikzstyle{white dot}=[fill=white, draw=black, shape=circle, minimum size=3pt, inner sep=0pt]
\tikzstyle{flat box}=[fill=white, draw=black, shape=rectangle, minimum width=1.3cm, minimum height=0.5cm,fill=morphismcolor]
\tikzstyle{square}=[fill=white, draw=black, shape=rectangle]
\tikzstyle{flat box 2}=[fill=white, draw=black, shape=rectangle, minimum height=0.5cm, minimum width=0.01cm,fill=morphismcolor]
\tikzstyle{bigbox}=[fill=white, draw=black, shape=rectangle, minimum height=0.5cm, minimum width=0.8cm,fill=white]
\tikzstyle{over }=[front]
\tikzstyle{theta}=[fill=blue, draw=blue, shape=ellipse, minimum height=6pt, minimum width=6pt, inner sep=0pt]
\tikzstyle{thetabig}=[fill=blue, draw=blue, shape=ellipse, minimum width=1cm, minimum height=0.01cm]
\tikzstyle{thetainv}=[fill=blue, draw=red, shape=ellipse, minimum height=6pt, minimum width=6pt, inner sep=0pt]
\tikzstyle{thetabinv}=[fill=blue, draw=red, shape=ellipse, minimum width=1cm, minimum height=0.01cm]
\tikzstyle{bigdisk}=[draw=black, shape=circle, minimum width=0.5cm]
\tikzstyle{wdisk}=[shape=circle, minimum width=0.48cm,fill=white]
\tikzstyle{bigdisk2}=[draw=black, fill=lightgray, shape=circle, minimum width=3cm]
\tikzstyle{little disk}=[fill=white, draw=black, shape=circle, minimum width=0.5cm]

\tikzstyle{mid arrow}=[-, postaction={on each segment={mid arrow}}]
\tikzstyle{end arrow}=[->]
\tikzstyle{barrow}=[<->]
\tikzstyle{mover}=[-, link]
\tikzstyle{string}=[-, draw=blue,postaction={on each segment={mid arrow}}]
\tikzstyle{stringd}=[-, dotted,draw=blue,postaction={on each segment={mid arrow}}]
\tikzstyle{red}=[-,draw=red]
\tikzstyle{mydots}=[-,dotted,dashed,draw=gray]
\tikzstyle{mydotsblack}=[-,dotted,dashed,draw=black]
\tikzstyle{open}=[-, line width=1pt,draw=blue]
\tikzstyle{redopen}=[-, line width=1pt,draw=red]
\tikzstyle{yellowopen}=[-, line width=1pt,draw=yellow]
\tikzstyle{greenopen}=[-, line width=1pt,draw=green]
\tikzstyle{thick}=[-,line width=1pt]
\tikzstyle{rarrow}=[->,draw=red]
\tikzstyle{red mid arrow}=[-, draw={rgb,255: red,214; green,42; blue,51}, postaction={on each segment={mid arrow}}, line width=1pt]
\tikzstyle{RED}=[-, draw={rgb,255: red,214; green,42; blue,51}]
\tikzstyle{REDthick}=[-, draw={rgb,255: red,214; green,42; blue,51},line width=2.5pt]
\tikzstyle{REDdashed}=[-,dashed, draw={rgb,255: red,214; green,42; blue,51}]
\tikzstyle{REDarrow}=[->, draw={rgb,255: red,214; green,42; blue,51}]
\tikzstyle{darrow}=[->,dotted]
\tikzstyle{blue}=[-, draw=blue]
\tikzstyle{blue mid arrow}=[-, draw={rgb,255: red,23; green,37; blue,167}, postaction={on each segment={mid arrow}}, line width=1pt]
\tikzstyle{over}=[-, link]
\tikzstyle{bover}=[-, blink]
\tikzstyle{mover}=[-, link]
\tikzstyle{mapsto}=[{|->}]

\usetikzlibrary{decorations.pathreplacing,decorations.markings}
\tikzset{
	on each segment/.style={
		decorate,
		decoration={
			show path construction,
			moveto code={},
			lineto code={
				\path [#1]
				(\tikzinputsegmentfirst) -- (\tikzinputsegmentlast);
			},
			curveto code={
				\path [#1] (\tikzinputsegmentfirst)
				.. controls
				(\tikzinputsegmentsupporta) and (\tikzinputsegmentsupportb)
				..
				(\tikzinputsegmentlast);
			},
			closepath code={
				\path [#1]
				(\tikzinputsegmentfirst) -- (\tikzinputsegmentlast);
			},
		},
	},
	mid arrow/.style={postaction={decorate,decoration={
				markings,
				mark=at position .7 with {\arrow[#1]{stealth}}
	}}},
}
\tikzset{%
	link/.style    = { white, double = black, line width = 1.8pt,
		double distance = 0.4pt },
	channel/.style = { white, double = black, line width = 0.8pt,
		double distance = 0.8pt },
}
\tikzset{
	vtx/.style = {circle, fill, inner sep=2pt},
	edg/.style = {thick},
	squig/.style = {decorate, thick,
		decoration={snake, amplitude=.8mm,
			segment length=3mm, post length=2.5mm},
		-{Stealth[length=3mm,width=2.5mm]}}
}

\usetikzlibrary{calc}

\tikzset{%
	blink/.style    = { white, double = blue, line width = 2pt,
		double distance = 1pt },
	channel/.style = { white, double = blue, line width = 2pt,
		double distance = 1pt },
}

\usetikzlibrary{decorations.pathmorphing}

\tikzset{
	vtx/.style   = {circle, fill, inner sep=1.6pt},
	edg/.style   = {thick},
	squig/.style = {->, thick, decorate,
		decoration={snake, amplitude=.6mm,
			segment length=2.5mm, post length=1.5mm}}
}
\tikzset{
	vtx/.style    = {circle, fill, inner sep=2.4pt},
	edg/.style    = {thick},
	casing/.style = {draw=white, line width=6pt}   
}

\tikzstyle{tikzfig}=[baseline=-0.25em,scale=0.5]

\pgfkeys{/tikz/tikzit fill/.initial=0}
\pgfkeys{/tikz/tikzit draw/.initial=0}
\pgfkeys{/tikz/tikzit shape/.initial=0}
\pgfkeys{/tikz/tikzit category/.initial=0}

\pgfdeclarelayer{edgelayer}
\pgfdeclarelayer{nodelayer}
\pgfsetlayers{background,edgelayer,nodelayer,main}

\tikzstyle{none}=[inner sep=0mm]

\usetikzlibrary{arrows.meta,decorations.pathmorphing}

\tikzstyle{every loop}=[]

\usepackage{mathtools}
\mathtoolsset{showonlyrefs}
\usepackage{epstopdf}

\newtheoremstyle{mytheorem}
{\topsep}
{\topsep}
{\slshape}
{0pt}
{\bfseries}
{.}
{ }
{\thmname{#1}\thmnumber{ #2}\thmnote{ {\normalfont\slshape(#3)}}}

\newtheoremstyle{mydefinition}
{\topsep}
{\topsep}
{\normalfont}
{0pt}
{\bfseries}
{.}
{ }
{\thmname{#1}\thmnumber{ #2}\thmnote{ {\normalfont\slshape(#3)}}}

\theoremstyle{mytheorem}
\newtheorem{theorem}{Theorem}[section]

\makeatletter
\newtheorem*{rep@theorem}{\rep@title}
\newcommand{\newreptheorem}[2]{%
	\newenvironment{rep#1}[1]{%
		\def\rep@title{#2 \ref{##1}}%
		\begin{rep@theorem}}%
		{\end{rep@theorem}}}
\makeatother

\newreptheorem{theorem}{Theorem}
\newreptheorem{corollary}{Corollary}
\newreptheorem{proposition}{Proposition}
\newtheorem{lemma}[theorem]{Lemma}
\newtheorem{proposition}[theorem]{Proposition}
\newtheorem{corollary}[theorem]{Corollary}
\theoremstyle{mydefinition}
\newtheorem{definition}[theorem]{Definition}
\newtheorem{question}[theorem]{Question}
\newtheorem{exx}[theorem]{Example}
\newenvironment{example}
{\pushQED{\qed}\exx}
{\popQED\endexx}
\newtheorem{remm}[theorem]{Remark}
\newenvironment{remark}
{\pushQED{\qed}\remm}
{\popQED\endremm}
\numberwithin{equation}{section}
\usepackage{enumitem}

\newenvironment{pnum}{\begin{enumerate}[topsep=2pt,parsep=2pt,partopsep=2pt,itemsep=0pt,label={(\roman{*})}]}{\end{enumerate}}

\DeclareMathSymbol{\Phiit}{\mathalpha}{letters}{"08}\let\Phi\undefined\newcommand{\Phi}{\Phiit}
\DeclareMathSymbol{\Psiit}{\mathalpha}{letters}{"09}\let\Psi\undefined\newcommand{\Psi}{\Psiit}
\DeclareMathSymbol{\Sigmait}{\mathalpha}{letters}{"06}\let\Sigma\undefined\newcommand{\Sigma}{\Sigmait}
\DeclareMathSymbol{\Xiit}{\mathalpha}{letters}{"04}
\DeclareSymbolFont{extraup}{U}{zavm}{m}{n}
\DeclareMathSymbol{\vardiamond}{\mathalpha}{extraup}{87}
\DeclareMathSymbol{\Lambdait}{\mathalpha}{letters}{"03}\let\Lambda\undefined\newcommand{\Lambda}{\Lambdait}
\DeclareMathSymbol{\Piit}{\mathalpha}{letters}{"05}\let\Pi\undefined\newcommand{\Pi}{\Piit}
\DeclareMathSymbol{\Gammait}{\mathalpha}{letters}{"00}\let\Gamma\undefined\newcommand{\Gamma}{\Gammait}
\DeclareMathSymbol{\Omegait}{\mathalpha}{letters}{"0A}\let\Omega\undefined\newcommand{\Omega}{\Omegait}
\DeclareMathSymbol{\Upsilonit}{\mathalpha}{letters}{"07}\let\Upsilon\undefined\newcommand{\Upsilon}{\Upilonit}
\DeclareMathSymbol{\Thetait}{\mathalpha}{letters}{"02}\let\Theta\undefined\newcommand{\Theta}{\Thetait}

\def\Hom{\catf{Hom}}
\let\O\undefined\newcommand{\O}{\catf{O}}

\def\End{\catf{End}}
\def\id{\mathrm{id}}

\def\dim{\mathsf{dim}}
\let\to\undefined\newcommand{\to}{\longrightarrow}
\let\mapsto\undefined\newcommand{\mapsto}{\longmapsto}
\newcommand{\catf}[1]{\mathsf{#1}}
\newcommand{\Proj}{\operatorname{\catf{Proj}}}

\newcommand{\Map}{\catf{Map}}

\def\op{\mathrm{op}}

\newcommand{\Prc}{\catf{Pr}_\catf{c}}

\newcommand{\Vmod}{V\catf{-mod}}
\newcommand{\Wmod}{\cat{W}_{2,3}\catf{-mod}}

\def\tr{\catf{tr}\,}
\newcommand{\Ao}{\cat{A}_{\text{\normalfont \bfseries !}}}

\newcommand{\Vo}{V_{\text{\normalfont \bfseries !}}}
\newcommand{\Nat}{\catf{Nat}}

\newcommand{\ra}[1]{\xrightarrow{\   #1    \ }}

\newcommand{\Lexf}{\catf{Lex}^\catf{f}}
\newcommand{\Rexf}{\catf{Rex}^\catf{f}}
\newcommand{\Rex}{\catf{Rex}}

\def\Cat{\catf{Cat}}

\newcommand{\moduli}{\mathfrak{a}}

\newcommand{\act}{\triangleright}

\newcommand{\sing}{\catf{S}}

\newcommand{\SNA}{\catf{SN}_\cat{A}}

\newcommand{\colim}{\operatorname{colim}}
\newcommand{\colimsub}[1]{\underset{#1}{\operatorname{colim}}\,}

\newcommand{\Lex}{\catf{Lex}}

\newcommand{\Vect}{\catf{Vect}}
\newcommand{\vect}{\catf{vect}}
\newcommand{\block}{\mathbf{\Omega}_V}
\newcommand{\blockl}{\mathbf{\Omega}^V}

\newcommand{\im}{\operatorname{im}}
\newcommand{\spr}[1]{\left\langle #1\right\rangle}
\newcommand{\cat}[1]{\mathcal{#1}}

\newcommand{\Graphs}{\catf{Graphs}}
\newcommand{\CGraphs}{C\catf{-Graphs}}
\newcommand{\TwoGraphs}{\catf{Graphs}_2}

\newcommand{\Legs}{\catf{Legs}}

\newcommand{\M}{\catf{M}}

\newcommand{\CF}{\catf{CF}}

\newcommand{\catT}{\mathscr{T}}
\newcommand{\Lin}{\catf{Lin}}
\newcommand{\Bimod}{\catf{Bimod}}

\newcommand{\LinCat}{\catf{LinCat}}
\newcommand{\RV}{\catf{Rep}(V)}
\newcommand{\RW}{\catf{Rep}(\cat{W}_{2,3})}

\newcommand{\nakal}{\catf{N}^\ell}\newcommand{\nakar}{\catf{N}^r}

\usepackage{dingbat}

\makeatletter
\newcommand{\monthyeardate}{%
	\DTMenglishmonthname{\@dtm@month}, \@dtm@year
}
\makeatother

\setkomafont{caption}{\small\slshape}

\setkomafont{section}{\rmfamily\large}
\setkomafont{subsection}{\normalfont\slshape}
\setkomafont{subsubsection}{\rmfamily}
\setkomafont{subparagraph}{\normalfont\scshape}

\newtheorem*{theorem*}{Theorem}
\newtheorem*{corollary*}{Corollary}

\makeatletter
\renewcommand\section{\@startsection {section}{1}{\z@}%
	{-3.5ex \@plus -1ex \@minus -.2ex}%
	{2.3ex \@plus.2ex}%
	{\normalfont\scshape\centering}}
\makeatother

\usepackage{titlesec}
\titleformat{\subsection}[runin]
{\normalfont\slshape}
{\thesubsection}
{0.5em}
{}
[.]

\usepackage{titletoc}
\dottedcontents{section}[1.5em]{\rmfamily}{1.5em}{0.2cm}
\dottedcontents{subsection}[5em]{\rmfamily}{2.3em}{0.2cm}

\begin{document}

		\begin{center}	\textbf{\large{Modular Functors with Singularities \\[0.5ex] from Vertex Operator Algebras \\[0.5ex] Beyond Rigidity and Finiteness}} \\ 
		\vspace{1cm}{\large Lukas Müller $^a$ \quad and \quad Lukas Woike $^b$ }\\ 	\vspace{5mm}{\slshape $^a$ Ludwig-Maximilians-Universität München
		\\	Department für Physik\\
			Theresienstrasse 37\\
			D-80333 München \\
			Germany}	\\[7pt] 	\vspace{5mm}{\slshape  $^b$ Université Bourgogne Europe\\ CNRS\\ IMB UMR 5584\\ F-21000 Dijon\\ France }\end{center}	\vspace{0.3cm}	
	\begin{abstract}\noindent 
For a vertex operator algebra $V$ and a certain category of its modules, we propose a construction for spaces of conformal blocks organized into an open-closed modular functor with singularities. This is inspired by the idea of implementing directly from the start the principle of holomorphic factorization. More precisely, using the strategy of modular extension introduced by Costello and developed further in our previous work, we build for each surface $\Sigma$ with at least one boundary component per path component and specified boundary labels attached to marked intervals or boundary circles a representation $\block(\Sigma;-)$ of the mapping class group of $\Sigma$. The construction can be described explicitly on generating Dehn twists. This approach is a priori independent from other constructions based on algebraic geometry or topological techniques involving e.g.~surgery, but we include an overview over the available comparisons. In the special case in which the module category of $V$ is a not necessarily semisimple modular category $\cat{A}$, the spaces $\block(\Sigma)$ are equivalent to the string-net spaces for $\mathcal{A}$ and hence to the modular functor for the Drinfeld center $Z(\mathcal{A})\simeq \bar{\mathcal{A}}\boxtimes\mathcal{A}$. However, the construction of $\block$ in this paper has the advantage of being available beyond rationality, rigidity, self-contragredience and finiteness. Moreover, we prove that $\block$ satisfies excision, is finite-dimensional in the $C_2$-cofinite case and produces in genus one a generalization of the elliptic double of Brochier-Jordan. We prove for the triplet $\mathcal{W}_{2,3}$ with non-exact fusion product that the boundary conditions introduced by Gaberdiel-Runkel-Wood produce, as expected by these authors, correlation functions, provided that one uses the notion of a modular functor with singularities that we develop. 
\end{abstract}
	
	\newpage
	
		\normalsize
	\tableofcontents
\newpage

	\section{Introduction and summary}
 One axiomatic mathematical approach to two-dimensional conformal field theory is based on the theory of \emph{vertex operator algebras}, see e.g.~\cite{fhl,frenkelbenzvi} or, for an overview of the overall algebraic framework, the encyclopedia entry~\cite{algcften}. 
	 Given a vertex operator algebra $V$ and some appropriate notion of modules over it that serve as boundary labels,
	 one of the main tasks is to come up with some notion of \emph{spaces of conformal blocks} and to organize them into a \emph{modular functor}~\cite{Segal,ms89,turaev,tillmann,baki}, a system of projective mapping class group representations compatible with gluing. 
	 In the world of algebraic geometry, these correspond to projectively flat vector bundles over an appropriate moduli space of curves. 
	 
	 For a general vertex operator algebra, the construction of spaces of conformal blocks that form a modular functor is, despite the enormous progress over the last thirty years, an open problem, and so is the comparison of the many different approaches.
	 In this paper, we propose a construction centered around the notion of a modular functor with singularities that we will compare to existing attempts to the extent possible.
	 Let us first give an account of some of the approaches together with 
	 their advantages and their shortcomings:
	 
	 \subsection*{The algebro-geometric approach} Through the algebro-geometric construction of Frenkel-Ben-Zvi~\cite{frenkelbenzvi}, developed further in the work of Damiolini-Gibney-Tarasca~\cite{DGT1,DGT2},
	 one can --- more or less directly --- build spaces of conformal blocks $\mathbb{V}$ from a vertex operator algebra $V$. This is a construction that is a priori possible in great generality. 
	 Except for some extremely restrictive special cases, it is not known whether these spaces of conformal blocks produce mapping class group representations with any meaningful gluing properties, let alone form a modular functor.
	 In the strongly rational case, the gluing problem was solved in \cite{DGT2}. The fact that $\mathbb{V}$ forms a 
	 modular functor was expected by Frenkel and Ben-Zvi in the strongly rational case and proved in \cite{damioliniwoike}. For certain example classes, this was already established by Bakalov-Kirillov~\cite{baki} (partly inspired by ideas of Beilinson-Feigin-Mazur~\cite{BFM}) and Andersen-Ueno~\cite{andersenuenomf}.
	 For $C_2$-cofinite vertex operator algebras, compatibilities with certain types of gluing are treated by Gui-Zhang in~\cite{guizhang}, see also the article \cite{zhang} highlighting several significant problems such as jumping of dimensions for spaces of conformal blocks for smooth and nodal curves.
	 In any case, at the time of writing, no proof that  $\mathbb{V}$ forms a modular functor (or any structure that comes close to it) is available beyond the strongly rational case.\label{approach1}
	 
	 \subsection*{The topological approach}	 Instead of trying to build the spaces of conformal blocks using algebraic geometry, one can try to put more structure on a suitable representation category of $V$ and feed this category into the many constructions offered by quantum topology. In the strongly rational case, a result of Huang~\cite{huang} tells us that we can obtain a modular fusion category from which we may even build a once-extended three-dimensional topological field theory by means of the Reshetikhin-Turaev construction~\cite{rt1,rt2,turaev} and hence in particular a modular functor. The connection of the Reshetikhin-Turaev construction, that relies on surgery, to the construction of Frenkel-Ben-Zvi is generally not known~\cite[Section~3.3]{csrcft}, but by \cite[Theorem~5.5.1]{damioliniwoike}  $\mathbb{V}$ is canonically equivalent to the 
	 modular functor constructed from a geometrically constructed modular fusion category via the Reshetikhin-Turaev construction. For a slightly different comparison  applicable to certain example classes, see also \cite{andersenuenort}.
	 Beyond the rational case, no such comparison result is available.
	 Moreover, the topological approach becomes also a lot more involved: Under rather strong assumptions, one may find, on suitable categories of $V$-modules, the structure of a not necessarily semisimple modular category, see \cite{gannonnegron,mcrae}.
	 For a modular category, the construction of a modular functor was given by Lyubashenko~\cite{lyubacmp,lyu,lyulex,kl}, and as long as one demands finiteness, rigidity and compatibility with orientation reversal, the Lyubashenko construction gives us \emph{all} possible modular functors as shown in \cite[Theorem~5.11]{reflection}. Without this assumption, this is not true \cite{brochierwoike,microcosm}.
	 The limitations of this approach are different from the ones for the algebro-geometric context, but also rather heavy: For general conformal field theories, we cannot hope to obtain a modular category, at least with its current definition that includes finiteness and rigidity. Numerous examples  live outside the finite or rigid framework, see e.g.~\cite{grw,ghost} or also the table in \cite[Section~1.3]{crsy}.
	 
	 \subsection*{The goal of this paper}	In this paper, our goal is to  use  
	 the theory of cyclic operadic algebras in linear categories 
	 to propose a construction of spaces of conformal blocks that works at a high level of generality (no rationality, no rigidity, no self-contragredience, no finiteness) and that still produces spaces of conformal blocks for which we can establish the needed mapping class group representations and gluing properties.
	 The structures that we will find will, in some cases,
	 conform  to the axiomatic framework for modular functors from \cite{brochierwoike}, but we will also see  some new features, most notably the inclusion of singularities. 
	 The construction that we give in this paper has been prepared through our previous work on cyclic algebras \cite{cyclic,mwdiff,mwansular,envas} inspired by concepts from \cite{gk,gkmod,costello}, but the present paper is self-contained and many of the construction procedures are surprisingly elementary.

	 \subsection*{A new angle on the process of holomorphic factorization}	The  observation that forms the starting point for this paper is 
	 that the problem mentioned above is slightly more complicated than it needs to be.
	 To see this, it makes sense to just pretend for a second that we can actually build (some version of) the spaces of conformal blocks that we are looking for:
	 What would we actually do --- from the perspective of conformal field theory --- if we managed to 
	  produce a modular functor $\mathfrak{F}$ (or something closely related)?
	 The classical answer, see e.g.~\cite{algcften}, is that we would evaluate the modular functor on each surface \emph{and the surface with opposite orientation}; in other words, we consider $\bar{\mathfrak{F}}\boxtimes \mathfrak{F}$, where the bar indicates orientation reversal.
	 This combines `left and right movers' and is a part of the process of \emph{holomorphic factorization}. Afterwards, one would start searching for correlators to pass to the \emph{full conformal field theory}, see~\cite{jfcs,fspivotal} for more background. 
	 The key insight is that $\bar{\mathfrak{F}}\boxtimes \mathfrak{F}$ is actually a much more tractable object than $\mathfrak{F}$, and it is conceivable that there are situations in which $\bar{\mathfrak{F}}\boxtimes \mathfrak{F}$ admits a generalization while $ \mathfrak{F}$ does not.
	 In fact, the idea of combining left and right movers simply through a $\boxtimes$-tensoring is tested at best up to the finite rigid case~\cite{jfcs,fspivotal,correlators}, but not beyond.
	 The idea to re-imagine the process of holomorphic factorization is admittedly a little outside of the box, but also not entirely unmotivated: Combining left and right movers might only make sense if the spaces of conformal blocks are naturally compatible with orientation reversal as discussed in \cite{reflection}, and while that article focuses on the finite rigid case, we can already see in several places that a generalization beyond the finite rigid case would lead to all sorts of problems. 
	 We should also note that going straight to the modular functor obtained \emph{after} combination of left and right movers and skipping the troublesome step of passing through $\mathfrak{F}$ is not that unusual. In fact, in the rational case, a solution in such terms is known: 
	 Fuchs-Schweigert-Yang  give in \cite{rcftsn} a construction of conformal blocks plus correlators for rational conformal field theories based on a string-net model~\cite{levinwen,kirillovsn}. This construction builds directly an open-closed modular functor equivalent to the Reshetikhin-Turaev modular functor for the Drinfeld center $Z(\cat{A})\simeq \bar{\cat{A}}\boxtimes\cat{A}$
	 of a modular fusion category $\cat{A}$ (the Turaev-Viro modular functor \cite{turaevviro} for $\cat{A}$) without using
	  the modular functor for $\cat{A}$. 
	 The string-net model for the modular functor for $Z(\cat{A})$ exists still in the finite rigid case~\cite{sn}.  Also for the solution to finding consistent systems of correlators in the non-semisimple case~\cite{correlators},
	  the modular functor $\bar{\mathfrak{F}}\boxtimes \mathfrak{F}$ is treated again directly without passing through $\mathfrak{F}$.
	
	\spaceplease  Based on this motivation, we can now give a  list of the results achieved in this paper:

	 \begin{enumerate}[topsep=2pt,parsep=2pt,partopsep=2pt,itemsep=0pt,label={(\Alph{*})}]
	 	\item We give a construction of spaces of conformal blocks and modular functors with singularities
	 	without requiring rationality, rigidity, self-contragredience or finiteness. 
	 	\label{goala}
	 	\item We prove that this construction generalizes the spaces of conformal blocks \emph{after combination of left and right movers}, i.e.\ 
	 	we test the idea in the most general case in which such a combination has been described, which seems to be the finite rigid case~\cite{jfcs,fspivotal}.
	 	\label{goalb}
	 	\item We prove structural results about the spaces of conformal blocks such as excision, finiteness in the $C_2$-cofinite case 
	 	 and provide a description of the mapping class group representations on generators. \label{goalc}
	 	\item We demonstrate that, even beyond the finite rigid case, our construction can be used to confirm expectations
	 	formulated previously 
	 	 \cite{grw} based on physically motivated ideas.
	 	 \label{goald}
	 \end{enumerate}
	 The application of the theory of cyclic algebras to vertex operator algebras was already alluded to and prepared in \cite{mwdiff,mwansular,envas}, and the present paper finishes the construction by building open-closed modular functors and correlators \emph{with singularities}.
	 It is hopefully written in a way such that it can be applied by experts in the field of vertex operator algebras, a group to which the authors of this article most certainly do not belong. Although we make an effort to spell out all constructions as explicitly as we can, it is natural that the attempt to apply the construction proposed in this article to rather complicated vertex operator algebras leads at the end of the day to possibly hard representation-theoretic questions, and we feel that is important to record those. 
	 Therefore, the last goal of this article will be the following:	
	 \begin{enumerate}[topsep=2pt,parsep=2pt,partopsep=2pt,itemsep=0pt,label={(\Alph{*})}]\setcounter{enumi}{4}
	 	\item We identify and list the open problems that our topological construction naturally leads to   in the theory of vertex operator algebras and elsewhere.\label{goale}
	 \end{enumerate}
	 These are interspersed in the form \emph{questions} in the text on the pages \pageref{q1}, \pageref{q2}, \pageref{q22}, \pageref{q3}, 
	 \pageref{q4}, \pageref{qtrace} and \pageref{q5}.

	 \subsection*{The construction of spaces of conformal blocks}
	 In many constructions of conformal field theories,
	 the spaces of conformal blocks have a strong locality behavior and are part of an \emph{open-closed modular functor} that assigns spaces of conformal blocks not only to surfaces with boundary components, but also surfaces with marked intervals. 
	 Since passing from marked intervals to boundary circles has proven to be a good way to automatically obtain the modular functor combining left and right movers \cite{ffrsunique,rcftsn,sn}, we will start our considerations in the open sector.
	 
	 The following would certainly be a common denominator of all constructions that one would consider to be reasonable: For a vertex operator algebra $V$, we fix
	 a suitable collection of generalized $V$-modules that are organized into a linear category $\catT$. 
	 This linear category should be thought of as a collection of `generating' $V$-modules; it does not even have to be abelian and will generally not contain $V$ itself. There is no reason to already limit the construction at this point. We will allow \emph{anything} and just ask for the absolutely minimal amount of structure and relations.
	 We will need spaces of conformal blocks $\spr{t,u} , \spr{s\mid t,u} \in\Vect$ (not necessarily finite-dimensional) for a disk with two and three marked intervals:
	 \begin{align}
	 	\begin{tikzpicture}[scale=0.5]
	 		\begin{pgfonlayer}{nodelayer}
	 			\node [style=none] (2) at (-1.75, 1) {};
	 			\node [style=none] (3) at (-0.75, 1) {};
	 			\node [style=none] (4) at (0.25, 1) {};
	 			\node [style=none] (5) at (1.25, 1) {};
	 			\node [style=none] (6) at (-0.75, 3) {};
	 			\node [style=none] (7) at (0.25, 3) {};
	 			\node [style=none] (8) at (-11, 1) {};
	 			\node [style=none] (9) at (-10, 1) {};
	 			\node [style=none] (10) at (-9, 1) {};
	 			\node [style=none] (11) at (-8, 1) {};
	 			\node [style=none] (13) at (-5, 2) {$\spr{t,u}$};
	 			\node [style=none] (15) at (4.75, 2) {$\spr{s \mid t,u}$};
	 			\node [style=none] (16) at (-10.5, 0.5) {$t$};
	 			\node [style=none] (17) at (-8.5, 0.5) {$u$};
	 			\node [style=none] (18) at (-1.25, 0.5) {$t$};
	 			\node [style=none] (19) at (0.75, 0.5) {};
	 			\node [style=none] (20) at (0.75, 0.5) {$u$};
	 			\node [style=none] (21) at (-0.25, 3.5) {$s$};
	 			\node [style=none] (22) at (-7, 2) {$\mapsto$};
	 			\node [style=none] (23) at (2.25, 2) {$\mapsto$};
	 		\end{pgfonlayer}
	 		\begin{pgfonlayer}{edgelayer}
	 			\draw (6.center) to (2.center);
	 			\draw [bend left=90, looseness=2.75] (3.center) to (4.center);
	 			\draw (5.center) to (7.center);
	 			\draw [bend left=90, looseness=3.25] (9.center) to (10.center);
	 			\draw [style=open] (2.center) to (3.center);
	 			\draw [style=open] (4.center) to (5.center);
	 			\draw [style=open] (6.center) to (7.center);
	 			\draw [style=open] (11.center) to (10.center);
	 			\draw [style=open] (8.center) to (9.center);
	 			\draw [bend left=90, looseness=2.25] (8.center) to (11.center);
	 		\end{pgfonlayer}
	 	\end{tikzpicture}\label{eqngeneratorsnew}
	 \end{align}
	 The dependence on $t,u\in\catT$ is covariant while the dependence on $s$ is contravariant.
	 The reader might imagine `$\spr{t,u}=\catT(t',u)$' (here $\catT(-,-)$ denotes the hom spaces of $\catT$ and $'$ denotes the contragredient module)
	 and `$\spr{s\mid t,u}=\catT(s,t\otimes u)$', where $\otimes$ is maybe a tensor product of $V$-modules, but careful: We do not even ask $\catT$ to be preserved by the operation of taking the contragredient, and most certainly we do not ask the existence of a tensor product preserving $\catT$. These situations are special cases to which we will return later.
	 
	 What we should ask is the symmetry of $\spr{-,-}$ in the two arguments up to coherent isomorphism
	  (the topology dictates this), and moreover
	   the existence of a space of conformal blocks $\Delta(t,u)$ \emph{contravariant} in $t$ and $u$, to be thought of as associated to a disk with two oppositely oriented intervals.
	 For these two types of disks, we have to ask gluing relations
	  \begin{align}
	 \begin{tikzpicture}[scale=0.5]
	 	\begin{pgfonlayer}{nodelayer}
	 		\node [style=none] (8) at (-11, 0) {};
	 		\node [style=none] (9) at (-10, 0) {};
	 		\node [style=none] (10) at (-9, 2) {};
	 		\node [style=none] (11) at (-8, 2) {};
	 		\node [style=none] (12) at (-7, 4) {};
	 		\node [style=none] (13) at (-6, 4) {};
	 		\node [style=none] (14) at (-11, 2) {};
	 		\node [style=none] (15) at (-10, 2) {};
	 		\node [style=none] (16) at (-7, 2) {};
	 		\node [style=none] (17) at (-6, 2) {};
	 		\node [style=none] (18) at (-5, 2) {$=$};
	 		\node [style=none] (19) at (-4, 4) {};
	 		\node [style=none] (20) at (-4, 0) {};
	 		\node [style=none] (21) at (-3, 0) {};
	 		\node [style=none] (22) at (-3, 4) {};
	 		\node [style=none] (23) at (-1, 2) {and};
	 		\node [style=none] (26) at (1, 4) {};
	 		\node [style=none] (27) at (2, 4) {};
	 		\node [style=none] (32) at (2, 2) {};
	 		\node [style=none] (33) at (1, 2) {};
	 		\node [style=none] (34) at (3, 2) {};
	 		\node [style=none] (35) at (4, 2) {};
	 		\node [style=none] (36) at (5, 2) {};
	 		\node [style=none] (37) at (6, 2) {};
	 		\node [style=none] (38) at (5, 0) {};
	 		\node [style=none] (39) at (6, 0) {};
	 		\node [style=none] (40) at (6, 0) {};
	 		\node [style=none] (41) at (7, 2) {$=$};
	 		\node [style=none] (42) at (8, 4) {};
	 		\node [style=none] (43) at (8, 0) {};
	 		\node [style=none] (44) at (9, 0) {};
	 		\node [style=none] (45) at (9, 4) {};
	 	\end{pgfonlayer}
	 	\begin{pgfonlayer}{edgelayer}
	 		\draw [style=open] (11.center) to (10.center);
	 		\draw [style=open] (8.center) to (9.center);
	 		\draw [style=open] (12.center) to (13.center);
	 		\draw [bend left=90, looseness=3.50] (15.center) to (10.center);
	 		\draw [bend right=90, looseness=3.25] (11.center) to (16.center);
	 		\draw (16.center) to (12.center);
	 		\draw (13.center) to (17.center);
	 		\draw (14.center) to (8.center);
	 		\draw (15.center) to (9.center);
	 		\draw [bend left=90, looseness=2.25] (14.center) to (11.center);
	 		\draw [bend right=90, looseness=2.25] (10.center) to (17.center);
	 		\draw [style=open] (20.center) to (21.center);
	 		\draw [style=open] (19.center) to (22.center);
	 		\draw [style=open] (27.center) to (26.center);
	 		\draw [bend right=150, looseness=0.00] (27.center) to (32.center);
	 		\draw (26.center) to (33.center);
	 		\draw [style=open] (38.center) to (40.center);
	 		\draw [style=open] (34.center) to (35.center);
	 		\draw [bend right=90, looseness=3.50] (32.center) to (34.center);
	 		\draw [bend left=90, looseness=3.25] (35.center) to (36.center);
	 		\draw (36.center) to (38.center);
	 		\draw (37.center) to (40.center);
	 		\draw [bend right=90, looseness=2.25] (33.center) to (35.center);
	 		\draw [bend right=90, looseness=2.25] (37.center) to (34.center);
	 		\draw [style=open] (43.center) to (44.center);
	 		\draw [style=open] (42.center) to (45.center);
	 		\draw (19.center) to (20.center);
	 		\draw (21.center) to (22.center);
	 		\draw (42.center) to (43.center);
	 		\draw (44.center) to (45.center);
	 	\end{pgfonlayer}
	 \end{tikzpicture}
	 \end{align}
	 that are implemented algebraically through the isomorphisms \begin{align} \int^{u\in \catT}\spr{t,u}\otimes\Delta(u,v)\cong \catT(v,t)\quad \text{and}\quad  \int^{v\in\catT} \Delta(u,v) \otimes \spr{v,t}\cong \catT(u,t)\ , 
	 	\end{align} where $\int^{v\in\catT}$ denotes the coend over $\catT$. 
	 We need this to be able to have a mechanism to change the status of boundary labels from incoming to outgoing and vice versa.
	 In other words, $\spr{-,-}$ is a non-degenerate symmetric pairing through bimodules aka profunctors, see also \cite{coendFS} for a detailed justification for the appearance of coends in this context.
	 Note that $\Delta$ is not structure; its existence is just a property.
	 Beyond that, we have to ask $\spr{s\mid t,u}$ to come with associators (again, the topology requires it) and to be unital up to coherent isomorphism (the latter is just a property). 
	 Let us emphasize that the structure is again through profunctors only. We are not asking $\catT$ to be a monoidal category!
	 As a further compatibility, we must ask for the blocks with three insertions to reduce to the ones with two if we insert the unit.
	 This structure is referred to as a \emph{cyclic category of $V$-modules}; we give the formal definition in~Section~\ref{seccyclicmodulecat}.
	 
	 This  basic structure is present on subcategories of $V$-modules, see Section~\ref{secexamples} for the many different situations in which it arises. Just to be clear: The assignments~\eqref{eqngeneratorsnew} are a choice, but our claim is that 
	 once it is fixed, we can build mapping class group representations in an essentially unique way that respects gluing.
	Of course, this raises the difficult question how many cyclic structures there are.
	In the rational case, we prove in Corollary~\ref{corfiltration} that there is at least a filtration by finite terms.
	
	 To demonstrate the calculation recipe for spaces of conformal blocks, take a torus $\mathbb{T}_1^2$ with one boundary component on which we mark an interval.
	 This surface can be obtained by gluing together three disks with three marked intervals each along marked intervals.
	 First we glue them together to form a disk with five marked intervals, and then we glue two further pairs of intervals together as indicated by the arrows in the following diagram:
	 \vspace*{-1cm}
	 \begin{align}
	 \begin{tikzpicture}[scale=0.5]
	 	\begin{pgfonlayer}{nodelayer}
	 		\node [style=none] (0) at (-10, 7) {};
	 		\node [style=none] (1) at (-8, 7) {};
	 		\node [style=none] (2) at (-10, 5) {};
	 		\node [style=none] (3) at (-17, -5) {};
	 		\node [style=none] (4) at (-15, -5) {};
	 		\node [style=none] (5) at (-10, 3) {};
	 		\node [style=none] (6) at (-8, 1) {};
	 		\node [style=none] (7) at (-4, -5) {};
	 		\node [style=none] (8) at (-6, -5) {};
	 		\node [style=none] (9) at (-8, -1) {};
	 		\node [style=none] (10) at (-8, -5) {};
	 		\node [style=none] (11) at (-10, -5) {};
	 		\node [style=none] (12) at (-10, -3) {};
	 		\node [style=none] (13) at (-12, -5) {};
	 		\node [style=none] (14) at (-14, -5) {};
	 		\node [style=none] (15) at (-10, -1) {};
	 		\node [style=none] (16) at (-10, 1) {};
	 		\node [style=none] (19) at (-8, -4) {};
	 		\node [style=none] (23) at (-8, 2.25) {};
	 		\node [style=none] (25) at (-10, 1.5) {};
	 		\node [style=none] (26) at (1.5, 0.75) {};
	 		\node [style=none] (27) at (1.5, -0.5) {};
	 		\node [style=none] (28) at (1.5, -0.5) {};
	 		\node [style=none] (29) at (4, 0.5) {};
	 		\node [style=none] (30) at (6, 0.5) {};
	 		\node [style=none] (31) at (4.25, 0) {};
	 		\node [style=none] (32) at (5.75, 0) {};
	 		\node [style=none] (33) at (7.25, 0) {$\mathbb{T}_1^2$};
	 		\node [style=none] (34) at (-4, 0) {};
	 		\node [style=none] (35) at (0, 0) {};
	 		\node [style=none] (36) at (-2, 0.5) {gluing};
	 		\node [style=none] (37) at (-9, 7.5) {};
	 		\node [style=none] (38) at (-9, -5.5) {};
	 		\node [style=none] (39) at (-16, -5.5) {};
	 		\node [style=none] (40) at (-5, -5.5) {};
	 		\node [style=none] (41) at (-9, 6.5) {$r$};
	 		\node [style=none] (42) at (-5, -4.5) {$v$};
	 		\node [style=none] (43) at (-9, 0.5) {$s$};
	 		\node [style=none] (44) at (-9, 1.5) {$s$};
	 		\node [style=none] (45) at (-9, -4.5) {$u$};
	 		\node [style=none] (46) at (-12.75, -4.5) {$t$};
	 		\node [style=none] (47) at (2.5, 0.25) {$t$};
	 		\node [style=none] (48) at (-15.75, -4.5) {$w$};
	 		\node [style=none] (49) at (-9, -0.5) {$q$};
	 		\node [style=none] (50) at (-9, -1.5) {$q$};
	 	\end{pgfonlayer}
	 	\begin{pgfonlayer}{edgelayer}
	 		\draw (15.center) to (5.center);
	 		\draw (5.center) to (4.center);
	 		\draw (3.center) to (2.center);
	 		\draw (2.center) to (0.center);
	 		\draw (1.center) to (6.center);
	 		\draw (6.center) to (7.center);
	 		\draw [style=open] (14.center) to (13.center);
	 		\draw (8.center) to (9.center);
	 		\draw (9.center) to (10.center);
	 		\draw (11.center) to (12.center);
	 		\draw (12.center) to (13.center);
	 		\draw (15.center) to (14.center);
	 		\draw [style=redopen] (7.center) to (8.center);
	 		\draw [style=redopen] (3.center) to (4.center);
	 		\draw [style=yellowopen] (0.center) to (1.center);
	 		\draw [style=yellowopen] (11.center) to (10.center);
	 		\draw [style=greenopen] (15.center) to (9.center);
	 		\draw [style=greenopen] (16.center) to (6.center);
	 		\draw [bend left=90] (26.center) to (27.center);
	 		\draw [style=open, bend left=90] (26.center) to (27.center);
	 		\draw [in=-30, out=30, looseness=20.75] (26.center) to (28.center);
	 		\draw [bend right=90, looseness=1.50] (29.center) to (30.center);
	 		\draw [bend left] (31.center) to (32.center);
	 		\draw [bend right=90] (26.center) to (28.center);
	 		\draw [style=end arrow] (34.center) to (35.center);
	 		\draw [style=barrow, in=-30, out=45] (37.center) to (38.center);
	 		\draw [style=barrow, bend right=90, looseness=0.75] (39.center) to (40.center);
	 	\end{pgfonlayer}
	 \end{tikzpicture}
	 	\end{align}

	 For the disk with five marked intervals, we obtain
	 \begin{align}
	 	\int^{s,q\in\catT} \spr{r \mid w,s} \otimes \spr{s \mid q,v} \otimes \spr{q \mid t,u} \ . 
	 	\end{align}
	 Next we glue along the pair of yellow intervals. These have already opposite variance, so we just take a coend.
	 For the two red ones, this is not the case, so we insert a $\Delta$ in between to remedy this.
	 This yields 
	 \begin{align}
	 	\int^{s,q,u,v,w\in\catT} \spr{u \mid w,s} \otimes \spr{s \mid q,v} \otimes \spr{q \mid t,u} \otimes\Delta(v,w) \ . 
	 	\end{align}
 	This is not a vector space that we actually associate to the surface $\mathbb{T}_1^2$, but rather to the surface $\mathbb{T}_1^2$ plus a marking $\mu=(R,g)$ for $\mathbb{T}_1^2$ (corresponding to the decomposition into disks used above), i.e.\ a ribbon graph $R$ and a mapping class $g : |R|\to \mathbb{T}_1^2$ identifying the surface obtained by geometrically realizing $R$ with $\mathbb{T}_1^2$.
	 We set
	 \begin{align}
	 	\block^\mu(\mathbb{T}_1^2;t) :=\int^{s,q,u,v,w\in\catT} \spr{u \mid w,s} \otimes \spr{s \mid q,v} \otimes \spr{q \mid t,u} \otimes\Delta(v,w) \ . 
	 	\end{align}
	 The functor $\block^\mu(\Sigma;-)$ can be built for any surface $\Sigma$ (always smooth, compact and oriented) with $n\ge 1$ parametrized boundary intervals (also called marked intervals), at least one per connected component and a marking $\mu$.
	 Instead of seeing $\block^\mu(\Sigma;-)$ as a functor out of $\catT^{\otimes n}$, we may equivalently see it as a cocontinuous functor $\block^\mu (\Sigma;-):\RV^{\boxtimes n} \to \Vect$ out of the locally presentable category $\RV$ obtained by free cocompletion of $\catT$.
	 Next we need to ask how the spaces of conformal blocks for different markings are related.
	 
	\begin{repproposition}{propmarking}
	If $\Sigma$ has at least one marked interval per connected component, then for two markings $\mu$ and $\nu$ of $\Sigma$, there is a canonical isomorphism
	\begin{align}\alpha_{\mu,\nu} : \block^\mu(\Sigma;-) \ra{\cong} \block^{\nu}(\Sigma;-) \end{align}
	such that $\alpha_{\mu,\mu}=\id$ and $\alpha_{\nu,\zeta}\circ\alpha_{\mu,\nu} = \alpha_{\mu,\zeta}$ for three markings $\mu,\nu$ and $\zeta$, and we obtain a marking independent space of conformal block functor
	$\block(\Sigma;-)$ as the functor 
	equipped with an isomorphism \begin{align} \psi_\mu :  \block^\mu(\Sigma;-) \ra{\cong} \block(\Sigma;-)\end{align} for every marking $\mu$
	satisfying the universal property
	that for two markings $\mu$ and $\nu$ the triangle
	\begin{equation}
		\begin{tikzcd}
			\block^{\mu}(\Sigma;-)  \ar[]{rrd}{\psi_\mu}  \ar[swap]{dd}{\alpha_{ \mu,	\nu}}  \\ && 	\block(\Sigma;-)
			\\	\block^{\nu}(\Sigma;-) \ar[swap]{rru}{\psi_\nu}
		\end{tikzcd} 
	\end{equation} commutes.
		\end{repproposition}
	 
	 One then proceeds  through a generalization of the Lego-Teichmüller game \cite{bakifm}:
	 The mapping class group of $\Sigma$ acts on the markings of $\Sigma$
	  by $f.\mu=f.(R,g):=(R,fg)$ for  $f\in \Map(\Sigma)$. 
	 Since $\block^{\mu}(\Sigma;-)$ for $\mu=(\Omega,g)$, up to canonical isomorphism, does not depend on $g$, the isomorphism $\psi_\mu :  \block^\mu(\Sigma;-) \to\block(\Sigma;-)$ from above
	 induces an isomorphism $ \psi_\mu^f:\block^{f.\mu}(\Sigma;-)\ra{\cong} 	\block(\Sigma;-)$.
	 Now $f$ gives us an automorphism $f_*$ of $\block(\Sigma;-)$, namely
	 the one making the square
	 \begin{equation}
	 \begin{array}{c}	\begin{tikzcd}
	 		\block^{\mu}(\Sigma;-)  \ar[]{rr}{\psi_\mu}  \ar[swap]{dd}{\alpha_{ \mu,	f.\mu}} && 	\block(\Sigma;-) \ar{dd}{f_*} \\ 
	 		\\	\block^{f.\mu}(\Sigma;-) \ar[swap]{rr}{ \psi_\mu^f}  && 	\block(\Sigma;-)  
	 	\end{tikzcd} 
 	\end{array}\label{eqnsqc}\tag{M}
	 \end{equation}
	 commute.
	 
	 Before stating the fact that this actually gives us a mapping class group action, let us comment on how to include parametrized boundary circles as well: For a boundary circle, as opposed to a marked interval, the functor defined on $\RV$ factors through the factorization homology $\int_{\mathbb{S}^1} \RV$ (the categorical Hochschild homology) of $\RV$ along the functor $L:\RV\to \int_{\mathbb{S}^1} \RV$ induced by embedding an interval into the circle, sometimes referred to as \emph{band tensor functor}~\cite{bzbj2}.
	 This  observation  has been made in \cite[Section~5]{envas} (and in \cite[Section~3]{sn}  in the finite rigid setting using the string-net picture) and is ultimately a lower-dimensional version of \cite[Section~4]{brochierwoike}.
	 The main result on the spaces of conformal blocks can now be formulated as follows:
	 
	 \begin{reptheorem}{thmmain}[Modular functor with singularities from a vertex operator algebra]
	 For any cyclic category of $V$-modules, the construction $\Sigma \mapsto \block(\Sigma;-)$ extends from surfaces with at least one parametrized interval per connected component  to surfaces $\Sigma$ with $n$ parametrized boundary intervals, $m$ parametrized boundary circles and $p$ punctures
	 such that $n+m+p\ge 1$ holds separately for each path component of $\Sigma$.
	 For such a surface $\Sigma$, we obtain a cocontinuous
	 functor
	 \begin{align}
	 	\block(\Sigma;-): \RV^{\boxtimes n} \boxtimes \left(\int_{\mathbb{S}^1} \RV\right)^{\boxtimes m} \to \Vect  
	 \end{align}carrying an action of $\Map(\Sigma)$.
	 These mapping class group representations assemble into an open-closed modular functor with singularities
	 assigning $\RV$ to the boundary interval, $\int_{\mathbb{S}^1} \RV$ to the boundary circle, and for which the gluing properties, also referred to as excision, take the following form:
	 \begin{pnum}
	 	\item Excision~1: For any two fixed parametrized boundary intervals of $\Sigma$, there are canonical mapping class group equivariant isomorphisms
	 	\begin{align}
	 		\block(\Sigma';-)\cong \int^{t\in\catT} \block(\Sigma; \dots,t^*,\dots,t,\dots) \ , 
	 	\end{align}
	 	where $\Sigma'$ is the surface obtained by gluing along the two intervals, $*$ is the weak duality operation for compact projective objects of $\RV$ induced by the pairing, and the dummy variables are inserted into the slots affected by the gluing.
	 	\item Excision~2: For any two fixed parametrized boundary circles of $\Sigma$, there are canonical mapping class group equivariant isomorphisms
	 	\begin{align}
	 		\block(\Sigma'_+;-) \cong \int^{t\in \catT} \block (\Sigma;\dots,Lt^*,\dots,Lt,\dots) \  ,
	 	\end{align}
	 	where $\Sigma'_+$ is the surface obtained by gluing two boundary circles plus an additional puncture and $L:\RV\to\int_{\mathbb{S}^1} \RV$ is the band tensor functor.
	 \end{pnum}
	 \end{reptheorem}
 
The necessity to include the puncture is what motivates the term \emph{modular functor with singularities}.

\subsection*{Modular extension of a presentable representation category of $V$}
 After this very explicit description of the construction, let us also give the more condensed and conceptual version
 that also allows a better comparison to existing constructions in the literature and will provide some shortcuts for the readers already familiar with  \cite{cyclic,mwansular,envas}:
 The structure to be fixed on the subcategory $\catT$ of $V$-modules
 amounts to a cyclic associative algebra up to coherent isomorphism in the sense of \cite{cyclic} inside the symmetric monoidal bicategory of linear categories and profunctors.
 Actually, we even ask for a cyclic framed $E_2$-structure (this means that we add a braiding and a balancing), but it is not needed for Theorem~\ref{thmmain}.
 By freely adding colimits we obtain a braided monoidal
 representation category $\RV$ in locally presentable categories whose category of compact projective objects can be identified with $\catT$.
 The mapping class group representations on the open sector are then produced via the modular extension procedure for $\RV$
  developed in \cite{envas} using important topological results \cite{giansiracusa} on the modular envelope of the cyclic associative operad \cite{costello} that partly rely on ribbon graph models for surfaces \cite{harer86,penner,kontsevichintersection,kon94,costellographs,costellotcft}, see also \cite{wahlwesterland,egaskupers} for more recent results.
  The reader familiar with \cite{envas}, even though it was mostly adapted to the finite case, or experts on cyclic and modular operads might be able to just skim Section~\ref{secmodext}, but should take note of the new description 
  through the translation procedure from markings to coends over compact projective objects that is needed for the Dehn twist calculations in Section~\ref{secgenerators}.
  
  An $(\infty,1)$-categorical version of 
   an extension starting from a cyclic associative algebra is discussed also in the work of Barkan-Steinebrunner-Zhang~\cite{bsz}.\label{bszcoparison}
  The input datum of this construction is an $E_1$-Calabi-Yau algebra. While the notes \cite{steinebrunner} suggest that, in the bicategorical case, these amount to cyclic associative algebras in our sense~\cite{cyclic} (which is important because otherwise we would not know how apply their construction to vertex operator algebras), we are not aware of a detailed	 proof. 
  For such an $E_1$-Calabi-Yau algebra, \cite{bsz} gives a partially defined open-closed modular functor,
  with the caveat that 
   the unique extension of the $E_1$-Calabi-Yau algebra to an open modular functor in \cite{bsz} relies on an announced result of Barkan-Steinebrunner.
  This is related to, but not identical to a modular functor with singularities in Theorem~\ref{thmmain} which features on the closed sector a \emph{degenerate} copairing which is the structure that we propose to treat for instance the $\cat{W}_{2,3}$-model and is very crucial for the existence of correlators.

  \subsection*{Properties of the spaces of conformal blocks}
  The main part of Theorem~\ref{thmmain} is not the fact that such a construction exists for some abstract conditions that are practically hard to verify (because \cite{envas} essentially does that, even though it is just on the open sector, with calculation recipes only developed in the finite case), but rather to what kind of categories of $V$-modules it is actually applicable and also the fact that the spaces of conformal blocks can be calculated from markings by means of a gluing prescription that translates the marking into a  coend over compact projective objects in the representation category of $V$.
  For example, for the torus with one boundary component, we can express $\block(\mathbb{T}_1^2)$ using the presentable module category $\RV$ and its pairing that we denote also by $\spr{-,-}$. 
	 For a compact projective object $t$ in $\RV$, we can define a \emph{weak dual} $t^*$
	 (Definition~\ref{defdual}) and prove
	 \begin{align}
	 	\block(\mathbb{T}_1^2;X)\cong \spr{X,\mathbb{F}^{\otimes 2}} \label{eqVVB3}
	 	\end{align}
 	for the coend $\mathbb{F}=\int^{t\in\catT} t^* \otimes t$ over the compact projective objects $\catT\subset \RV$ that however does not live in $\catT$ itself (this is why the cocompletion is so important).
 	We calculate the action of the mapping class group (which in this case is the braid group $B_3$
 	on three strands) in Corollary~\ref{corT12}.
 	In the rigid case, we use the holographic principle from \cite{correlators,yeralopen}
 	to prove that $\mathbb{F}^{\otimes 2}$
 	is an algebra with $B_3$-action, namely the \emph{elliptic double} of Brochier-Jordan \cite{brochierjordane}, and that this $B_3$-action actually induces the	 one in \eqref{eqVVB3}, see Corollary~\ref{corell}.
 	We should note, however that $\block$ exists beyond the rigid case and that generally $\mathbb{F}^{\otimes 2}$ has no reason to be an algebra.
 	For higher genus surfaces, we describe the action of the mapping class group on generators in Theorem~\ref{thmmcggen}.
 	
 	In Theorem~\ref{thmfinite} we prove that in the $C_2$-cofinite case $\block$ is finite-dimensional on compact objects.
 	This is a general expectation in conformal field theory
 	for any reasonable construction of spaces of conformal blocks, and to the best of our knowledge, this is the most general framework in which such a statement is a mathematically rigorous result.
 	Our proof mostly relies in the calculation recipes for 2-colimits in presentable categories given in \cite{chirjf}.
 	
 	\subsection*{Comparison results}
 	In the finite rigid case (we recall this framework  in Section~\ref{secfiniterigid}),
 	the presentable module category $\RV$ is the ind completion of a modular category $\cat{A}$ of $V$-modules, and
 	$\block$ recovers, as desired, the Lyubashenko modular functor $\mathfrak{F}_{\bar{\cat{A}}\boxtimes \cat{A}}$ for $\bar{\cat{A}}\boxtimes\cat{A}$
 	(Corollary~\ref{corfiniterigid}). We also give a description in terms of the moduli algebras of Alekseev-Grosse-Schomerus~\cite{alekseevmoduli,agsmoduli,asmoduli}
 	in Corollary~\ref{corfiniterigidmodulialgebra}.
 One might think that the approach of this article only gives access to the non-projective mapping class group representations of $\mathfrak{F}_{\bar{\cat{A}}\boxtimes \cat{A}}$, and never to the projective ones of $\mathfrak{F}_{\cat{A}}$, but actually it does through the \emph{Heisenberg picture} on spaces of conformal blocks from \cite{reflection}, with the terminology borrowed from \cite{jfheisenberg}:
 Roughly, the representations of $\mathfrak{F}_{\bar{\cat{A}}\boxtimes \cat{A}}$ inherit an algebra structure with  $\mathfrak{F}_{\cat{A}}$ being the only simple module; the details are a bit subtle and given in Corollary~\ref{corsimplemodule}.
 In other words, in the finite rigid case, we can take a `square root' of $\block$. 
Generally --- and this is maybe one of the main insights of this article --- $\block$ does not need to have any interpretation as a `square' of a different construction.

\subsection*{Beyond the rigid case: Spaces of conformal blocks and correlators for $\cat{W}_{2,3}$}
Without a rigid tensor product, the construction of $\block$ is much more difficult to describe, and 
it is a priori unclear whether the modular functor with singularities is useful.
We provide the following piece of evidence that this is indeed the case by looking at the triplet $\cat{W}_{2,3}$ studied by Gaberdiel-Runkel-Wood
in \cite{grw,grw-proc}:
The authors provide in these papers the notion of a boundary theory and expect it to produce a full open-closed conformal field theory. This is an abstract expectation because it is not formulated in reference to any rigorous construction of spaces of conformal blocks.
In Section~\ref{sectriplet23} we show that  our notion of modular functor with singularities 
naturally comes, through modular microcosm principle \cite{microcosm}, 
with a notion of correlators with singularities.
In this framework, we are then able to prove that, after a small cosmetic correction, the boundary theories from \cite{grw,grw-proc} actually provide a consistent system of open-closed correlators with singularities (Theorem~\ref{thmW23}); in particular, the chiral symmetry algebra $\cat{W}(0)$ provides such a system.
The search for other systems is reduced to the classification of symmetric Frobenius algebra in Grothendieck-Verdier categories for the notion of Frobenius algebra in~\cite{fsswfrobenius,microcosm}.
In Corollary~\ref{corcalculate} we give a description of the higher genus correlators in terms of the structure maps of the Frobenius algebra and the canonical end.

	\vspace*{0.2cm}\textsc{Acknowledgments.} 
We are grateful to Jorge Becerra, Chiara Damiolini,
Max Demirdilek, Alexander Kupers, Simon Lentner, Christoph Schweigert, Simon Wood and Deniz Yeral for helpful exchanges and comments related to this project. Parts of this work were presented at the Oberwolfach Workshop
on `Higher Structures from Symmetries in Quantum Field Theory' in March 2026, see \cite{woikemfo} for a brief summary. We thank the organizers for the invitation and the participants for their useful comments and questions.
The research underlying this article was conducted and the manuscript was written without the use of AI tools beyond literature searches with the goal to get an overview over the state-of-the-art of certain questions. Needless to say, for such searches, the source that we ultimately rely on for our manuscript is the original reference.
LW  gratefully acknowledges support
by the ANR project CPJ n°ANR-22-CPJ1-0001-01 at the Institut de Mathématiques de Bourgogne (IMB).
The IMB receives support from the EIPHI Graduate School (ANR-17-EURE-0002).

\needspace{10\baselineskip}
\section{Presentable categories of modules over a vertex operator algebra}
In this section, we single out the type of categories of vertex operator algebra modules that we use in the sequel.
We first focus on the abstract categorical algebra. The question how this structure can actually be obtained will be answered in the next section.

\subsection{A reminder on linear categories, bimodules and presentable categories\label{seclincat}}
For the facts recalled in this subsection on symmetric monoidal bicategories of linear categories,
we refer e.g.\ to \cite[Section~1.1]{skeinfin}.
More generally, for more	 background on symmetric monoidal bicategories, we refer to \cite[Chapter 2]{schommerpries}.

Denote by $\LinCat$ the symmetric monoidal bicategory of linear categories, linear functors and linear natural transformations, all over a fixed algebraically closed field $k$ that we suppress from the notation. The monoidal product is the `na\" ive' monoidal product $\otimes$ taking pairs of objects and monoidal products of morphism spaces. The monoidal unit is $k$, seen as one-object linear category.
We will denote the morphism spaces of linear and other categories $\mathscr{S}$ by $\mathscr{S}(-,-)$. 

Moreover, we denote by $\Bimod$ the symmetric monoidal bicategory of linear categories, bimodules (a bimodule from $\mathscr{S}$ to $\mathscr{T}$ is a linear functor $\mathscr{S}\otimes\mathscr{T}^\op \to \Vect$, where $\Vect$ is the category of $k$-vector spaces) and linear natural transformations of bimodules. The composition of bimodules 
$B:\mathscr{S}\otimes\mathscr{T}^\op \to \Vect$ and $C:\mathscr{T}\otimes\mathscr{U}^\op \to \Vect$ is the bimodule given by the coend
$\int^{t\in\mathscr{T}} C(t,-)\otimes B(-,t):\mathscr{S}\otimes\mathscr{U}^\op \to \Vect$. 
The monoidal product is induced from the monoidal product of $\LinCat$. 
There is a  symmetric monoidal functor
\begin{align}
	\LinCat \to \Bimod \label{eqnlincatbimod}
\end{align} 
that is the identity on objects and sends
a linear functor $F:\mathscr{S}\to\mathscr{T}$ to the bimodule $\mathscr{T}(-,F-):\mathscr{S}\otimes\mathscr{T}^\op \to \Vect $.

We denote by $\PrL$ the symmetric monoidal bicategory of locally presentable categories over $k$, cocontinuous functors (functors preserving colimits; by `colimits' we mean throughout \emph{small} colimits) and linear natural transformations. The monoidal product is the Deligne-Kelly product $\boxtimes$, see e.g.~\cite{Franco2013,bcjf,bzbj} for details.
For $\cat{A}\in\PrL$, an object $P\in \cat{A}$ is called \emph{compact projective} if the hom functor $\cat{A}(P,-)$ is cocontinuous; it is called \emph{compact} if $\cat{A}(P,-)$ preserves filtered colimits.

There is  a symmetric monoidal functor
\begin{align}
	\widehat{-}:	\Bimod &\to \PrL \label{eqnwidehat}
\end{align}
referred to as \emph{cocompletion}
that sends a linear category 
$	\mathscr{S}$ to $ \widehat{\mathscr{S}}:=\Lin(\mathscr{S}^\op,\Vect)$.
The functor $\widehat{-}$ is an embedding whose image are exactly the locally presentable categories with enough compact projective objects. 

Let us also recall  some terminology from \cite[Section~2.2]{cyclic}:
In a symmetric monoidal bicategory $\cat{M}$ with monoidal product $\boxtimes$ and monoidal unit $I$,
a \emph{non-degenerate pairing}  on an object $X\in\cat{M}$ is a 1-morphism $\kappa : X \boxtimes X \to I$ that exhibits $X$ as its own dual in the homotopy category of $\cat{M}$. In that case, $\kappa$ comes with a copairing $\Delta : I \to X \boxtimes X$ that together with $\kappa$ fulfills the usual zigzag identities up to isomorphism (note that $\Delta$ is unique if it exists). 
A \emph{symmetric structure} on such a pairing is the structure of a homotopy fixed point on $\kappa : X \boxtimes X \to I$ with respect to the homotopy involution that exchanges the two $X$-factors via the symmetric braiding in $\cat{M}$.
Of course, if $\cat{M}=\PrL$, the coevaluation $\Delta$ amounts to an object in $X\boxtimes X$, referred to as \emph{coevaluation object}.

\subsection{Cyclic module categories over vertex operator algebras\label{seccyclicmodulecat}}

In order to apply this terminology to vertex operator algebras, we will use the set-up of
\cite[Section~2]{alsw} that is partly extracted from \cite{hlzi}:
For abelian groups $A\subset B$, let $V$ be (here and in the rest of the paper --- unless otherwise stated) an $A$-graded vertex operator algebra.
By a $V$-module we understand a $B$-graded
generalized $V$-module (at least if $B$ is fixed through the context).

The following technical definition lies at the heart of this paper and will be justified by 
Theorem~\ref{thmcfe2} below:
\begin{definition}\label{defvcyclic}
	A full subcategory $\catT$ of the category of all graded $V$-modules is called \emph{cyclic} if it is  equipped with the following structure plus relations:
\begin{enumerate}
	\item[(P)] The category $\catT$ comes equipped with a pairing
	\begin{align}
		\spr{-,-} : \catT\otimes\catT \to \Vect \ , \quad (t,u) \mapsto \spr{t,u}
		\end{align}
	in $\Bimod$	that is \emph{non-degenerate}, i.e.\
		there exists a copairing $\Delta:\catT^\op \otimes \catT^\op\to\Vect$ with
		\begin{align}
			\int^{v\in\catT} \Delta (u,v) \otimes \spr{v,t} &\cong \Hom_V(u,t) \ , \label{eqndelta1} \\
			\int^{u\in \catT} \spr{t,u} \otimes \Delta(u,v) &\cong \Hom_V(v,t) \ . 
		\end{align}
		(Note that $\Delta$ is not structure. We require its existence; it is then unique.)
	\item[(M)] The structure of an associative algebra in $\Bimod$, i.e.\ a \emph{monoidal structure through bimodules} \begin{itemize}
		\item $B_\otimes : \catT^\op \otimes \catT\otimes \catT \to \Vect$, \item and $\eta: \catT^\op\to\Vect $,
		\item that come with natural transformations $\alpha, r , \ell$ implementing associativity and unitality and satisfy the usual triangle and pentagon axioms.
		\end{itemize}     
		This structure is also called a \emph{(linear) promonoidal category} \cite{daypromonoidal}, and we refer to \cite{nlab:promonoidal_category} for the diagrams expressing the triangle and pentagon axioms. 
	We write the value of the bimodule $B_\otimes$
	on $t,u,v\in\catT$ as $\spr{ v\mid t, u}$.
	It is covariant in $t,u$ and contravariant in $v$. The induced maps for morphisms $f:t\to r$, $g:u\to s$ and $h:w\to v$ are denoted by
	\begin{align}
		(f,g)_* : \spr{v\mid t,u } &\to \spr{v\mid r,s}  \ , \\
		h^* : \spr{v\mid t,u } &\to \spr{w\mid t,u} \ . 
		\end{align}
	\item[(B1)] A \emph{braiding}, i.e.\ a natural isomorphism of bimodules 
	\begin{align}
		c_{v\mid t,u } : \spr{v\mid t, u} \ra{\cong} \spr{v\mid u,t} \quad \text{for all}\quad t,u,v\in\catT
	\end{align}
	satisfying the two hexagon equations, sometimes referred to as \emph{Yang-Baxter equations}. 
	 
	\item[(B2)] A \emph{balancing}, i.e.\ a natural automorphism
	\begin{align}
		\theta_{t,u} : \catT(t,u)\ra{\cong}\catT(t,u)\quad \text{for all}\quad t,u\in\catT
	\end{align} 
(by the Yoneda Lemma this is automatically
the postcomposition with an automorphism $\theta_u:u\to u$ or, equivalently, precomposition with an automorphism $\theta_t:t\to t$)
	such that \begin{align} \theta_v^* =  	(c_{v\mid u,t} \circ c_{v\mid t,u})\circ ({\theta_t} , {\theta_u}) _*\ . 
		\end{align}
	\item[(B3)] The equality $\spr{\theta,\id}=\spr{\id,\theta}$  of bimodules holds.
	\item[(C)] An isomorphism $\gamma: \kappa \circ (\eta\otimes B_{\otimes}) \ra{\cong} \kappa$ of bimodules
	or, equivalently, in components 
	\begin{align}
		\gamma_{t,u} :  \int^{r,s\in\catT} \spr{r,s}\otimes  \eta(r)\otimes \spr{s\mid t,u } \ra{\cong}\spr{t,u} \ , \label{eqngamma}
	\end{align}
	such that $\gamma$ agrees with the left unitor if precomposed with $\eta$.
\end{enumerate}     
	\end{definition}

\begin{remark}
A priori, we do not ask the pairing to be symmetric, but the left hand side of~\eqref{eqngamma}, and hence also $\spr{-,-}$,
	is symmetric in $t$ and $u$ via the isomorphisms induced by
	\begin{align}
	\spr{s\mid t,u} \ra{(\theta_t ,\id_u)_*} \spr{s\mid u,t} \ra{c_{s \mid t,u}^{-1}}\spr{s\mid t,u} \ . 
	\end{align}
	We equip $\spr{-,-}$ always with this symmetric structure.
	\end{remark}

We discuss example classes in Section~\ref{secexamples}.

\subsection{Passing to the cocompletion}
In \cite{cyclic} we define, building on the notions of \cite{gk,costello}, the notion of a \emph{cyclic framed $E_2$-algebra} in a symmetric monoidal bicategory $\cat{M}$
by extending the theory of operadic algebras inside symmetric monoidal bicategories up to coherent isomorphism to cyclic operadic algebras. 
In this paper, we make an effort to present 
this notion  without
discussing in too much depth the theory of operads. It suffices to say that a cyclic framed $E_2$-algebra in $\cat{M}$
is an object in $\cat{M}$ equipped with operations of total arity $n+1$ for $n\ge -1$ parametrized by embeddings of $n+1$ disks in a two-dimensional sphere.
On these operations, the automorphisms of these embeddings act in a coherent way.
It is proved in \cite[Section~5]{cyclic} that 
this amounts to 
a balanced braided commutative algebra object in $\cat{M}$ with a compatible non-degenerate symmetric pairing, as discussed in Section~\ref{seclincat}, as this will allow us to make the connection to Definition~\ref{defvcyclic}.
For a condensed introduction to cyclic algebras and the needed operadic framework, we refer also to \cite[Section~2]{hdr}.
An overview with the extensive use of operads can be found in \cite{Muller2026GVnotes}.

We will discuss now how,
from a cyclic category of $V$-modules,
we can produce a presentable cyclic framed $E_2$-category of $V$-modules.

\begin{definition}\label{defRV}
	For any cyclic category $\catT$ of $V$-modules,
	we define $\RV$ as the cocompletion of $\catT$, i.e.\
	as the category
	\begin{align}
		\RV := \Lin( \catT^\op,\Vect  ) 
	\end{align}
	of linear functors $\catT^\op\to \Vect$. 
\end{definition}

Note that $\RV$ is defined \emph{in reference} to the type of module that we allow in $\catT$ (we suppress this dependence for legibility).
All choices conforming to Definition~\ref{defvcyclic} are possible.
Needless to say, the choice of modules matters greatly;  this is standard in the theory of vertex operator algebras, see e.g.~the many different module categories over the same vertex operator algebra in the table in \cite[Section~1.3]{crsy}. 

\begin{remark}\label{remobjT}
		Throughout, we see $\mathscr{T}$ as full subcategory of $\RV$ via the Yoneda embedding $t\mapsto \catT(-,t)$.
		Since
		\begin{align} \Hom_{\RV} (\catT(-,t),\catT(-,u))\cong\Hom_V(t,u) \ ,
			\end{align} for $t,u\in\catT$ by the Yoneda Lemma, we allow ourselves to denote the morphism spaces of $\RV$ by $\Hom_V(-,-)$ as well.
		\end{remark}

		\begin{definition}\label{defdual}
		We define the \emph{weak dual} of $t\in\catT$ through the coend
		\begin{align}
			t^* := \int^{u\in\catT} \Delta(u,t) \otimes u \in \RV \ , 
			\end{align}
		where $\Delta :\catT^\op \otimes\catT^\op \to \Vect$ is the coevaluation for $\catT$,
		 and the symbol $\otimes$ in the expression $\Delta(u,t) \otimes u$ denotes the $\Vect$-module structure of $\RV$, i.e.\
		the $\Vect$-tensoring of objects in $\RV$.
			\end{definition}
		
		Note that cocompleteness of $\RV$ guarantees the existence of the coend in $\RV$.

\begin{theorem}\label{thmcfe2}
	For any cyclic category $\catT$ of $V$-modules, the cocompletion
	\begin{align}
		\RV=\Lin(\catT^\op,\Vect)
		\end{align}
	is a cyclic framed $E_2$-algebra in $\PrL$ with enough compact projective objects; in fact,
	the compact projective objects of $\RV$ can, up to equivalence in $\Bimod$, be identified with $\catT$.
	The non-degenerate symmetric pairing $\spr{-,-} : \RV\boxtimes \RV \to\Vect$ of $\RV$
	 is given by the cocontinuous extension of the pairing on $\catT$.
	 The coevaluation object  is given by
	 \begin{align}
	 \label{eqnDelta}	\Delta = \int^{t\in\catT} t^*\boxtimes t \ . 
	 	\end{align}
	\end{theorem}

A large part of the proof of this result relies on the following observation:
The structure in Definition~\ref{defvcyclic} 
amounts to a  $\Bimod$-valued cyclic $E_2$-algebra by \cite[Theorem~5.6]{cyclic} (this results applies to all symmetric monoidal bicategories and just needs to be spelled out in $\Bimod$). Let us record this:

\begin{proposition}\label{propcyce2}
	A cyclic structure on a full subcategory $\catT$ of the category of all $B$-graded $V$-modules is exactly the structure of a $\Bimod$-valued cyclic algebra over the cyclic framed $E_2$-operad.
\end{proposition}

For people familiar with operads, Proposition~\ref{propcyce2} might as well be treated as a definition, but we really want to highlight that this structure can be spelled out very explicitly, which is why we prefer to give the definition in terms of the list in Definition~\ref{defvcyclic}.

\begin{proof}[\slshape Proof of Theorem~\ref{thmcfe2}]
	$\RV$ is the cyclic framed $E_2$-algebra obtained by applying the symmetric monoidal embedding $\Bimod\to\PrL$ from~\eqref{eqnwidehat}
	to the cyclic framed $E_2$-algebra in $\Bimod$ from Proposition~\ref{propcyce2}. 	By design $\RV$ always has enough compact projective objects,
	and the compact projective objects of $\RV$ agree with $\catT$ up to equivalence in $\Bimod$.
	The pairing on $\RV$ is determined on compact projective objects through the one on $\catT$.

	In order to prove that
	 $\Delta=\int^{t\in \catT} t^*\boxtimes t $ (the coend exists because $\RV$ is cocomplete) is the coevaluation object, we need to prove
	\begin{align}
		\spr{ -,\Delta'}\otimes \Delta''&\cong \id \ , \label{eqnsnake1}\\
		\Delta' \otimes \spr{  \Delta'',-}&\cong \id  \label{eqnsnake2}
		\end{align}
	with Sweedler notation $\Delta = \Delta' \boxtimes \Delta''$
	(as usual, Sweedler notation does not imply that $\Delta$ is a `pure tensor'). 
	These isomorphisms of cocontinuous functors just need to be established on compact projective objects, i.e.\ on $u\in\catT$.
	Then~\eqref{eqnsnake1}, since $\spr{ u,-}$ is cocontinuous, amounts to
	\begin{align}
	\int^{t\in\catT}	\spr{    u,t^*}\otimes t \cong u \ . 
		\end{align}
	Indeed, by applying Definition~\ref{defdual} we find \begin{align}
	\label{eqnsprhom}	\spr{ u,t^*}\cong \int^{v\in\catT} \Delta(v,t) \otimes \spr{ u,v} \cong \Hom_V(t,u) \ , 
		\end{align} where we use \eqref{eqndelta1} in the last step. Hence,
	$	\int^{t\in\catT}	\spr{ u,t^*} \otimes t \cong 	\int^{t\in\catT} \Hom_V(t,u) \otimes t\cong u$. 
	Analogously, we verify~\eqref{eqnsnake2}.
	\end{proof}

\begin{remark} Ind completions (see Remark~\ref{remarkprc} below)
	 instead of free cocompletions of categories of modules over a vertex operator algebra are considered in the recent article
	\cite{mcraenegron}. The initial assumptions made for this are generally not fulfilled for our cyclic subcategories. As motivated in the introduction, these will generally not be equipped with a monoidal product. 
	\end{remark}

\subsection{The Spider Theorem}
Theorem~\ref{thmcfe2} takes the rather basic algebraic data of a cyclic category of $V$-modules and produces a topological object, namely a cyclic
framed $E_2$-algebra in presentable categories. 	
We will see now why this is useful:
Suppose that we are given a cyclic category $\catT$ of $V$-modules.
This means in particular that we have the non-degenerate pairing
\begin{align}
	\spr{-,-} : \RV\boxtimes \RV\to\Vect
	\end{align}
with its symmetry isomorphism $\spr{X,Y}\cong\spr{Y,X}$. 
We will now explain how we can extend this to maps
\begin{align}
\spr{-,\dots,-} : \RV^{\boxtimes n} \to\Vect \quad \text{for}\quad n\ge 0\label{eqnsprmaps}
\end{align}
with cyclic symmetry (i.e.\ invariant under cyclic permutation up to coherent isomorphism)
 such that for $n=2$ we recover $\spr{-,-}$.
For $n=0$, we use $\spr{I,I}$ with the monoidal unit $I\in\RV$,
 while for $n=1$, we use $\spr{I,-}$ or the isomorphic $\spr{-,I}$. 
For $n=3$, and $t,u,r\in\catT$, we define
\begin{align}
	\spr{r,u,t} := \int^{s\in\catT} \spr{s,r}\otimes \spr{s \mid t,u}  
	\end{align}
and extend cocontinuously to $\RV^{\boxtimes 3}$.
For $n=4$ and $t,u,r,s \in \catT$, we set
\begin{align}
	\spr{t,u,r,s} := \int^{v,w\in\catT} \spr{t,u,v}\otimes\Delta(v,w) \otimes \spr{w,r,s} 
	\end{align}
and extend cocontinuously to $\RV^{\boxtimes 4}$.
Actually, we might have the idea to set instead
\begin{align}
	\spr{t,u,r,s} := \int^{v,w\in\catT} \spr{u,r,v}\otimes\Delta(v,w) \otimes \spr{w,s,t} \ , \label{eqncoendexpression}
\end{align}
and we will see in moment that it does not matter up to a contractible choice as long as the cyclic order of the elements on the left hand side and the coend expression on the right hand side is the same.
The following result, that may be seen as a coherent version or a categorification
 of the Spider Theorem for Frobenius algebras \cite{coeckeduncan}, tells us that we can safely make a choice:	

\begin{theorem}[Coherent Spider Theorem]\label{thmcohspider}
	Define $\spr{-,\dots,-}_c$ through a choice $c$ of coend expression as in \eqref{eqncoendexpression} involving $\kappa,\Delta$ and $\spr{- \mid -,-}$ respecting the cyclic structure.
	Then for another choice $c'$, there is a canonical isomorphism \begin{align}
		\alpha_{c,c'}:\spr{-,\dots,-}_c\ra{\cong } \spr{-,\dots,-}_{c'}\end{align} with $\alpha_{c,c}=\id$ and $\alpha_{c',c''}\circ\alpha_{c,c'}=\alpha_{c,c''}$. 
	\end{theorem}

\begin{proof}
	By Theorem~\ref{thmcfe2} $\RV$ is a cyclic framed $E_2$-algebra, so it has an underlying cyclic associative algebra. Such a cyclic associative algebra can be described through \begin{enumerate}
		\item a unital and associative product up to coherent isomorphism plus a compatible non-degenerate pairing
	(the description involving $\kappa,\Delta$ and $\spr{- \mid -,-}$), \item or as  a cyclic algebra over the cyclic operad whose operations of total arity $n$ are the cyclic orders of a set with $n$ elements. \end{enumerate}
	The coherent Spider Theorem is a consequence of the fact that these two descriptions are actually equivalent \cite[Theorem 4.2]{cyclic}.
	From the second description, we obtain
	the maps \begin{align} \spr{-,\dots,-}:\RV^{\boxtimes n}\to\Vect\label{eqnsprproof}\end{align} with the desired cyclic symmetry and, through the equivalence of both descriptions,
	  for any choice $c$ as above a canonical isomorphism $\gamma_c : \spr{-,\dots,-}_c \ra{\cong} \spr{-,\dots,-}$, so that we can set $\alpha_{c,c'}=\gamma_{c'}\circ\gamma_c$ to obtain the statement of the coherent Spider Theorem.
	\end{proof}

If we want to define
 the maps \eqref{eqnsprmaps} without making any choice, we can define them as a colimit over the contractible groupoid spanned by all possible choices $c$, or we just use \eqref{eqnsprproof} from the proof.
Both options are canonically isomorphic.

\begin{corollary}\label{corspr}
	The maps 	\begin{align}
		\spr{-,\dots,-}:	\RV^{\boxtimes n}\to\Vect \ , \quad
		X_1 \boxtimes \dots \boxtimes X_n \mapsto \spr{X_1,\dots,X_n}
	\end{align}  from 
	Theorem~\ref{thmcfe2}
\begin{itemize}
	\item	come with an action of the ribbon braid group of $n-1$ strands,
\item	and there is a canonical isomorphism
	\begin{align}
	\label{eqnY}	\int^{t\in\catT} \spr{X_1,\dots,X_n,t^*} \otimes \spr{Y_1,\dots,Y_m,t} \ra{\cong} \spr{X_1,\dots,X_n,Y_1,\dots,Y_m} \ , 
		\end{align}
	where the $t$ and $t^*$ can also be inserted in different positions.
	\end{itemize}
	\end{corollary}

\begin{proof}
	To obtain the action of the ribbon braid group on $\spr{-,\dots,-}$, one needs to apply \cite[Proposition 5.3]{cyclic}
	to the situation of Theorem~\ref{thmcfe2}.
	The isomorphism~\eqref{eqnY} is the gluing compatibility for the cyclic framed $E_2$-algebra $\RV$ in combination with \eqref{eqnDelta}.
	\end{proof}

\needspace{10\baselineskip}
\section{Different cases and sources of examples\label{secexamples}}
Theorem~\ref{propcyce2} produces a cyclic framed $E_2$-category $\RV$
of modules over a vertex operator algebra
from a cyclic category of $V$-modules.
This begs the question where cyclic categories actually appear.
In this section, we describe a number of situations in which we can 
 obtain rich classes of
 cyclic categories of $V$-modules and describe $\RV$.
 The reader just interested in the topological and geometric constructions and already convinced that the categorical structure described in Definition~\ref{defvcyclic} is  ubiquitous in quantum algebra may skip this section.

\subsection{The finite case}
In Theorem~\ref{thmcfe2}, there are no finiteness assumptions in place, but the situation simplifies considerably if we add some.

\begin{definition}\label{defcompact}
	We say that a cyclic category $\catT$ of $V$-modules (and then also $\RV$) is \emph{compact} if
	all the structure maps of the cyclic framed $E_2$-algebra $\RV$, including the coevaluation,	 preserve compact objects.
	In other words, in the compact case, we ask $\RV$ to live actually in $\catf{Pr}_\catf{c}$, the symmetric monoidal bicategory of compactly generated presentable categories, cocontinuous functors preserving compact objects and linear natural transformations.
	\end{definition}

\begin{remark}\label{remarkprc}
Denote by $\Rex$ the symmetric monoidal bicategory of finitely cocomplete linear categories, functors preserving finite colimits (so-called \emph{right exact functors}) and linear natural transformations. By \cite[Section~3.1]{bzbj} there is a pair of inverse equivalences
\begin{equation}\label{equiv}
	\begin{tikzcd}
		\Rex \ar[rrrr, shift left=2,"\catf{ind}=\text{ind completion}"] &&\simeq&& \ar[llll, shift left=2,"\catf{comp}=\text{compact objects}"] \catf{Pr}_\catf{c}\ . 
	\end{tikzcd}
\end{equation}
This means that, in the compact case, the compact objects 
$\catf{comp}(  \RV   )$
of $\RV$ can equivalently be seen as a cyclic framed $E_2$-algebra in $\Rex$.
\end{remark}

Note that $\Rex$ has a full symmetric monoidal subcategory $\Rexf\subset \Rex$ spanned by \emph{finite categories} in the sense of Etingof-Ostrik~\cite{etingofostrik} (see also the textbook \cite{egno}), i.e.\ linear abelian categories with finite-dimensional morphism spaces, enough projective objects, finitely many isomorphism classes of simple objects and finite length for every object. These are exactly those linear categories that can be written as finite-dimensional modules over a finite-dimensional algebra.

\begin{definition}\label{deffinite}
	We say that a cyclic category $\catT$ of $V$-modules (and then also $\RV$) is \emph{finite} if it is compact in the sense of Definition~\ref{defcompact}
	 and if $\catf{comp}(  \RV   )\in\Rexf$. 
	\end{definition}

\begin{lemma}\label{lemmafiniteness}
	A cyclic category $\catT$
	of $V$-modules is finite if and only if $\catT$ is equivalent in $\Bimod$ to a linear category
	with one object $\star$ whose endomorphism algebra $A:=\End(\star)$ is finite-dimensional. 
\end{lemma}
	
\begin{proof}
	Suppose that $\catT$ is finite, then the compact objects of $\RV$ are equivalent to the finite-dimensional modules over a finite-dimensional algebra $A$, and
	 hence $\catT\simeq BA$ in $\Bimod$, where $BA$ has one object $\star$ and $A=\End(\star)$.
	
	Conversely, if $\catT \simeq BA$ in $\Bimod$, then $\RV$ is equivalent to the cocompletion of $BA	$, i.e.\
	the category of $A$-modules.
	The subcategory of compact objects is the category of finite-dimensional $A$-modules, which is a finite category. 
	\end{proof}

\subsection{Grothendieck-Verdier duality}
 The finite situation from Definition~\ref{deffinite}
 can be described using the following notion: A \emph{ribbon Grothendieck-Verdier category},
 as introduced by Boyarchenko-Drinfeld \cite{bd} based on notions appearing in work of Barr~\cite{barr},	 is a balanced braided monoidal category $\cat{A}$ 
\begin{itemize}
	\item together with an object $K\in \cat{A}$ making the hom functors $\cat{A}(-\otimes Y,K)$ representable via $\cat{A}(X\otimes Y,K)\cong \cat{A}(X,DY)$ such that $X \mapsto DX$ becomes an anti-equivalence of $\cat{A}$, 
	\item such that $D\theta_X=\theta_{DX}$ for all $X\in\cat{A}$.
\end{itemize}One calls $K$ the \emph{dualizing object} and $D$ the \emph{Grothendieck-Verdier duality}. One has necessarily $K=DI$. For an overview about the connection between Grothendieck-Verdier categories, low dimensional topology and representation theory,
 we refer to~\cite{Muller2026GVnotes}, and to \cite{cyclic,alsw} for the details. 

The connection between the finite version of Theorem~\ref{thmcfe2} and Grothendieck-Verdier duality is the following:

\begin{proposition}\label{propfinite}
	The cyclic categories $\catT$ of $V$-modules
	that are finite are,
	up to equivalence of $\Bimod$-valued cyclic framed $E_2$-algebras, exactly the ones that arise as the projective objects of a
	 ribbon Grothen\-dieck-Verdier category $\cat{A}$ in $\Rexf$.
	 In that case, $\RV$ is equivalent to the ind completion of that $\Rexf$-valued ribbon Grothendieck-Verdier category $\cat{A}$.
	 Moreover, for $t,u \in \catT$, the pairing $\kappa$ is determined by
	 \begin{align}
	 	\spr{t,u} \cong \catT( D\nakar t,u)
	 	\end{align} in terms of the right Nakayama functor
	 	$\nakar = \int^{X\in\cat{A}} \cat{A}(-,X)^* \otimes X$  of $\cat{A}$ and the Grothendieck-Verdier duality $D$ of $\cat{A}$. 
	 	The coevaluation object is given by
	 	\begin{align}
	 		\Delta \cong \int_{X \in \cat{A}} DX \boxtimes X \cong \int^{t\in \Proj \cat{A}} \nakar Dt \boxtimes t \label{eqndeltafinite}
	 		\end{align}
	\end{proposition}

For more background on Nakayama functors in the context of finite categories, we refer to \cite[Section~3.5]{fss}.

\begin{proof}
	By assumption $\cat{A}=\catf{comp}(\RV)$ is finite and hence a cyclic framed $E_2$-algebra in $\Rexf$ which can equivalently be described a ribbon Grothendieck-Verdier category in $\Rexf$ \cite[Theorem~5.13]{cyclic}
	(originally formulated for left exact functors, see Remark~\ref{remlexrex} below).

	Now the projective objects of $\cat{A}=\catf{comp}(\RV)$ are given by
	\begin{align}
\Proj\cat{A}=	\Proj \catf{comp}(\RV)\simeq 	\catf{CProj}(\RV)\simeq \catT\end{align}
where $\catf{CProj}$ denotes the subcategory of compact projective objects, and
$\simeq$ denotes the equivalence of $\Bimod$-valued cyclic framed $E_2$-algebras.
From \eqref{equiv}, we conclude that $\RV$ is the ind completion of $\cat{A}$.

With \cite[Corollary~8.1]{brochierwoike} (\cite{cyclic} already treats the $\Lexf$-valued case, see also Remark~\ref{remlexrex} for the comparison), we observe
$\spr{ t,u}\cong \cat{A}(t\otimes u,K)^*\cong \cat{A}(t,Du)^*$ with the dualizing object $K$ of $\cat{A}$ and hence
\begin{align}
\spr{ t,u}\cong	\cat{A}(Du,\nakar t)\cong \cat{A}(D\nakar t,u)\label{eqnsprnakar}
	\end{align} through the connection between the right Nakayama functor  and twisted Calabi-Yau structures~\cite{shibatashimizu,tracesw}.
Finally, we realize $\cat{A}(D\nakar t,u)=\catT(D\nakar t,u)$ because $D\nakar t$ actually lives again in $\catT$, i.e.\ it is projective.
This is because the Nakayama functor sends projective objects to injective ones \cite[Section~3]{ivanov}, just like the anti-equivalence $D$.

The first isomorphism in \eqref{eqndeltafinite} is (dual to) \cite[Proposition 2.26]{cyclic}.
For the second one,
use
\begin{align}
	\Delta (t,u) = \kappa(Dt,Du)\stackrel{\eqref{eqnsprnakar}}{\cong} \cat{A}(D\nakar Dt,Du)\cong \cat{A}(\nakal t,Du)\cong \cat{A}(u,\nakar Dt)
	\end{align} for 
$t,u\in \Proj \cat{A}$, where we have used 
 the left Nakayama functor $\nakal = \int_{X\in\cat{A}} \cat{A}(X,-)\otimes X$ and the relation $D\nakar = \nakal D$ that can be verified directly.
This leads to $t^* \cong \nakar Dt$. 
Now the second isomorphism in \eqref{eqndeltafinite} follows from \eqref{eqnDelta}.
Note that the translation between the end and the coend could also be achieved using \cite[eq.~(3.52)]{fss}.
	\end{proof}

\begin{remark}[$\Rexf$ versus $\Lexf$]\label{remlexrex}
	The result \cite[Theorem~5.13]{cyclic} treats cyclic framed $E_2$-algebras and ribbon Grothendieck-Verdier categories in $\Lexf$, the symmetric monoidal bicategory of finite categories, \emph{left exact} functors and linear natural transformations.
	The notion of Grothendieck-Verdier duality in $\Lexf$ has to be defined \emph{dually}. More precisely, one asks for isomorphisms
	\begin{align}
		\cat{A}(K,X\otimes Y)\cong \cat{A}(DX,Y) \ . 
		\end{align}
	But the classification of cyclic framed $E_2$-algebras in \cite[Section~5]{cyclic} applies to all symmetric bicategories and yields, when properly unpacked, 
	 an equivalence of $\Rexf$-valued cyclic framed $E_2$-algebras and $\Rexf$-valued ribbon Grothendieck-Verdier categories, see  \cite{mwansular,envas}.
	Actually, given a ribbon Grothendieck-Verdier category $(\cat{A},\otimes,K,D,c,\theta)$ in $\Rexf$ (or $\Lexf$, respectively), one may pass to the `second monoidal product' \cite[Section~4.1]{bd}
	\begin{align} X\odot Y := D(DY \otimes DX) \label{eqnsmp}
	\end{align} which then makes $(\cat{A},\odot,I,D,c,\theta)$ a ribbon Grothendieck-Verdier category in $\Lexf$ (or $\Rexf$, respectively).
	\end{remark}

\subsection{The finite rigid case\label{secfiniterigid}}
Suppose that for a vertex operator algebra $V$ the category $\Vmod$ of strongly graded
$V$-modules is a finite tensor category in the sense of \cite{etingofostrik} (this means that $\Vmod$ is a finite category with rigid monoidal product and simple unit), with the rigid dual agreeing with the operation of taking the contragredient
(note that this is stronger than just asking $\Vmod$ to be rigid as a property).
We refer to this situation as the \emph{finite rigid case}. An important example is the triplet  $\cat{W}_{p}$ \cite{gannonnegron}. 

In the finite rigid  case, one finds a modular category, i.e.\ a finite ribbon category whose braiding is non-degenerate, see \cite{shimizumodular} for the different equivalent characterizations.
Let us record this:

\begin{corollary}\label{corvmod}
	In the finite rigid case, 
	$\catT=\Proj \Vmod$ is a cyclic category of $V$-modules. 
For this choice,	$\RV$ is the ind completion of the modular category $\Vmod$.	
	\end{corollary}

\begin{proof}
	By assumption $\Vmod$ is a finite ribbon category, and actually a modular category \cite[Theorem 1]{mcrae}.
In particular, it is a finite ribbon Grothendieck-Verdier category, and hence Proposition~\ref{propfinite} tells us
that $\Proj \Vmod$ is a cyclic category of $V$-modules.
With this choice, $\RV$ is the ind completion of $\Vmod$.
\end{proof}

A special case of the finite rigid case is the rational case:
A vertex operator algebra $V$ is called \emph{rational} if its category of strongly graded modules $\Vmod$ is semisimple with finitely many isomorphism classes of simple objects.
This includes many classical examples (see for example~\cite[Chapter 6]{LepowskyLi2004} for a textbook reference).
In case $V$ is, in addition to being rational, \emph{self-contragredient}, i.e.\ $V\cong V'$,
the category $\Vmod$ is a modular fusion category~\cite{huang}, i.e.\ a \emph{semisimple} modular category, and hence a cyclic framed $E_2$-algebra in finite semisimple categories~\cite{cyclic}. In this situation,
we may choose for our cyclic subcategory simply $\catT=\Vmod$.
Again, $\RV$ ends up being the ind completion of $\Vmod$.

\subsection{The $C_2$-cofinite case\label{secC2}}
	Not all cyclic categories of $V$-modules are of the form described in Section~\ref{secfiniterigid}.
	A counterexample comes from the  triplet $\cat{W}_{2,3}$ at $c=0$ that gives rise to ribbon Grothendieck-Verdier in $\Rexf$ by \cite{grw,grw-proc,alsw}; the monoidal product is not exact and hence cannot be rigid. The projective objects of this 
	finite ribbon Grothendieck-Verdier category still form a cyclic category in the sense of Definition~\ref{defvcyclic}
	as follows again from Proposition~\ref{propfinite}.
	
	Important warning: It is tempting to think that simply any vertex operator algebra $V$ plus choice of modules good enough to produce a Grothendieck-Verdier category through the results of \cite{alsw} would also produce a cyclic category of $V$-modules.
	This conclusion is not correct! The ribbon Grothendieck-Verdier categories from \cite{alsw} live a priori in (linear) categories, where we do not even have interesting dualizable objects, see \cite[Remark 4.13]{cyclic}.

	But we can make the following statement: 
	
	\begin{proposition}\label{propc2cofinite}
		If $V$ is a $C_2$-cofinite vertex operator algebra, and let $\Vmod$ be the category
		 of strongly graded $V$-modules, then $\Proj \Vmod$ is cyclic,
		and $\RV$ is  the ind completion of $\Vmod$.
		\end{proposition}

	\begin{proof}
		On $\Vmod$, we obtain  a ribbon Grothendieck-Verdier structure via \cite[Theorem~2.11]{alsw} with right exact tensor product, and
	\cite[Theorems~3.23 \& 3.24]{huangfin} tell us that $\Vmod$ is actually 
	finite
	(this seems to be known among vertex operator algebra experts, and we thank Simon Wood for pointing out an explicit reference).
	 This means that $\Vmod$ is a ribbon Grothendieck-Verdier category in $\Rexf$, so that $\Proj \Vmod$ is cyclic thanks to Proposition~\ref{propfinite}.
	With this choice, $\RV$ is again the ind completion of $\Vmod$.
\end{proof}

\subsection{The non-finite rigid case\label{secghost}}
Suppose that a suitable module category $\Vmod$ is linear abelian with enough projective objects and a rigid monoidal product. An example can be obtained from the $\beta\gamma$ ghosts \cite{ghost} (these are neither rational nor $C_2$-cofinite), see the table in \cite[Section~1.3]{crsy} for more examples.
In that situation, it is proved in \cite[Proposition 3.2 \& Example 3.8]{yeralopen} that the ind completion of $\Vmod$ is a cyclic framed $E_2$-algebra, see also \cite[Section~4]{yeralthesis}. 
Since $\catf{ind}\, \Vmod$ has enough compact projective objects, the compact projective objects of  $\catf{ind}\, \Vmod$, which are the projective objects of $\Vmod$ are a cyclic category, and $\RV$ agrees with  $\catf{ind}\, \Vmod$. 
This class of examples is different from the ones in Section~\ref{secfiniterigid} and \ref{secC2} because there $\RV$ is the ind completion of a cyclic framed $E_2$-algebra in $\Rexf$. In the non-finite rigid case, the underlying framed $E_2$-algebra $\Vmod$ lives still in $\Rex$ (though not in $\Rexf$), but the technical subtlety is that $\Vmod$ is generally \emph{not} cyclic in $\Rex$, but in $\PrL$.
This is another instance where the observation that $\Vmod$ has a ribbon Grothendieck-Verdier structure in linear categories (which clearly it has because it is rigid) is not enough to extract the relevant cyclic framed $E_2$-structure needed in the sequel, see \cite[Example 3.3]{yeralopen} for more details on this important technical subtlety.

\subsection{The algebro-geometric approach\label{secag}}
In the Sections~\ref{secfiniterigid}-\ref{secghost}, we make use of the tensor product theory of	\cite{hlzi}
to produce cyclic categories of $V$-modules, but we do not have to rely on it in all cases:
For a strongly rational vertex operator algebra, one may define spaces of conformal blocks using tools of algebraic geometry \cite{frenkelbenzvi,DGT1,DGT2} and obtain a modular functor $\mathbb{V}$ \cite{damioliniwoike}. 
This endows the category of admissible $V$-modules with the structure of a modular fusion category $\cat{C}_V$ of geometric origin \cite[Theorem 5.3.1]{damioliniwoike}.
While one would expect the monoidal product of $\cat{C}_V$ to agree with the one of \cite{hlzi}, this is actually not known (and it seems difficult to check).
The good news is that this is not at all needed in Definition~\ref{defvcyclic}: $\cat{C}_V$ is cyclic regardless of this point.

\begin{question}\label{q1}
	What examples of cyclic subcategories of module categories over a vertex operator algebra exist that do not fall into the above example classes?
	\end{question}

\needspace{10\baselineskip}
\section{Spaces of conformal blocks via markings and modular extension\label{secmodext}}
In this section, we present the construction of
 spaces of conformal blocks 
based on a cyclic subcategory $\catT$ of modules over a vertex operator algebra $V$.
We  use the tools from \cite{envas} developed mostly for the finite case, and any reader familiar with this paper will be able to digest this section quickly.
One of the key points of this section will be the description of the spaces of conformal blocks
 in terms of markings and coends over compact projective objects.

\subsection{Reminder on the extension of cyclic associative algebras to ribbon graphs\label{secmodex}}		
For $n\ge 0$, denote by $\RGraphs(n)$ the category of connected ribbon graphs~\cite{costello,giansiracusa} whose legs are identified with the legs of a standard graph $T$ with one vertex and $n$ legs 	(a so-called \emph{corolla}); its morphisms are contractions of disjoint unions of trees. 
In more detail,
a ribbon graph $R \in \RGraphs(n)$ consists of \begin{itemize}
	\item a finite graph $\Gamma$ and an identification of the graph obtained by contracting all internal edges of $\Gamma$ with $T$,
	\item and cyclic orders $c_1,\dots,c_m$ for the $m$ vertices of $\Gamma$.
\end{itemize}
The categories $\RGraphs(n)$ form 
a category-valued modular operad \cite{gkmod} whose more detailed definition is given in \cite[Example 7.3]{cyclic} building on \cite{costello}.

Suppose now that $\cat{A}$ is a cyclic associative algebra in $\PrL$, i.e.\ an associative algebra, up to coherent isomorphism, with a non-degenerate symmetric pairing subject to several, rather intricate compatibilities, see \cite[Section~4]{cyclic}. In particular, any cyclic framed $E_2$-algebra in $\PrL$ gives rise to a cyclic associative algebra in $\PrL$ by forgetting the braiding and the balancing. 
Then $\cat{A}$ can be evaluated on ribbon graphs, i.e.\ we get compatible functors \begin{align} \RGraphs(n) \to \Hom_{\PrL}(\cat{A}^{\boxtimes n},\Vect) \ , \quad R \mapsto \Ao^{(R)} \ . \label{eqndefassoc}
\end{align} The formal definition is through \cite[Proposition~7.1 \& Example~7.3]{cyclic}
that treat the extension of operadic algebras to Costello's modular envelope \cite{costello}.
Here is the explicit description:\begin{itemize}
	\item
Take the ribbon graph $R$ and consider the disjoint union $\nu(R)=R_1 \sqcup \dots \sqcup R_\ell$ of corollas with cyclically ordered legs that are obtained by cutting $R$ at all internal edges.
Let us suppose that $R_i$ has $m_i$ legs for $1\le i\le \ell$. 
\item The cyclic associative algebra can be evaluated on each $R_i$ and yields a map $\cat{A}_{R_i} : \cat{A}^{\boxtimes m_i} \to \Vect$.
\item Being a cyclic associative algebra means that we have a coevaluation object $\Delta \in \cat{A}\boxtimes \cat{A}$ at our disposal.
\item Take $\boxtimes _{i=1}^\ell \cat{A}_{R_i} : \boxtimes_{i=1}^\ell \cat{A}^{\boxtimes m_i} \to \Vect$ and insert $\Delta$ in all those pairs of slots belonging to the legs that have to be glued together to recover $R$ from the $R_i$.
The resulting map goes from $\cat{A}^{\boxtimes n}$ to $\Vect$, with $n$ being the number of legs of $R$.\end{itemize}
This is the definition of $\Ao^R$ from \eqref{eqndefassoc}.

\subsection{Marked spaces of conformal blocks}
Denote by $\O(n)$ the groupoid of connected compact oriented surfaces with at least one boundary component and $n$ parametrized boundary intervals. The morphisms in $\O(n)$ are isotopy classes of diffeomorphisms preserving the orientation and the marked intervals.
This means that the automorphism group of $\Sigma \in \O(n)$ is the (pure) mapping class group $\Map(\Sigma)$ of $\Sigma$ (for the chosen parametrization of the boundary intervals). 

By sending a ribbon graph $R\in \RGraphs(T)$ to its associated surface $|R| \in \O(T)$ (the one obtained 
by fattening $R$), we obtain
a functor $|-|:\RGraphs(n) \to \O(n)$.

	For an operation $\Sigma \in \O(n)$
in the open surface operad, we denote by $\M(\Sigma)$ the \emph{category of markings for $\Sigma$}, which is the category of pairs
$(R,f)$ with a ribbon graph $R\in \RGraphs(n)$ and a mapping class $f:|R|\to \Sigma$ identifying the surface $|R|$ obtained by geometrically realizing $R$ and $\Sigma$.
In other words, $\M(\Sigma)$ is the homotopy fiber of $|-|$.
Note that the word `marking' is used in a closely related way in \cite{egaskupers}.

Now let $\mu=(R,f)\in\M(\Sigma)$. Through~\eqref{eqndefassoc}, we may evaluate a cyclic associative algebra $\cat{A}$ 
in $\PrL$ on $\mu$ by forgetting the mapping class, thereby giving us via \eqref{eqndefassoc} \begin{align} \Ao^\mu(\Sigma;-):= \Ao^{(R)} :\cat{A}^{\boxtimes n}\to\Vect \ . \label{eqnAomarking}
\end{align} 
For $\mu = (R,f) \in \M(\Sigma)$ and $g \in \Map(\Sigma)$, denote $g.\mu =(R,gf)$. This defines a $\Map(\Sigma)$-action on $\M(\Sigma)$.

\begin{remark}\label{remmarking}
The assignment
$\M (\Sigma)\ni \mu \mapsto \Ao^{\mu}(\Sigma;-)$ factors through the forgetful functor
$\M(\Sigma)\to \RGraphs(T)$. This means that for $\mu=(R,f)\in\M(\Sigma)$,
	we have 
	natural isomorphisms
	\begin{align}
		\chi_g : \Ao^\mu(\Sigma;-) \ra{\cong} \Ao^{g.\mu}(\Sigma;-) \ . \label{eqnchig}
	\end{align}
Note that~\eqref{eqnchig} is a tautological isomorphism. Actually, $\Ao^\mu(\Sigma;-)$ really just depends
	on the ribbon graph data, so we may model $\Ao^\mu(\Sigma;-)$ and $\Ao^{g.\mu}(\Sigma;-)$ by the same functor in which case $\chi_g$ is the identity. But this is just an artifact of the model that we choose. 
	The functors 
	$\Ao^\mu(\Sigma;-)$ and $\Ao^{g.\mu}(\Sigma;-)$
	will later play different roles in the construction; therefore, it is better to think of them as canonically isomorphic because equality on the nose is not something that we should ask categorically speaking.
\end{remark}

	The following is one of the key technical definitions of this article:

\begin{definition}[Marked spaces of conformal blocks]
	Let $\catT$ be a cyclic category of $V$-modules.
	We define
	$\block^\mu(\Sigma;-):\RV^{\boxtimes n}\to\Vect$ for $\mu=(\Omega,f) \in \M(\Sigma)$ with $\Sigma \in \O(\Sigma)$
	by specializing \eqref{eqnAomarking} applied to the cyclic associative algebra obtained from the $\PrL$-valued cyclic framed $E_2$-algebra $\RV$ 
	from Theorem~\ref{defRV}
	 \emph{by forgetting the braiding and the balancing}.
	We call $\block^\mu(\Sigma;-)$ the \emph{space of conformal blocks for $V$, $\Sigma$ and the marking $\mu$}.
\end{definition}

\subsection{Modular extensions via markings}
Just like for the ribbon graphs, the groupoids $\O(n)$ form 
 a category-valued modular operad.
 A modular algebra over $\O$ is an \emph{open modular functor} or \emph{categorified two-dimensional topological field theory}, see \cite[Section~3]{sn} and \cite{envas} for the definition within the framework that interests us here and \cite{Lazaroiu,mooresegal}  for more background on open field theories.

The functors \begin{align}
\RGraphs(n) \to \Hom_{\PrL}(\cat{A}^{\boxtimes n},\Vect) \ , \quad R \mapsto \Ao^{(R)}\end{align}
built in \eqref{eqndefassoc} from a $\PrL$-valued cyclic associative algebra  send
every morphism to an isomorphism~\cite[Remark 7.2]{cyclic}, which means that they  descend
to the $\infty$-groupoid $|B \RGraphs(n)|$ obtained by taking nerve and geometric realization of the category $\RGraphs(n)$.
Actually, the functor even descends to the fundamental groupoid $\Pi |B \RGraphs(n)|$ of this space because $\PrL$ is `just' a symmetric monoidal bicategory and not an arbitrary symmetric monoidal $(\infty,1)$-category.
This is used in \cite{envas} for the following construction:
As a strengthening of \cite[Theorem~B]{giansiracusa}, we show \begin{align}\label{eqnequivO}\Pi |B \RGraphs(n)|\simeq \O(n) \ , \end{align} meaning that  we obtain  from $\cat{A}$ actually functors \begin{align} \O(n)\to \Hom_{\PrL}(\cat{A}^{\boxtimes n},\Vect) \end{align} producing a modular $\O$-algebra $\Ao$
that we
call 
the \emph{modular extension} of $\cat{A}$.

\begin{theorem}[$\text{\cite[Theorem~2.2 \& eq.~(2.2)]{envas}}$]\label{thmclassification}
	The assignment $\cat{A}\mapsto \Ao$ provides an equivalence \begin{align}
		\text{cyclic associative algebras in $\PrL$}\quad  \ra{\simeq}\quad  \text{open modular functors in $\PrL$} \ . \label{eqnclassopen}
	\end{align} of 2-groupoids.
\end{theorem}

\begin{remark}
This classification is true for any symmetric monoidal bicategorical target.
In \cite{bsz} an $(\infty,1)$-categorical generalization
 is announced to appear in upcoming work of Barkan-Steinebrunner.
 This would be a genuine generalization because \cite{envas} only covers the bicategorical case.
\end{remark}
 
Next we will give a description
of $\Ao$ in terms of the marking dependent quantity in \eqref{eqnAomarking}.
To this end, we need the following observation 
  based on results in \cite{giansiracusa}, with some necessary additions coming from \cite{envas}:
\begin{lemma} \label{lemmaM}
	If $\Sigma \in \O(n)$ is an annulus without marked intervals,
	the space $|B\M(\Sigma)|$ is a circle. Otherwise, $\M(\Sigma)$ is contractible, i.e.\
	$|B\M(\Sigma)|$ is homotopy equivalent to a point.
\end{lemma}

\begin{proof}
	If $\Sigma \in \O(n)$ is an annulus without marked intervals, we can obtain the statement from the proof of \cite[Theorem 2.2]{envas}.
	
	If $\Sigma$ is not the annulus without marked intervals, one argues as follows: We conclude from \cite[Section~5, in particular Lemma~5.1.5]{giansiracusa} that $\RGraphs(n) \to \O(n)$
	can be written as
	\begin{align}
		\RGraphs(n) \ra{\simeq} \int F_n \to \O(n) \ , \label{eqndecomp}
	\end{align} in terms of the following objects:
	\begin{itemize}\item 	 $F_n : \O(n)\to\Cat$ is a functor
		whose value on any $\Sigma \in \O(n)$
		that is not the annulus without marked intervals is a contractible category (this means that $|BF_n(\Sigma)|$ is homotopy equivalent to a point).
		\item	  $\int F_n$ is its Grothendieck construction, see e.g.~\cite[Section I.5]{maclanemoerdijk}, and the second functor is the projection to $\O(n)$.
	\end{itemize}
	The first functor in~\eqref{eqndecomp}
	 is a categorical equivalence.
	For this reason, the homotopy fiber of $\RGraphs(n) \to \O(n)$
	over $\Sigma \in \O(n)$  if $\Sigma$ is not the annulus without marked intervals is equivalent to $F_n(\Sigma)$, which is a contractible category. 
	This means that $\M(\Sigma)$ is contractible.
\end{proof}

\begin{proposition}\label{propcolim} 
	Let $\cat{A}$ be a cyclic associate algebra in $\PrL$. 
	For any  $\Sigma \in \O(n)$ that is not the annulus without marked intervals,
	there is a canonical $\Map(\Sigma)$-equivariant isomorphism
	\begin{align}
		\colimsub{ \mu \in \M(\Sigma)} \Ao^\mu(\Sigma; -)   \ra{\cong}	\Ao(\Sigma; -) \ , \label{eqncolim}
	\end{align}
	where on the left hand side, the mapping class group action is induced by action on the markings and~\eqref{eqnchig}.
	Moreover, the structure map
	\begin{align}
		\psi_\mu:	\Ao^\mu(\Sigma; -)   \ra{\cong} 	\colimsub{ \mu \in \M(\Sigma)} \Ao^\mu(\Sigma; -) \label{eqncolim2}
	\end{align}
	of the colimit is an isomorphism.
\end{proposition}

\begin{proof}
	We may describe $\O(n) \ni \Sigma \mapsto \Ao(\Sigma;-)$ as the functor \eqref{eqndefassoc}, but transported along the equivalence of modular operads
	$E:\Pi |B\RGraphs(n)|\simeq\O(n)$ from \eqref{eqnequivO}; in total arity zero, we consider a version of $\O(n)$ and $\RGraphs(n)$ without the annulus. 
	Therefore, we may see $\Sigma \mapsto \Ao(\Sigma;-)$ also
	as the left Kan extension of \eqref{eqndefassoc} along $E$.
	This tells us 
	\begin{align}
		\Ao(\Sigma;-) \cong \colim \left( E/\Sigma \ra{\text{forget}} \Pi |B\RGraphs(n)| \ra{\eqref{eqndefassoc}}       \Hom_{\PrL}(\cat{A}^{\boxtimes n},\Vect)\right) \ ,
		\label{eqnkanextension} \end{align} see e.g.\ \cite[Theorem 6.2.1]{riehlcat}.

	By the universal property of homotopy fibers, there is a canonical functor $\M(\Sigma) \to E / \Sigma$, and the colimit on the left hand side of \eqref{eqncolim} is the restriction of \eqref{eqnkanextension} along
	  $\M(\Sigma) \to E / \Sigma$.
	 This restriction  induces an isomorphism for the colimits because $\M(\Sigma) \to E / \Sigma$ 
	is final thanks to Lemma~\ref{lemmaM}, see e.g.\ \cite[Section~8.3]{riehl} for more background.
	This gives us \eqref{eqncolim}. In fact, this isomorphism is equivariant because the same is true for \eqref{eqnkanextension}, where on the left hand side the action originates from the action on the slices by postcomposition.
	After restriction along $\M(\Sigma) \to E / \Sigma$,
	 this translates to the action
	induced by the action on markings.
	
The isomorphism~\eqref{eqncolim2} is now an easy consequence: Since $\M(\Sigma)\ni \mu \mapsto \Ao^\mu(\Sigma;-)$
	sends all morphisms to isomorphisms, the colimit descends to the localization of $\M(\Sigma)$ at all morphisms, which is homotopy equivalent to a point by Lemma~\ref{lemmaM}. This makes the inclusion $\{\mu\} \to \M(\Sigma)$ final and gives us~\eqref{eqncolim2}.
\end{proof}

\begin{proposition}\label{propmarking}
If $\Sigma$ has at least one marked interval per connected component, then for two markings $\mu$ and $\nu$ of $\Sigma$, there is a canonical isomorphism
\begin{align}\alpha_{\mu,\nu} : \block^\mu(\Sigma;-) \ra{\cong} \block^{\nu}(\Sigma;-) \end{align}
such that $\alpha_{\mu,\mu}=\id$ and $\alpha_{\nu,\zeta}\circ\alpha_{\mu,\nu} = \alpha_{\mu,\zeta}$ for three markings $\mu,\nu$ and $\zeta$, and we obtain a marking independent space of conformal block functor
$\block(\Sigma;-)$ as the functor 
equipped with an isomorphism \begin{align} \psi_\mu :  \block^\mu(\Sigma;-) \ra{\cong} \block(\Sigma;-)\end{align} for every marking $\mu$
satisfying the universal property
that for two markings $\mu$ and $\nu$ the triangle
\begin{equation}
	\begin{tikzcd}
		\block^{\mu}(\Sigma;-)  \ar[]{rrd}{\psi_\mu}  \ar[swap]{dd}{\alpha_{ \mu,	\nu}}  \\ && 	\block(\Sigma;-)
		\\	\block^{\nu}(\Sigma;-) \ar[swap]{rru}{\psi_\nu}
	\end{tikzcd} 
\end{equation} commutes.
\end{proposition}

\begin{proof}
	Apply Proposition~\ref{propcolim} to $\cat{A}=\RV$
	and set 
	\begin{align}
		\block (\Sigma;-) := \colimsub{ \mu \in \M(\Sigma)} \block^\mu(\Sigma;-) \ . 
		\end{align}
	 This colimit comes with the isomorphism 
		$\psi_\mu:	\block^\mu(\Sigma; -)   \ra{\cong} 	\block(\Sigma; -)$
		from \eqref{eqncolim2}. Now set $
		\alpha_{\mu,\nu}:= \psi_{\nu}^{-1}\circ \psi_\mu$.
		This automatically characterizes $\block$ through the above commuting triangles.
\end{proof}

	\begin{theorem}\label{thmmcgaction}
	The $\PrL$-valued
	cyclic associative algebra underlying the cyclic framed $E_2$-algebra $\RV$ extends uniquely to an open modular functor $\Vo$, and $\Vo$ is the unique open modular functor satisfying
	\begin{align}
		\Vo(\mathbb{D}^2_n; X_1,\dots,X_n)\cong \spr{X_1,\dots,X_n}\quad \text{for}\quad \quad X_1,\dots,X_n\in\RV \label{eqnrvdisk}
	\end{align} 
	on a disk $\mathbb{D}^2_n$ with $n$ marked intervals, 
	where $\spr{...}$ was introduced in Corollary~\ref{corspr}.
	In particular, for each surface $\Sigma \in \O(n)$, we obtain a functor
	$\Vo(\Sigma;-):\RV^{\boxtimes n}\to\Vect$ in $\PrL$ on which the mapping class group of $\Sigma$ acts through linear natural automorphisms.
	There is a canonical mapping class group equivariant isomorphism
	\begin{align}
		\block(\Sigma;-)\cong \Vo(\Sigma;-)\ , 
		\end{align} whenever $\Sigma$ has at least one marked interval per connected component.
	This affords a canonical extension of $\block$ to an open modular functor.
\end{theorem}

\begin{proof}
	By Theorem~\ref{thmcfe2} $\RV$ is a cyclic framed $E_2$-algebra, so we may consider the underlying cyclic associative algebra
	by forgetting the braiding and the balancing.
	Now we apply the classification 
	 Theorem~\ref{thmclassification}.
	 The rest is a consequence of Proposition~\ref{propcolim} and \ref{propmarking}.
\end{proof}

\begin{corollary}\label{corexcinterval}
	If $\Sigma'$ arises from $\Sigma$ by gluing two marked boundary intervals together, there is a canonical isomorphism
	\begin{align}
		\block(\Sigma';-) \cong \int^{t\in \catT} \block(\Sigma; \dots, t^*,\dots,t,\dots)  \ , \label{eqnexcint}
	\end{align}
	where the coend variable is placed into the $\RV$-slots corresponding to the intervals that are glued together.
	The isomorphism \eqref{eqnexcint} is $\Map(\Sigma)$-equivariant, where the left hand side carries the $\Map(\Sigma)$-action obtained by restricting the $\Map(\Sigma')$-action along the group morphism $\Map(\Sigma)\to \Map(\Sigma')$. 
\end{corollary}

\begin{proof} 
	By Theorem~\ref{thmmcgaction} $\Vo$ is an open modular functor and satisfies, as any modular algebra, 
	excision, see e.g.~\cite[Remark 2.6]{hdr} for a short reminder and \cite[Section~2.4]{cyclic} for the details. 
	This implies 
	\begin{align}
		\Vo(\Sigma';-)\cong \Vo(\Sigma; \dots, \Delta',\dots,\Delta'',\dots)
	\end{align}
	for the coevaluation object $\Delta = \Delta' \boxtimes \Delta''$ of $\RV$.
	Now we use the formula for $\Delta$ provided in Theorem~\ref{thmcfe2} and the fact that $\Vo$ is cocontinuous to obtain
		\begin{align}
		\Vo(\Sigma';-) \cong \int^{t\in \catT} \Vo(\Sigma; \dots, t^*,\dots,t,\dots)  \ .
	\end{align}
Finally, we use Proposition~\ref{thmmcgaction} to transfer
the statement to $\block$.
\end{proof}

For $\mu \in \M(\Sigma)$ and $f \in \Map(\Sigma)$, we define $\psi_\mu^f:\block^{f.\mu}(\Sigma;-)\to \block(\Sigma;-)$ as the composition
\begin{align}
	\block^{f.\mu}(\Sigma;-) \ra{\chi_f^{-1}} \block^\mu(\Sigma;-) \ra{\psi_\mu} \block(\Sigma;-)\ . 
	\end{align}
It is now convenient to record the following consequence of Proposition~\ref{propcolim}:
\begin{corollary}
	If $\Sigma$ has at least one parametrized interval per connected component, 
	then the representation  of the mapping class group $\Map(\Sigma)$ of $\Sigma$ on $\block(\Sigma;-)$ is the unique one such 
	 that for every mapping class $f:\Sigma \to \Sigma$ 
	the associated automorphism $f_* $ of $ \block(\Sigma;-)$  makes for 
	every marking $\mu\in\M(\Sigma)$ the square 
	\begin{equation}
		\begin{array}{c}	\begin{tikzcd}
				\block^{\mu}(\Sigma;-)  \ar[]{rr}{\psi_\mu}  \ar[swap]{dd}{\alpha_{ \mu,	f.\mu}} && 	\block(\Sigma;-) \ar{dd}{f_*}  \\ 
				\\	\block^{f.\mu}(\Sigma;-) \ar[swap]{rr}{ \psi_\mu^f}  && 	\block(\Sigma;-)  
			\end{tikzcd} 
		\end{array}\label{eqnsquare}\tag{$\star$}
	\end{equation}
	commute.
\end{corollary}

\begin{proof}
	The action on $\block(\Sigma;-)$ defined by the commutativity of \eqref{eqnsquare} is by definition the one it inherits from the colimit description in Proposition~\ref{propcolim}. In combination with Theorem~\ref{thmmcgaction}, this yields the claim.
	\end{proof}

\section{Modular functors with singularities from vertex operator algebras}
For a cyclic subcategory $\catT$ of $V$-modules, we have accomplished through Theorem~\ref{thmmcgaction} a construction of spaces of conformal blocks on the open sector that comes with an explicit calculation prescription.
In this section, we extend further.

\subsection{Open-closed modular functors with singularities\label{secmfsing}}
In order to deal with colored modular operads below, recall
 from \cite{costello}
that a graph $\Gamma$ consists of two sets $H$ (its half edges)
and $V$ (its vertices) equipped with an involution $\iota : H \to H$ assigning to a half edge the half edge that it is glued to  and a map $H\to V$ sending a half edge to the vertex that it is attached to. The fixed points of $\iota$ are the external legs of $\Gamma$. From a graph $\Gamma$, we can build
 two new graphs, namely $\pi_0(\Gamma)$ by contracting all internal edges,
  and $\nu (\Gamma)$ by cutting them open (in other words, by 
  replacing $\iota$ by the identity map). 

Let $C$ be a set that we will refer to as \emph{set of colors}. A \emph{$C$-coloring} of a graph $\Gamma$ is a function $F: H \to C $ such that $F(\iota (h))=F(h)$ for all half edges $h\in H$. A morphism of colored graphs is a morphism of graphs which is compatible with the colorings. If $(\Gamma,F)$ is a colored graph then both $\pi_0(\Gamma)$ and $\nu(\Gamma)$ inherit a natural $C$-coloring. Costello's construction of the symmetric monoidal category $\Graphs$ explained in detail in~\cite{costello} straightforwardly generalizes to colored graphs.
We only sketch the essential points: 
\begin{itemize}
	\item The objects of $\CGraphs$ are disjoint unions of colored corollas, i.e.\ graphs with only one vertex with some number of legs plus a coloring (zero legs are allowed). 
	\item A morphism $T \to T'$ is an equivalence class of colored graphs $\Gamma$ together with identifications $T \cong \nu (\Gamma)$ and $T'\cong \pi_0(\Gamma)$. The equivalence relation we impose is given by isomorphisms of graphs which are compatible with the identifications.
\end{itemize}
Composition is defined by inserting colored graphs into each other exactly as in~\cite{costello}. Disjoint union of graphs equips $\CGraphs$ with a symmetric monoidal structure whose monoidal unit is the empty graph. 

If $\mathcal{S}$ is  a higher symmetric monoidal category (for example a symmetric monoidal bicategory such as $\Cat$) and $C$ a set of colors, a \emph{$C$-colored modular operad} is a symmetric monoidal functor $\mathcal{O}: \CGraphs \to \mathcal{S}$,
where `symmetric monoidal functor' always has to be understood in the weak sense as explained 
\cite[Section~2.1]{cyclic}.

	Up to isomorphism, a colored corolla is determined
	 by the number of legs for a given color. For example, if $C$ has two elements,
	  the objects of operations are $\mathcal{O}(n,m)\in \mathcal{S}$ 
	  in arity $(n,m)$ for $n\geq 0$ and $m\geq 0$ with $n$ and $m$ many inputs of the two colors.

The following is a generalization of \cite[Section~3]{sn}:
\begin{definition}[Singular surface operad]
	For total arities $n,m\ge 0$, we define the groupoids $\sing(n,m)$ as follows: 
	Objects are
	connected compact oriented surfaces $\Sigma$ with boundary parametrization $[0,1] ^{\sqcup n} \sqcup (\mathbb{S}^1)^{\sqcup m} \to \partial \Sigma$
	with $p\ge 0$ unparametrized boundary circles, that are also referred to as \emph{free boundary components} or \emph{punctures}, such that $n+m+p\ge 1$. 
	We make these groupoids into a two-colored modular operad $\sing$, the \emph{singular surface operad},
	as follows: The two colors are the interval and the circle.
	The operadic composition over parametrized intervals is by gluing of those intervals.
	For the parametrized boundary circles, one needs to add a puncture after the gluing; more precisely, one adds
	parametrized intervals to the associated boundary 
	 circles (depicted below in blue together with a collar) around the image of a fixed base point on $\mathbb{S}^1$ under the parametrization (red dot), then glues along the intervals and leaves an unparametrized boundary circle, i.e.\ a removed disk; schematically, the gluing process is the following:
	\begin{align}
		\label{eqnpucturedgluing} \begin{array}{c}\begin{tikzpicture}[scale=0.5]
				\begin{pgfonlayer}{nodelayer}
					\node [style=none] (0) at (-9, 4) {};
					\node [style=none] (1) at (-9, 2) {};
					\node [style=none] (2) at (-6, 4) {};
					\node [style=none] (3) at (-6, 2) {};
					\node [style=none] (4) at (-12, 4) {};
					\node [style=none] (5) at (-12, 2) {};
					\node [style=none] (6) at (-3, 4) {};
					\node [style=none] (7) at (-3, 2) {};
					\node [style=reddot] (8) at (-8.25, 2.5) {};
					\node [style=reddot] (9) at (-6.75, 3.5) {};
					\node [style=none] (10) at (2, 4) {};
					\node [style=none] (11) at (2, 2) {};
					\node [style=none] (12) at (5, 4) {};
					\node [style=none] (13) at (5, 2) {};
					\node [style=none] (14) at (-1, 4) {};
					\node [style=none] (15) at (-1, 2) {};
					\node [style=none] (16) at (8, 4) {};
					\node [style=none] (17) at (8, 2) {};
					\node [style=none] (18) at (2, 2) {};
					\node [style=none] (19) at (2.75, 2.5) {};
					\node [style=none] (20) at (4.25, 3.5) {};
					\node [style=none] (21) at (5, 4) {};
					\node [style=none] (22) at (-2.5, 3) {};
					\node [style=none] (23) at (-1.5, 3) {};
					\node [style=none] (24) at (9, 3) {};
					\node [style=none] (25) at (9, -1) {};
					\node [style=none] (26) at (2, 0) {};
					\node [style=none] (27) at (2, -2) {};
					\node [style=none] (28) at (5, 0) {};
					\node [style=none] (29) at (5, -2) {};
					\node [style=none] (30) at (-1, 0) {};
					\node [style=none] (31) at (-1, -2) {};
					\node [style=none] (32) at (8, 0) {};
					\node [style=none] (33) at (8, -2) {};
					\node [style=none] (34) at (2, -2) {};
					\node [style=none] (35) at (2.75, -1.5) {};
					\node [style=none] (36) at (4.25, -0.5) {};
					\node [style=none] (37) at (5, 0) {};
					\node [style=none] (62) at (-9, 0) {};
					\node [style=none] (63) at (-9, -2) {};
					\node [style=none] (64) at (-6, 0) {};
					\node [style=none] (65) at (-6, -2) {};
					\node [style=none] (66) at (-12, 0) {};
					\node [style=none] (67) at (-12, -2) {};
					\node [style=none] (68) at (-3, 0) {};
					\node [style=none] (69) at (-3, -2) {};
					\node [style=none] (72) at (-1.5, -1) {};
					\node [style=none] (73) at (-2.5, -1) {};
					\node [style=bigdisk] (74) at (-7.5, -1) {};
					\node [style=none] (75) at (3.25, -1.5) {};
					\node [style=none] (76) at (3.5, -2) {};
					\node [style=none] (77) at (3.75, 0) {};
					\node [style=none] (78) at (4, -0.5) {};
				\end{pgfonlayer}
				\begin{pgfonlayer}{edgelayer}
					\draw [style=open, bend left=90, looseness=1.50] (0.center) to (1.center);
					\draw [style=open, bend right=90, looseness=1.50] (0.center) to (1.center);
					\draw [style=open, bend left=90, looseness=1.50] (2.center) to (3.center);
					\draw [style=open, bend right=90, looseness=1.50] (2.center) to (3.center);
					\draw (0.center) to (4.center);
					\draw (1.center) to (5.center);
					\draw (3.center) to (7.center);
					\draw (2.center) to (6.center);
					\draw [style=open, bend left=90, looseness=1.50] (10.center) to (11.center);
					\draw [style=open, bend right=90, looseness=1.50] (10.center) to (11.center);
					\draw [style=open, bend left=90, looseness=1.50] (12.center) to (13.center);
					\draw [style=open, bend right=90, looseness=1.50] (12.center) to (13.center);
					\draw (10.center) to (14.center);
					\draw (11.center) to (15.center);
					\draw (13.center) to (17.center);
					\draw (12.center) to (16.center);
					\draw [style=REDthick, bend right=15, looseness=1.25] (18.center) to (19.center);
					\draw [style=REDthick, bend left, looseness=1.25] (20.center) to (21.center);
					\draw [style=end arrow] (22.center) to (23.center);
					\draw [style=end arrow, bend left=90, looseness=1.25] (24.center) to (25.center);
					\draw [style=open, bend left=90, looseness=1.50] (26.center) to (27.center);
					\draw [style=open, bend right=90, looseness=1.50] (26.center) to (27.center);
					\draw [style=open, bend left=90, looseness=1.50] (28.center) to (29.center);
					\draw [style=open, bend right=90, looseness=1.50] (28.center) to (29.center);
					\draw (26.center) to (30.center);
					\draw (27.center) to (31.center);
					\draw (29.center) to (33.center);
					\draw (28.center) to (32.center);
					\draw [style=REDthick, bend right=15, looseness=1.25] (34.center) to (35.center);
					\draw [style=REDthick, bend left, looseness=1.25] (36.center) to (37.center);
					\draw (62.center) to (66.center);
					\draw (63.center) to (67.center);
					\draw (65.center) to (69.center);
					\draw (64.center) to (68.center);
					\draw [style=end arrow] (72.center) to (73.center);
					\draw (64.center) to (62.center);
					\draw (65.center) to (63.center);
					\draw [style=REDthick] (35.center) to (75.center);
					\draw [style=REDthick] (34.center) to (76.center);
					\draw [style=REDthick] (36.center) to (78.center);
					\draw [style=REDthick] (37.center) to (77.center);
					\draw [style=REDthick, in=180, out=0, looseness=0.75] (76.center) to (78.center);
					\draw [style=REDthick, in=180, out=0, looseness=0.75] (75.center) to (77.center);
				\end{pgfonlayer}
		\end{tikzpicture}\end{array}
	\end{align}
\end{definition}

For a symmetric monoidal bicategory $\cat{M}$, consider an object $X\in \cat{M}$ and a \emph{symmetric copairing}~\cite[Section~2.2]{cyclic}, i.e.\ a morphism $\Delta : I \to X \boxtimes X$ that is a homotopy fixed point under the homotopy involution on $\cat{M}(I,X\boxtimes X)$ induced by the symmetric braiding of $\cat{M}$. 
Denote by $\End_X^\Delta$ the modular endomorphism operad of $(X,\Delta)$, namely the $\Cat$-valued
modular operad $\Graphs \to \Cat$ sending a corolla $T$ to
$\cat{M}(X^{\boxtimes \Legs(T)},I)$
whose definition of morphisms in $\Graphs$ is afforded by $\Delta$ \cite[Section~2.3]{cyclic} (see in particular the comments at the end of that section).
Note that we do not ask $\Delta$ to be the copairing to a non-degenerate symmetric pairing, see also
\cite[Section~2.3]{costello} where the non-degeneracy condition is dropped as well.

More generally, for two objects $X$ and $Y$ with symmetric copairings $\Delta_X : I\to X \boxtimes X$ and $\Delta_Y:I\to Y\boxtimes Y$,
one may define a  two-colored modular endomorphism operad $\End_{X,Y}^{\Delta_X,\Delta_Y}$ with values in $\Cat$
that in total arity $(n,m)$ assigns the category $\cat{M}(   X^{\boxtimes n}   \boxtimes Y^{\boxtimes m} ,I )$ whose definition on compositions is defined via the copairings as in \cite[Section~2.3]{cyclic}.

\begin{definition}[Modular functor with singularities]\label{defmfsingularities}
	Let $\cat{M}$ be a symmetric monoidal bicategory.
	An $\cat{M}$-valued
	\emph{(open-closed) modular functor with singularities}
	consists of
	\begin{pnum}
		\item an object $\cat{A}$ (for the open boundary) equipped with a \emph{non-degenerate} symmetric pairing $\kappa :\cat{A}\boxtimes\cat{A}\to I$ whose associated copairing we denote by $\Delta : I \to \cat{A}\boxtimes\cat{A}$,
		\item an object $\cat{B}$ (for the closed boundary, i.e.\ for the bulk) equipped with a \emph{possibly degenerate} symmetric copairing $\Delta_\circ:I\to \cat{B}\boxtimes\cat{B}$,
		\item plus a morphism of two-colored modular operads $\sing \to \End_{\cat{A},\cat{B}}^{\Delta,\Delta_\circ}$ sending the interval to $\cat{A}$ and the circle to $\cat{B}$.
		\end{pnum}
\end{definition}

By restriction to the open sector, we obtain an open modular functor as discussed in \cite{envas}; in other words, we obtain a modular $\O$-algebra structure on $\cat{A}$ which is equivalent, by restriction to disks with marked intervals, to a cyclic associative algebra structure.
The restriction to the closed sector is what one might call \emph{modular functor with singularities} (this time without the qualifier `open-closed'), but we will focus in this paper entirely on the open-closed situation.

The addition `with singularities' takes into account \begin{itemize}
\item the fact that 
we exclude closed surfaces, but instead allow surfaces with punctures,
\item  
and that we need to add a puncture when gluing boundary circles.\end{itemize}

\subsection{Reminder on factorization homology}
For a $\PrL$-valued 
framed $E_2$-algebra $\cat{A}$,
one may define the \emph{factorization homology} \cite{AF,bzbj} $\int_{\Sigma} \cat{A}\in\PrL$
for any oriented surface $\Sigma$ as the homotopy colimit over $\cat{A}^{\boxtimes n}$ running over all oriented embeddings $\varphi : (\mathbb{D}^2)^{\sqcup n} \to \Sigma$ of $n\ge 0$ disks into $\Sigma$.
The embedding $\emptyset \to \Sigma$ induces an object $\cat{O}_\Sigma^\cat{A} \in \int_{\Sigma} \cat{A}$, the so-called \emph{quantum structure sheaf}~\cite{bzbj}. 

In particular, we obtain  the category $\int_{\mathbb{S}^1 \times [0,1]} \cat{A}$, namely the \emph{(categorical) Hochschild homology} of $\cat{A}$.
For brevity, we will denote it by $\int_{\mathbb{S}^1} \cat{A}$.
Note that this category can be defined if $\cat{A}$ is just an $E_1$-algebra in $\PrL$, but if $\cat{A}$ is actually framed $E_2$, we obtain
from the stacking along cylinders a monoidal structure on
$\int_{\mathbb{S}^1} \cat{A}$
 with unit $\cat{O}_{\mathbb{S}^1}^\cat{A}$.
By embedding an interval around a fixed base point of the circle, the induced embedding $[0,1]^2 \to \mathbb{S}^1 \times [0,1]$ yields a cocontinuous functor
\begin{align}
	L:\cat{A} \to \int_{\mathbb{S}^1 } \cat{A} \label{eqnbandtensorfunctor}
\end{align}
sending the monoidal unit in $\cat{A}$ to $\cat{O}_{\mathbb{S}^1}^\cat{A}$. 
In \cite[Section 4.2]{bzbj2} it is explained why this functor,
that is called the  \emph{band tensor functor}, is monoidal.

	\begin{question}\label{q2}
		In $\PrL$, every cocontinuous functor has a right adjoint (which of course need not be cocontinuous).
		This means that \eqref{eqnbandtensorfunctor} has a right adjoint $R: \int_{\mathbb{S}^1 } \cat{A} \to \cat{A}$.
		Consider now the monad $T=RL$ on $\cat{A}$ 
		and the Eilenberg-Moore category $\cat{A}^T$ of $T$-algebras in $\cat{A}$.
		The adjunction $(L,R)$ is called \emph{monadic}
		if the canonical comparison functor
		\begin{align}
			\int_{\mathbb{S}^1} \cat{A} \to \cat{A}^T
			\end{align}
			sending $Y \in \int_{\mathbb{S}^1} \cat{A}$ 
			to $RY$ plus the $T$-algebra structure induced by the counit $LR\to \id_{\int_{\mathbb{S}^1} \cat{A}}$
			is an equivalence, see e.g.\ \cite{riehlverity} for more background.	
		In the (possibly non-finite) rigid situation from Section~\ref{secghost}, the 
	 the reconstruction results in \cite{bzbj} tell us that $(L,R)$ is monadic.
	 For $\cat{A}=\RV$ for a vertex operator algebra $V$ and a cyclic category of $V$-modules, when is $(L,R)$ monadic?
	 Are there non-rigid examples producing a monadic adjunction?
		\end{question}

\subsection{Construction of the modular functor with singularities}
We may now state the main result on the construction of spaces of conformal blocks using the concept of an open-closed modular functor with singularities.
For the comparison to \cite{bsz}, we refer back to the comments on page~\pageref{bszcoparison} in the introduction.

	\begin{theorem}\label{thmmain}
For any cyclic category of $V$-modules, the construction $\Sigma \mapsto \block(\Sigma;-)$ extends from surfaces with at least one parametrized interval per connected component  to surfaces $\Sigma$ with $n$ parametrized boundary intervals, $m$ parametrized boundary circles and $p$ punctures
such that $n+m+p\ge 1$ holds separately for each path component of $\Sigma$.
For such a surface $\Sigma$, we obtain a cocontinuous
functor
\begin{align}
	\block(\Sigma;-): \RV^{\boxtimes n} \boxtimes \left(\int_{\mathbb{S}^1} \RV\right)^{\boxtimes m} \to \Vect  
\end{align}carrying an action of $\Map(\Sigma)$.
These mapping class group representations assemble into an open-closed modular functor with singularities
assigning $\RV$ to the boundary interval, $\int_{\mathbb{S}^1} \RV$ to the boundary circle, and for which the gluing properties, also referred to as excision, take the following form:
\begin{pnum}
	\item Excision~1: For any two fixed parametrized boundary intervals of $\Sigma$, there are canonical mapping class group equivariant isomorphisms
	\begin{align}
		\block(\Sigma';-)\cong \int^{t\in\catT} \block(\Sigma; \dots,t^*,\dots,t,\dots) \ , 
	\end{align}
	where $\Sigma'$ is the surface obtained by gluing along the two intervals, $*$ is the weak duality operation for compact projective objects of $\RV$ induced by the pairing, and the dummy variables are inserted into the slots affected by the gluing.
	\item Excision~2: For any two fixed parametrized boundary circles of $\Sigma$, there are canonical mapping class group equivariant isomorphisms
	\begin{align}
		\block(\Sigma'_+;-) \cong \int^{t\in \catT} \block (\Sigma;\dots,Lt^*,\dots,Lt,\dots) \  ,
	\end{align}
	where $\Sigma'_+$ is the surface obtained by gluing two boundary circles plus an additional puncture and $L:\RV\to\int_{\mathbb{S}^1} \RV$ is the band tensor functor.
\end{pnum}
\end{theorem}

\begin{proof}
	In the first step, let us construct
\begin{align}\label{eqnfunctorsfhx}	\block(\Sigma;-): \RV^{\boxtimes n} \boxtimes \left(\int_{\mathbb{S}^1} \RV\right)^{\boxtimes m} \to \Vect \ . 
	\end{align}
	 The descent of the modular extension
	 to factorization homology was already remarked in \cite[Theorem 5.1 and its proof]{envas} and is a lower-dimensional version of  the argument from \cite[Section~4]{brochierwoike}; in an $(\infty,1)$-categorical setting,
	  this has also been observed in \cite{bsz}. For completeness, 
	 let us give the details:
	 For an embedding of $\ell\ge 0$ intervals into the union $S$ of the $m$ boundary circles of $\Sigma$, we obtain from Theorem~\ref{thmmcgaction} a map \begin{align}\label{eqnmapsinducingfh}
	 \RV^{\boxtimes n} \boxtimes \RV^{\boxtimes \ell} \to \Vect \ . 
	 \end{align} This descends to the colimit over all such embeddings by virtue of $\block$ or, equivalently, $\Vo$ being an open modular functor,
	  and gives us a map \begin{align}\RV^{\boxtimes n}\boxtimes \int_{S\times [0,1]} \RV\to\Vect\end{align} on which the mapping class group acts via 2-automorphisms (because the diffeomorphisms that we consider preserve the boundary parametrizations). Through the parametrization, we obtain \eqref{eqnfunctorsfhx}.
	  
	  The functors \eqref{eqnfunctorsfhx} are exactly what we need for an open-closed modular functor with singularities assigning $\RV$ to the interval and $\int_{\mathbb{S}^1} \RV$ to the circle.
	  We also need to exhibit the (co)pairings. For $\RV$, we already have this of course, see \eqref{eqnDelta}.
	  	The needed symmetric  copairing for $\int_{\mathbb{S}^1} \RV$ is
	  \begin{align}
	  	\Delta_\circ : \Vect \ra{\Delta} \RV\boxtimes \RV \ra{L\boxtimes L} \int_{\mathbb{S}^1} \RV\boxtimes \int_{\mathbb{S}^1} \RV \ , 
	  \end{align}
	  where $L$ is the band tensor functor from \eqref{eqnbandtensorfunctor}.
	  The symmetry of $\Delta_\circ$ is induced by the one of $\Delta$ and the symmetric braiding of $\PrL$.
	  
	  It remains to verify the gluing properties:
	  For the gluing of intervals, this is Corollary~\ref{corexcinterval}.
	  For the gluing along boundary circles (with the added puncture), we need a preparation:
	  If  for the surface $\Sigma$
	   we denote by $\Sigma^\ddagger$ the surface for which we have replaced the two specific boundary circles with a marked interval around the image of some fixed base point on the standard circle under the parametrization, then \eqref{eqnfunctorsfhx} is defined such that
	  \begin{align}\block(\Sigma;\dots,LX,\dots,LY,\dots)\cong \block(\Sigma^\ddagger;\dots,X,\dots,Y,\dots ) \, \end{align}
	  where $X,Y\in\RV$ and their images under the band tensor functor are inserted into the affected slots.
	  If we now set $X=t^*$ and $Y=t$ for $t\in\catT$ and take the coend over $\catT$, then we are left with
	   \begin{align} \int^{t\in\catT} \block(\Sigma;\dots,Lt^*,\dots,Lt,\dots) \cong \int^{t\in\catT} \block(\Sigma^\ddagger;\dots,t^*,\dots,t,\dots )\cong \block(\Sigma'_+;-) \ , 
	   	\end{align}
   	where the last step is again Corollary~\ref{corexcinterval}.
   	In order to see that this amounts indeed to the operadic gluing law for the boundary circles, we need to verify 
   	\begin{align} \int^{t\in\catT} \block(\Sigma;\dots,Lt^*,\dots,Lt,\dots) \cong \block(\Sigma;\dots,\Delta_\circ',\dots,\Delta_\circ'',\dots) \ , 
   		\end{align} but this follows because $\Delta_\circ = \int^{t\in\catT} Lt^* \boxtimes Lt$ (use Theorem~\ref{thmcfe2} and cocontinuity of $L$) and cocontinuity of $\block(\Sigma;-)$.
	\end{proof}

	\begin{remark}
		The functors\begin{align}
		\block(\Sigma;-): \RV^{\boxtimes n} \boxtimes \left(\int_{\mathbb{S}^1} \RV\right)^{\boxtimes m} \to \Vect\end{align} constituting the modular functor with singularities
		 come with a canonical isomorphism
			\begin{align}
			\block(\widetilde\Sigma;-)\cong \block( \Sigma;-,\O(V)^{\boxtimes q}):\RV^{\boxtimes n} \to \Vect \quad \text{with}\quad \O(V) = \cat{O}_{\mathbb{S}^1}^{\RV} \ , \label{eqnevO}
			\end{align}
			where $\widetilde \Sigma$ is obtained by replacing the boundary circles of $\Sigma$ with punctures.
			The isomorphism~\eqref{eqnevO} was already derived in the proof of \cite[Theorem~5.1]{envas}.
			We call the algebra
			\begin{align}
			\CF (V):= 
			\End_{\int_{\mathbb{S}^1} \RV}(\cat{O}(V))
			\end{align}
			the \emph{algebra of class functions of $V$} (and its choice of modules).
			The terminology is justified because we recover in the finite rigid case the class functions of \cite{shimizucf} as explained in \cite[Section~5]{envas}.
			Since $\RV$ is braided, $\int_{\mathbb{S}^1} \RV$ is monoidal with unit $\cat{O}(V)$, so that $\CF(V)$ is commutative by the Eckmann-Hilton argument.
			From \eqref{eqnevO},
			we deduce that $\block(\Sigma;-)$ inherits the structure of a $\CF(V)$-module for every puncture, and the $\CF(V)$-module structures for different punctures commute.
		\end{remark}
	
	\begin{question}\label{q22}
		How can one compute the algebra $\CF(V)$ of class functions directly in terms $V$?
		\end{question}

We note one immediate consequence of Theorem~\ref{thmmain}:
\begin{corollary}\label{corpairingfh}
The evaluation
of the construction of Theorem~\ref{thmmain}
on the annulus with two parametrized boundary circles yields a symmetric pairing
\begin{align}
	\beta : \int_{\mathbb{S}^1}     \RV\boxtimes \int_{\mathbb{S}^1}     \RV \to \Vect \ . 
\end{align}
\end{corollary}

\begin{remark}\label{rembetanondeg}
In the finite rigid case, this pairing has been constructed in \cite[Section~3]{sn} and is also non-degenerate.
Beyond the finite rigid case, this is not known.
We should also remark that the copairing $\Delta_\circ \in \int_{\mathbb{S}^1} \RV\boxtimes \int_{\mathbb{S}^1}\RV$ from the proof of Theorem~\ref{thmmain}
 is generally \emph{not} the copairing for $\beta$ even in cases in which $\beta$ is non-degenerate.
 There is nonetheless an interesting relation between $\beta$ and $\Delta_\circ$: By Theorem~\ref{thmmain} and its proof, the vector space
 $
 	\beta \circ \Delta_\circ$ is isomorphic to the space of conformal blocks 
 	$\block(\mathbb{T}^2_\times)$
 	for the punctured torus  and hence comes with an action of its mapping class group, the braid group $B_3$ on three strands.
\end{remark}

\begin{remark}\label{rembulkboundary}
	In Theorem~\ref{thmmain}, the value for the circle is the factorization homology for the value on the interval.
	This is not the only possibility, but it is the unique choice if we want the closed sector, i.e.\ the bulk, to be induced by the open boundary as in the proof of Theorem~\ref{thmmain}.
	In a physical language, it is argued by Fuchs-Schweigert-Valentino in \cite{fsv} that this is the natural requirement for a bulk-boundary system in which the bulk is induced by the boundary, see also \cite[Remark~6.3]{sn} for a comment in the context of the string-net model.
	Roughly, the argument by Fuchs-Schweigert-Valentino says that for a system in which $\cat{A}$ describes the boundary (so it is at least monoidal) and $\cat{B}$ the bulk (so it is at least balanced braided), there is a map $\cat{B}\to Z(\cat{A})$ pushing operators to the boundary.
	Operadically, this map is part of the structure of a categorical Swiss-Cheese algebra \cite{najib}. If we want the bulk to be induced by the boundary, this map should be an equivalence, thereby providing a Witt trivialization for $\cat{B}$. 
	In other words, the system placing $Z(\cat{A})$ (or in a dual picture $\int_{\mathbb{S}^1} \cat{A}$) on the circle is distinguished.
	 In an $(\infty,1)$-categorical set-up, one may also prove that it is \emph{initial} in a precise sense \cite[Theorem 1.6]{bsz}.
	\end{remark}

	\subsection{Representations of surface braid groups and mapping class groups of closed surfaces}
	Modular functors with singularities do not include
	representations of mapping class groups of closed surfaces.
	There is, however, always at least one possibility to produce such representations even though
	are not necessarily part of a modular functor.
	To this end, recall that 
for a compact oriented surface $\Sigma$, the \emph{surface braid group $B_n(\Sigma)$ on $n\ge 1$ strands} is the fundamental group of the configuration space of 
$n$ points in $\Sigma$~\cite[Section~9.1.4]{farbmargalit}.

\begin{proposition}
	Let $\Sigma_\times$ be a connected closed surface $\Sigma$ plus $n\ge 1$ punctures. Then
	$\block(\Sigma_\times)$ inherits a representation of the surface braid group $B_n(\Sigma)$; this is a restriction to a subgroup unless $\Sigma = \mathbb{S}^2,\mathbb{T}^2$.
	The invariants $\block(\Sigma_\times)^{B_n(\Sigma) }$ carry an action of $\Map(\Sigma)$.
\end{proposition}

\begin{proof} 
	By	\cite[Theorem~9.1]{farbmargalit} (and its proof),
	we obtain an exact sequence
	\begin{align}
	B_n(\Sigma) \ra{p} \Map(\Sigma_\times)\ra{\pi} \Map(\Sigma) \to 1 \ ,
	\end{align} and $p$ is injective unless $\Sigma=\mathbb{S}^2,\mathbb{T}^2$. 
	The modular functor with singularities gives us on $\block(\Sigma_\times)$ 
	 an action of $\Map(\Sigma_\times)$ that we restrict along $p$ to obtain
	the desired $B_n(\Sigma)$-representation. 
	The invariants of this restriction $p^* \block(\Sigma_\times)$ along $p$
	are the invariants of $\im p=\ker \pi$ for the original representation 
	of $\Map(\Sigma_\times)$ on $\block(\Sigma_\times)$. 
	
	Now we use the following basic fact:
	For any exact sequence $1\to A \to G\to H\to 1$ of groups, the invariants $V^A$ for any $G$-representation carry an action of $H$.
	We apply this to \begin{align}1\to \ker \pi \to \Map(\Sigma_\times)\ra{\pi} \Map(\Sigma) \to 1 \ . 
	\end{align}
\end{proof}

\subsection{Comparison result in the finite rigid case}
The spaces of conformal blocks $\block$ based on a cyclic category $\catT$ of $V$-modules
extend always to an open-closed modular functor with singularities 
in the sense of Theorem~\ref{thmmain}.
At that level of generality (without semisimplicity, rigidity or finiteness), there is no modular functor construction that we could compare it to.
In the finite rigid special case, however, there is:
Suppose that $\RV$ is, as in Section~\ref{secfiniterigid}, the ind completion of a finite tensor category, actually even a modular category $\Vmod$, see Corollary~\ref{corvmod}.
In that case, $\Vmod$ can also be seen as a cyclic framed $E_2$-algebra in the symmetric monoidal bicategory $\Lexf$.
The $\Lexf$-version of Theorem~\ref{thmmain} then gives us an open-closed modular functor with singularities $\blockl$ (we are turning the subscript `$V$' into a superscript). 
We can now easily compare with 
Lyubashenko's modular functor \cite{lyubacmp,lyu,lyulex,kl} for $Z(\Vmod)$ by re-using the comparison results of \cite{sn,envas}.
We will afterwards strengthen the comparison by using also the results of \cite{mwskein,damioliniwoike,reflection,correlators}.
Note that, after relying on the results, none of the comparisons in this subsections are difficult, but we still want to collect them in one place.

\begin{corollary}[based on \cite{sn,envas}]\label{corfiniterigid}
	In the finite rigid case, the singularities of $\blockl$ can be removed, i.e.\ $\blockl$ extends to a $\Lexf$-valued open-closed modular functor without singularities that associates \begin{itemize}
		\item $\Vmod=\catf{comp}(\RV)$ to the interval,
		\item  and $Z(\Vmod)\simeq \overline{\Vmod}\boxtimes \Vmod$ to the circle.\end{itemize}	
	The topologically inherited cyclic framed $E_2$-structure on $Z(\Vmod)$ agrees with the standard balanced braided structure with the rigid duality of $Z(\Vmod)$.
	The closed part of the open-closed modular functor $\blockl$ is equivalent to the Lyubashenko modular functor for $Z(\Vmod)$.
	\end{corollary}

This result follows from another one that we prove instead:
In \cite{sn} the string-net techniques \cite{levinwen,kirillovsn} are generalized to produce a string-net open-closed modular functor $\SNA$ built from a pivotal finite tensor category $\cat{A}$.

\begin{corollary}[based on \cite{envas}]\label{corfiniterigid2}
	Under the assumptions of Corollary~\ref{corfiniterigid},
	$\blockl \simeq \catf{SN}_{\Vmod}$
	as open-closed modular functors with singularities.
	In other words, $\blockl$ generalizes the string-net construction and inherits in the finite rigid case an extension to an open-closed modular functor without singularities from the string-net construction.
\end{corollary}

\begin{proof}
	This follows from Theorem~\ref{thmmcgaction}
	and \cite[Theorem~4.2]{envas}.
	\end{proof}

\begin{proof}[\slshape Proof of Corollary~\ref{corfiniterigid}]
	Corollary~\ref{corfiniterigid2} gives us the claimed extension.
	Now \cite[Theorem 7.1]{sn}, that asserts that $\catf{SN}_{\Vmod}$ models the Lyubashenko modular functor for $Z(\Vmod)$,
	yields the statement.
	Finally note that $Z(\Vmod)\simeq \overline{\Vmod}\boxtimes \Vmod$ because modular categories are factorizable \cite{eno-d,shimizumodular}.
\end{proof}

Suppose that $\Sigma$ is connected and has $n\ge 0$ embedded intervals.
This turns $\int_\Sigma \Vmod$ into an $\Vmod^{\boxtimes n}$-module category
whose action functor we denote by $-\act - : \Vmod^{\boxtimes n} \boxtimes \int_\Sigma \Vmod\to\int_\Sigma \Vmod$, see \cite{bzbj}, and the results of this article actually imply that this is an action inside $\Rexf$, see e.g.\ \cite{becerrawoike} for further explanations.
For $X\in \int_\Sigma \Vmod$, the internal hom $\HOM(X,-):\int_\Sigma \Vmod\to\Vmod^{\boxtimes n}$ is defined as the right adjoint to the functor $-\act X:\Vmod^{\boxtimes n} \to \int_\Sigma\Vmod$. 
The \emph{moduli algebra} $\moduli_\Sigma^V:= \END(\cat{O}_\Sigma^V)$ is an algebra in $\Vmod^{\boxtimes n}$ carrying a $\Map(\Sigma)$-action through algebra maps~\cite[Section 5.2]{bzbj}. It generalizes the moduli algebras of Alekseev-Grosse-Schomerus~\cite{alekseevmoduli,agsmoduli,asmoduli}.
For the torus with one boundary component, the moduli algebra is the \emph{elliptic double} from \cite{brochierjordane}.

For the next statement, observe that for a finite ribbon category $\cat{A}$, the power $\cat{A}^{\boxtimes n}$ inherits a non-degenerate symmetric pairing $X \boxtimes Y \mapsto \Hom_{\cat{A}^{\boxtimes n}}(I,X\otimes Y)$, where $\otimes$ is the monoidal product in $\cat{A}^{\boxtimes n}$. 
This pairing induces an equivalence
\begin{align}
\Lex[\cat{A}^{\boxtimes n},\vect]\simeq \Lex[\vect,\cat{A}^{\boxtimes n}]\simeq \cat{A}^{\boxtimes n}\label{eqnequivAn}
\end{align} ($\vect$ denotes the category of finite-dimensional vector spaces)
that we use to turn $\Lex[\cat{A}^{\boxtimes n},\vect]$ into a monoidal category.

\begin{corollary}[based on \cite{correlators}]\label{corfiniterigidmodulialgebra}
	Under the assumptions of Corollary~\ref{corfiniterigid},
	\begin{align}\blockl(\Sigma;-)\cong \Hom_{\Vmod^{\boxtimes n}}(I,-\otimes \moduli_\Sigma^V)\label{eqnblocklma} \end{align}
	as left exact functor $\Vmod^{\boxtimes n}\to\Vect$ with $\Map(\Sigma)$-action.		
	In particular, $\blockl(\Sigma;-)$ inherits the structure of an algebra
	in $\Lex[\Vmod^{\boxtimes n},\vect]$ such that the mapping class group representation is through algebra maps. 
	\end{corollary}
	
\begin{proof}
	The isomorphism~\eqref{eqnblocklma} is a consequence of
	the holographic principle from \cite[Theorem 6.1 and eq.~(6.5)]{correlators}.
	Since $\blockl(\Sigma;-)$ is an algebra under the equivalence \eqref{eqnequivAn}, namely the moduli algebra, it inherits an algebra structure. The mapping class group action is through algebra maps because this is already the case for the moduli algebra.
	\end{proof}

\begin{corollary}[based on \cite{mwskein,reflection}]\label{corsimplemodule}
	Under the assumptions of Corollary~\ref{corfiniterigid}, let $H$ be a three-dimensional handlebody with boundary surface $\Sigma = \partial H$.
	Then the algebra $\blockl(\Sigma)$ has, up to isomorphism, one simple module, namely the admissible skein module $\catf{Sk}_V(H)$ for $\Vmod$ of Costantino-Geer-Patureau-Mirand.
	The action map $\blockl(\Sigma) \ra{\cong}\End(\catf{Sk}_V(H))$ is an isomorphism.
	The unique projective $\Map(\Sigma)$-action on $\catf{Sk}_V(H)$ making the module action equivariant is equivalent to Lyubashenko's mapping class group representation for the closed surface $\Sigma$ and the modular category $\Vmod$.
	\end{corollary}

\begin{proof}
	The algebra $\blockl(\Sigma)$ is a matrix algebra and has a unique simple module, namely the value $\widehat{V}(H)$
	of the ansular functor for $\Vmod$ on $H$, and the action map is an isomorphism; this follows from  Corollary~\ref{corfiniterigidmodulialgebra} and \cite[Theorem~6.8]{reflection}.
	
	Under the identification of $\blockl(\Sigma)$ with the moduli algebra for $\Sigma$, i.e.\ the skein algebra (the one defined using factorization homology), $\widehat{V}(H)$ is isomorphic, as skein module, to the admissible skein module from $\catf{Sk}_V(H)$ from \cite{asm} by the comparison in \cite{mwskein}.
	The connection to Lyubashenko's mapping class group representations is again \cite[Theorem~6.8]{reflection}.
	\end{proof}

\begin{question}\label{q3}
	In the finite rigid case, when $\Vmod$ is a modular category, the proof above together with \cite[Corollary 6.6]{reflection}
	tells us that the factorization homology of $V$-modules over the closed surface
	$\Sigma$ is the category of modules over the skein algebra, and that this module category is trivial in the sense that it is generated by one simple module, namely the skein module for any handlebody $H$ with $\partial H$; this is the Morita triviality of skein algebras, which is one of the main ideas of \cite{roberts,masbaumroberts,asmoduli} in the rational case.
	For arbitrary $V$ and its presentable module category $\RV$ based on some cyclic subcategory of $V$-modules, no such results are known, and the best course of action would be to understand $\int_\Sigma \RV$. Is there, in that case, any generalization of the Morita triviality of skein algebras?
	\end{question}

\begin{question}\label{q3a}
	We insist on the fact the  input datum for our construction is a cyclic framed $E_2$-algebra even though we forget, as for the string-net construction, the braiding and the balancing in the first step and just retain the cyclic associative algebra. In the finite rigid situation, the lift to a framed $E_2$-algebra is what enables us to obtain the holographic description from Corollary~\ref{corfiniterigidmodulialgebra} and the three-dimensional skein description from Corollary~\ref{corsimplemodule}.
	It seems likely that a similar description is available in the infinite rigid case based on the results in \cite{yeralopen,yeralthesis}. 
	Beyond the rigid case, it is  unclear to us what the correct  generalizations of Corollary~\ref{corfiniterigidmodulialgebra} and~\ref{corsimplemodule} could be.
	Most likely, one would first need to answer Question~\ref{q3}.
\end{question}

	\subsection{Relation to the spaces of conformal blocks à la Frenkel-Ben-Zvi}
Let $V$ be a strongly rational vertex operator algebra and denote by $\mathbb{V}$ the spaces of conformal blocks built from $V$ in that situation in the work of Frenkel-Ben-Zvi \cite{frenkelbenzvi} and Damiolini-Gibney-Tarasca \cite{DGT1,DGT2}.
Denote moreover by $\cat{C}_V$
the modular fusion category of $V$-modules obtained by evaluation of $\mathbb{V}$ in genus zero through \cite{damioliniwoike}.
 As explained in Section~\ref{secag},
 this is a possible choice for a cyclic category of $V$-modules.
 With this choice, the following holds:

\begin{corollary}[based on \cite{damioliniwoike}]\label{corthroughdw}
	Let $V$ be a strongly rational vertex operator algebra, and choose as cyclic category of $V$-modules the category $\cat{C}_V$ of admissible $V$-modules with its geometrically constructed structure of a modular fusion category inherited from the spaces of conformal blocks $\mathbb{V}$ in the sense of Frenkel-Ben-Zvi.
	Then for any surface $\Sigma$ with $n$ boundary components, there is a isomorphism
	\begin{align}
		\bigoplus_{i_1,\dots,i_n=1}^m	\mathbb{V}(\bar \Sigma;Y_{i_1}^\vee,\dots,Y_{i_n}^\vee)\otimes\mathbb{V}(\Sigma; X_1\otimes Y_{i_1},\dots, X_n\otimes Y_{i_n}) \ra{\cong} \block(\Sigma; X_1,\dots,X_n)
	\end{align} natural in $X_1,\dots,X_n \in\cat{C}_V$, where all $Y_{i_j}$ runs over the simple objects of $\cat{C}_V$. 
	More compactly, there is an isomorphism
	\begin{align}
		\Nat \left(  \mathbb{V}( \Sigma;-) ,  \mathbb{V}(\Sigma;X_1\otimes-,\dots,X_n\otimes - \right) \ra{\cong} \block(\Sigma;X_1,\dots,X_n) \ . 
	\end{align}
	In particular, if $\Sigma$ is closed, there is a canonical isomorphism
	\begin{align}\label{eqnVVblock}
		\mathbb{V}(\bar \Sigma)\otimes \mathbb{V}(\Sigma)\cong	\End \left(  \mathbb{V}( \Sigma)\right) \ra{\cong} \block(\Sigma) \ . 
	\end{align}
\end{corollary}

\begin{proof}
	Since $\cat{C}_V$ is finite rigid, $\block$ extends to an open-closed modular functor, namely the one for $Z(\cat{C}_V)\simeq \overline{\cat{C}_V} \boxtimes \cat{C}_V$, see Corollary~\ref{corfiniterigid}.
	Now the statement follows from \cite[Theorem 5.5.1]{damioliniwoike}.
	\end{proof}

The isomorphism~\eqref{eqnVVblock} can, even without the results of this article, be easily formulated for $\Vo$ instead of $\block$. The novelty is really the isomorphism to $\block$ with its description as a colimit over markings. Such a description of the Frenkel-Ben-Zvi blocks is, to be best of our knowledge, new.
Note that in contrast to the corollaries on the finite rigid case, Corollary~\ref{corthroughdw} requires semisimplicity.
Some open questions regarding the finite rigid non-semisimple case are recorded in \cite[Section~5.6]{damioliniwoike}.

\needspace{15\baselineskip} 
\section{Description of the mapping class group representations on generators\label{secgenerators}}
In this section,
we describe the mapping class group representations 
constructed in the previous sections
 on generators of 
the mapping class groups by using  the description in terms of markings.

\subsection{The action of the Dehn twists of an annulus}
As always, we fix our cyclic subcategory $\catT$
of $V$-modules.
Denote by
\begin{align}
	\gamma_1 , \gamma_2 : \int^{t_1,t_2 \in \catT} \spr{X,t_1,t_2^*} \otimes \spr{Y,t_1^*,t_2}\ra{\cong} \int^{s\in\catT} \spr{X,s,Y,s^*}\label{eqngamma12}
\end{align}
the isomorphisms obtained by applying \eqref{eqnY} 
 to the coend over $t_1$ and $t_2$, respectively.
Roughly, this is nothing but a version of the Yoneda Lemma, but applied to the two different coends.

\begin{proposition}\label{propdehntwist}
	For the annulus $A$ with two parametrized intervals, $\block(A;X,Y)$ for $X,Y \in \RV$ is given by
	\begin{align}
		\block(A;X,Y)\cong	\int^{s\in\catT} \spr{X,s,Y,s^*} \ , 
	\end{align}
	and under this identification, the action $d_*$ of the Dehn twist $d:A\to A$ is given by
	\begin{align}
		d_* = \gamma_2\circ \gamma_1^{-1} \ , 
	\end{align} i.e.\
	the Dehn twist acts by the operator describing the multiplicative difference between $\gamma_1$ and $\gamma_2$.
\end{proposition}

\begin{proof}
	To compute the action of a Dehn twist,
	we start by picking the marking $\mu$ of $A$ drawn on the left in Figure~\ref{Fig:Dehn}. 
	By acting with the Dehn twist we obtain the marking on the right.
	
	\begin{figure}[h]
		\centering
		\begin{overpic}[scale=0.2]{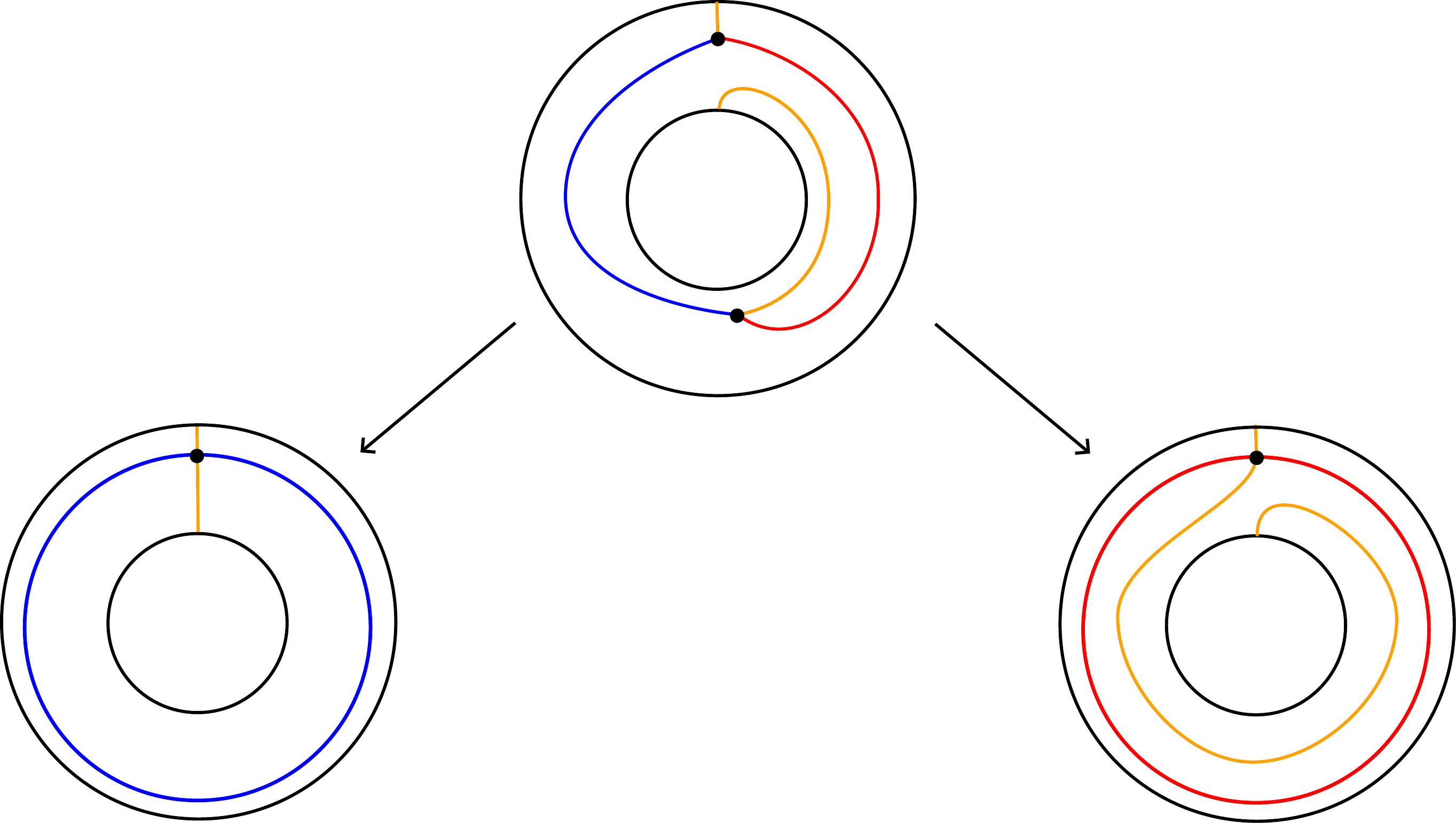}
			\put(12,29){$\mu$} 
			\put(49,26.5){$\nu$} 
			\put(84,29){$d.\mu$}
		\end{overpic} 
		\caption{A zigzag between the marking $\mu$ for an annulus and its image $d.\mu$ under a Dehn twist $d$.
			The pieces of the marking are drawn in color, but we choose different colors to distinguish the different pieces.
			The arrow from the middle to the left contracts the red part; the arrow from the middle to the right contracts the blue part.}
		\label{Fig:Dehn}
	\end{figure}

	The space of conformal blocks that we calculate for both markings $\mu$ and $d.\mu$
	is $\int^{s\in\catT} \spr{-,s,-,s^*}$, see also Remark~\ref{remmarking}.
	The blank spaces `$-$' can be filled with boundary labels.
	Figure~\ref{Fig:Dehn} also shows a zigzag between the two markings. Fixing two boundary objects $X,Y$ and evaluating $\block$ on the zigzag gives us the following:
	\begin{equation}
		\begin{tikzcd}
			& \hspace{-1.3cm} \ar[rd, "\gamma_2"] \ar[ld, "\gamma_1",swap] \block^{\nu}(A,X,Y) = \int^{t_1,t_2 \in \catT} \spr{X,t_1,t_2^*} \otimes \spr{Y,t_1^*,t_2} \hspace{-1.3cm} &  \\ 
			\block^{\mu}(A,X,Y) = \int^{s\in\catT} \spr{X,s,Y,s^*}\hspace{-0.6cm} & & \hspace{-0.6cm} \block^{d.\mu}(A,X,Y) = \int^{s\in\catT} \spr{X,s,Y,s^*}
		\end{tikzcd} 
	\end{equation}
	Now the claim follows from Theorem~\ref{thmmain}.
\end{proof}

The same argument
carries over to an annulus with more than one boundary label.

\begin{corollary}\label{cordehn}	
	For an annulus with $n\ge 1$ and $m\ge 1$ marked intervals on the two boundary circles,
	\begin{align}
		\block(A; X_1,\dots,X_n,Y_1,\dots,Y_m) \cong \int^{s\in\catT} \spr{X_1,\dots,X_n,s,Y_1,\dots,Y_m,s^*}
	\end{align} 
	and denoting by a slight abuse of notation by $\gamma_1$ and $\gamma_2$ also the generalizations of the maps defined in~\eqref{eqngamma12} to multiple objects 
	\begin{align}
		d_* = \gamma_2\circ \gamma_1^{-1} \ . 
	\end{align}
\end{corollary}

\begin{proof}
	After multiplying together the labels on the two boundary components, the statement reduces to Proposition~\ref{propdehntwist}.
\end{proof}

\begin{corollary}
	There is an isomorphism
	\begin{align}
		\block(\mathbb{S}^1 \times [0,1]) \cong \int^{t\in \catT}\Hom_V(t,t)  \ . \label{eqnHHcoend}
	\end{align}
	In particular,
	for $t,u\in\catT$, there is a canonical map
	\begin{align}
		\catT(t,u)\otimes\catT(u,t) \to 	\block(\mathbb{S}^1 \times [0,1]) \ . \label{eqncomposition}
	\end{align}
	Moreover, $\int^{t\in \catT}\Hom_V(t,t)$ inherits a $\mathbb{Z}_2$-action.
\end{corollary}

\begin{proof}
	The isomorphism \eqref{eqnHHcoend} follows from
	Corollary~\ref{corexcinterval} and \eqref{eqnsprhom}.
	The map \eqref{eqncomposition} composes over $t$ or $u$ and maps to the coend. By the universal property of the coend it actually does not matter whether we compose over $t$ or $u$.
	The $\mathbb{Z}_2$-action on $\int^{t\in \catT}\Hom_V(t,t)$ is a consequence of the fact the
	mapping class group of the annulus without parametrized boundary is $\mathbb{Z}_2$, see e.g.\ \cite[Section~2.2.2]{farbmargalit}.
\end{proof}

\begin{remark}	One the right hand of \eqref{eqnHHcoend}, we obtain the zeroth Hochschild homology $HH_0(\RV)$ of $\RV$ (or the so-called \emph{cocenter}), which is defined as a coend over the endomorphism spaces of compact projective objects
	\cite{keller}.
	Note that the observation~\eqref{eqnHHcoend} was already made in special cases \cite{mwdiff,mwansular,envas,yeralopen}.
	One may also interpret \eqref{eqnHHcoend} in the sense that the Hattori-Stallings trace $\catT(t,t)\to HH_0(\RV)$ \cite{hattori,stallings}
	of endomorphisms of compact projective objects of $\RV$ is naturally  $\block(\mathbb{S}^1 \times [0,1])$-valued.
\end{remark}

\subsection{The action of Dehn twists adapted to a marking}
Let $\mu=(R,f) \in \M(\Sigma)$, where $R$ has just one vertex, so that it is a bouquet of $r$ circles, each of which is mapped by $f$ to a simple closed curve in $\Sigma$. We  say that such a  simple closed curve is a part of $\mu$.
Denote by $d$ the Dehn twist about one of these curves $\gamma$.
Then by excision for $\block$ (Corollary~\ref{corexcinterval}) and \eqref{eqnrvdisk}
\begin{align}
	\block^\mu(\Sigma;-)\cong \int^{\dots} \spr{\dots,s,\dots,s^*} \cong \int^{\dots} \int^{t_1,t_2} \spr{\dots, t_1}\otimes \spr{t_1, s, t_2,s^*} \otimes \spr{t_2,\dots}  \ , \label{eqnaddcoend}
\end{align}
where the first coend runs over $r-1$ dummy variables, and the second one adds the two new dummy variables $t_1$ and $t_2$ via Corollary~\ref{corspr} splitting the coend as indicated in Figure~\ref{figeqnaddcoend}: The two half edges corresponding to $\gamma$ split the half edges at the single vertex into two parts (the ones on top in the right half of Figure~\ref{figeqnaddcoend} and the ones at the bottom; 	 we allow one of the parts to be empty). The dummy variables $t_1$ and $t_2$ are inserted so that each of the parts containing dots in the expression contains one part.

\begin{figure}[h]
	\centering
	\begin{overpic}[scale=0.5]{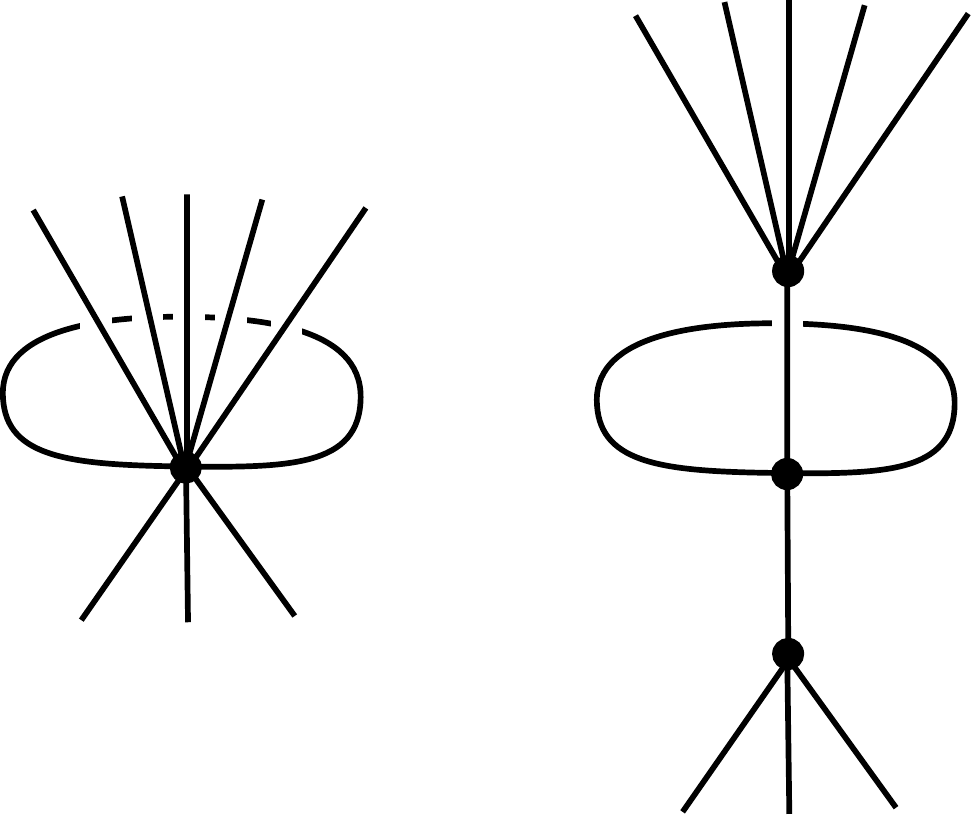}
		\put(38,40){$s$} 
			\put(100,40){$s$}
			\put(82,45){$t_1$}
			\put(82,25){$t_2$}
	\end{overpic}
	\caption{Graphical explanation for \eqref{eqnaddcoend}. The half edges at the top and bottom are connected in an arbitrary way to form the bouquet on the left. They are connected the same on the right}
	\label{figeqnaddcoend}
\end{figure}
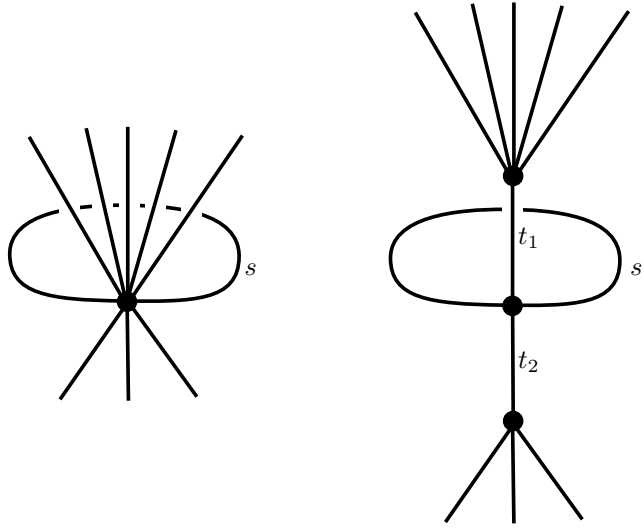

By applying~\eqref{eqngamma12}
to the inner coend, we obtain
isomorphisms $\gamma_1(d)$ and $\gamma_2(d)$ 
from
the right hand side of~\eqref{eqnaddcoend}
to $\block^\mu(\Sigma;-)$.
We denote them by $\gamma_1^\mu(d)$ and $\gamma_2^\mu(d)$, respectively.

\begin{corollary}\label{cordehnadapted}
	With the notation as above,  the isomorphism
	\begin{align}
		\psi_\mu:	\block^\mu(\Sigma;-) \ra{\cong} \block(\Sigma;-) \ , 
	\end{align}
	from Proposition~\ref{propmarking}	
	becomes equivariant for the $\mathbb{Z}$-action that 
	\begin{itemize}
		\item on the left hand side is given by $\gamma_2^\mu(d) \circ\gamma_1^\mu(d)^{-1}$,
		\item and on the right hand side by the action ${d}_*$ of the Dehn twist $d$ about a simple closed curve that is part of $\mu$.
	\end{itemize}
	In equations,
	\begin{align}
		d_* \circ \psi_\mu = \psi_\mu \circ \gamma_2^\mu(d) \circ \gamma_1^\mu(d)^{-1} \ . 
	\end{align}
\end{corollary}

\begin{proof}
	This follows from Theorem~\ref{thmmain} and Corollary~\ref{cordehn}.\end{proof}

\subsection{The elliptic space of conformal blocks}
Thanks to Corollary~\ref{cordehnadapted}, we can calculate the action on the space of conformal blocks of the torus with one boundary component:
Recall that we can obtain the torus $\mathbb{T}^2_1$
with one boundary component from the  ribbon graph $R$ in Figure~\ref{figtorus}. We model $\mathbb{T}^2_1$ simply as $|R|$ and denote the resulting marking by $\mu$.
Then
\begin{align}
	\block^\mu(\mathbb{T}^2_1;p) \cong \int^{u,v \in \catT} \spr{p,u^*,v^*,u,v} 
\end{align} on any compact projective object $p\in\RV$, see also \cite[Example 3.8]{yeralopen} for the non-finite rigid case.

\begin{figure}[h]
	\centering
	\begin{tikzpicture}[scale=0.5]
		\begin{pgfonlayer}{nodelayer}
			\node [style=none] (8) at (0.25, 3.5) {};
			\node [style=none] (9) at (0.25, -3.5) {};
			\node [style=none] (10) at (0.25, 1.5) {};
			\node [style=none] (11) at (0.25, 0.5) {};
			\node [style=none] (12) at (0.25, -0.5) {};
			\node [style=none] (13) at (0.25, -1.5) {};
			\node [style=none] (14) at (0.25, -2.5) {};
			\node [style=none] (15) at (0.25, -3.5) {};
			\node [style=none] (16) at (0.25, -2.5) {};
			\node [style=none] (17) at (0.25, 3.5) {};
			\node [style=none] (18) at (0.25, 2.5) {};
			\node [style=none] (19) at (-4.75, 1) {};
			\node [style=none] (20) at (-4.75, -1) {};
			\node [style=none] (22) at (-4.75, -1) {};
			\node [style=none] (23) at (-5.25, 0) {$1$};
			\node [style=none] (25) at (0, 3.75) {$2$};
			\node [style=none] (26) at (0, 2) {$3$};
			\node [style=none] (27) at (0, 0) {$4$};
			\node [style=none] (28) at (0, -3.75) {$5$};
			\node [style=none] (32) at (-1.75, 0) {};
			\node [style=none] (33) at (-4.75, 0) {};
			\node [style=none] (35) at (0.25, -3) {};
			\node [style=none] (36) at (0.25, -1) {};
			\node [style=none] (37) at (0.25, 1) {};
			\node [style=none] (38) at (0.25, 3) {};
			\node [style=none] (39) at (-2.25, 0.5) {$R$};
		\end{pgfonlayer}
		\begin{pgfonlayer}{edgelayer}
			\draw [style=open] (10.center) to (11.center);
			\draw [style=open] (13.center) to (12.center);
			\draw [style=open] (16.center) to (15.center);
			\draw [style=open] (17.center) to (18.center);
			\draw [bend right=90, looseness=3.00] (18.center) to (10.center);
			\draw [bend right=90, looseness=3.25] (11.center) to (12.center);
			\draw [bend right=90, looseness=3.00] (13.center) to (16.center);
			\draw (22.center) to (15.center);
			\draw (19.center) to (17.center);
			\draw [style=open] (19.center) to (20.center);
			\draw [style=mydotsblack, bend left=90, looseness=1.50] (18.center) to (12.center);
			\draw [style=mydotsblack, bend left=90, looseness=1.75] (17.center) to (13.center);
			\draw [style=mydotsblack, bend left=90, looseness=1.75] (10.center) to (15.center);
			\draw [style=mydotsblack, bend left=90, looseness=1.75] (11.center) to (16.center);
			\draw [style=red] (33.center) to (32.center);
			\draw [style=green, in=-180, out=90, looseness=0.75] (32.center) to (38.center);
			\draw [style=violet, in=-180, out=30] (32.center) to (37.center);
			\draw [style=green, in=180, out=-30] (32.center) to (36.center);
			\draw [style=violet, in=180, out=-90, looseness=0.75] (32.center) to (35.center);
			\draw [style=violet, bend right=90, looseness=1.75] (35.center) to (37.center);
			\draw [style=green, bend right=90, looseness=1.75] (36.center) to (38.center);
		\end{pgfonlayer}
	\end{tikzpicture}
		
	\caption{Ribbon graph for the torus with one boundary component.
		The Dehn twist $d_1$ about the green simple closed curve and the Dehn twist $d_2$ about the violet simple closed curve generate the mapping class group  $\Map(\mathbb{T}_1^2)\cong B_3$.}
	\label{figtorus}
\end{figure}
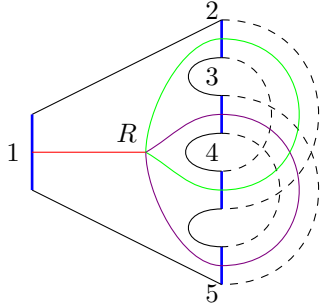

The mapping class group of $\mathbb{T}_1^2$ is given by
\begin{align}
	\Map(\mathbb{T}_1^2)\cong B_3 = \spr{d_1,d_2 \,|\, d_1d_2d_1=d_2 d_1 d_2} \ , 
\end{align}
where $B_3$ is the braid group on three strands $d_1$ and $d_2$ are the Dehn twists described in Figure~\ref{figtorus}, see e.g.\  \cite[Section~3.1.1]{farbmargalit} for more background.
A more common presentation 
in quantum topology is by means of the generators $T=d_1$ and $S=d_2d_1d_2$ (see e.g.~\cite[Section~3.1]{baki}) subject to the relation
\begin{align}
	(ST)^3=S^2 \ . 
\end{align}

Now Corollary~\ref{cordehnadapted} tells us:

\begin{corollary}\label{corT12}
	For a compact projective object $p\in \RV$, the representation 
	on $\block(\mathbb{T}_1^2;p)$
	is equivalent to the representation
	on $\int^{u,v \in \catT} \spr{p,u^*,v^*,u,v} $ in which 
	the generators $T$ and $S$ of $\Map(\mathbb{T}_1^2)$ from above act
	 by 
	\begin{align} T &\mapsto \gamma_2^\mu(d_1) \circ\gamma_1^\mu(d_1)^{-1} \ , \\
		S &\mapsto  \gamma_2^\mu(d_2) \circ\gamma_1^\mu(d_2)^{-1} \circ \gamma_2^\mu(d_1) \circ\gamma_1^\mu(d_1)^{-1}\circ \gamma_2^\mu(d_2) \circ\gamma_1^\mu(d_2)^{-1} \ . 
	\end{align}
	The same remains true for the punctured torus instead of the torus with one boundary, after replacing $\int^{u,v \in \catT} \spr{p,u^*,v^*,u,v} $ with $\int^{u,v \in \catT} \spr{u^*,v^*,u,v} $.
\end{corollary}

We may now deduce a relation to the elliptic double of Brochier-Jordan~\cite{brochierjordane}:

\begin{corollary}\label{corell}
	Through the pairing, we may see $\block(\mathbb{T}_1^2;-):\RV\to\Vect$ as an object in $\RV$ that, by abuse of notation, we denote also by
	$\block(\mathbb{T}_1^2)$.
	Then
	\begin{align}
		\block(\mathbb{T}_1^2) \cong \mathbb{F}^{\otimes 2}
	\end{align} with $\mathbb{F}=\int^{t\in \catT} t^* \otimes t$; the identification uses the braiding of $\RV$ in a non-canonical way.
	The induced $B_3$-action on $\mathbb{F}^{\otimes 2}$
	obtained via Corollary~\ref{corT12} generalizes the $B_3$-action on the elliptic double of Brochier-Jordan in the finite rigid case.
\end{corollary}

\begin{proof}We have
	\begin{align}
		\block^\mu(\mathbb{T}^2_1;p) \cong \int^{u,v \in \catT} \spr{p,u^*,v^*,u,v} \cong \int^{u,v \in \catT} \spr{p, u^* \otimes v^* \otimes u \otimes v  }
	\end{align} thanks to the compatibility between the pairing and the monoidal product.
	After braiding $v^*$ over $u$, we arrive at 
	\begin{align}
		\block^\mu(\mathbb{T}^2_1;p) &\cong \int^{u,v \in \catT} \spr{p, u^* \otimes u \otimes v^*\otimes v  }\\&\cong \spr{p, \int^{u,v \in \catT}u^* \otimes u \otimes v^*\otimes v}\\&\cong 
		\spr{p, \left(\int^{u \in \catT}u^* \otimes u\right) \otimes \left(\int^{v\in\catT} v^*\otimes v\right)}\\&\cong \spr{p,\mathbb{F}^{\otimes 2}} \ . 
	\end{align}
	The comparison result to the elliptic double, i.e.\ the moduli algebra for the torus with one boundary component, is a consequence of Corollary~\ref{corfiniterigidmodulialgebra}.
\end{proof}

\subsection{Description on generators}
In Corollary~\ref{corT12}, we have treated the torus with one boundary component.
In this subsection, we cover surfaces of genus $g\ge 2$ (we leave out the torus with more boundary components and genus zero surfaces with more than two boundary components).

Suppose that for $1\le i\le m$ we are given a Dehn twist $d_i$ about a simple closed curve that is part of a marking $\mu_i$. Then we abbreviate (the notation generalizes the one already in use in Corollary~\ref{cordehnadapted})
\begin{align}
	\psi_i &:= \psi_{\mu_i} \ , \\
	\gamma_j^i &:= \gamma_j^{\mu_i}(d_i) \quad \text{for}\quad j=1,2 \ .
\end{align}

\begin{theorem}\label{thmmcggen}
	Suppose that $\Sigma$ is a connected surface of genus $g\ge 2$, $n$ boundary components and $p$ punctures with $n+p\ge 1$.
	Then the $\Map(\Sigma)$-action on $\block(\Sigma;-)$ is determined on a system of generators for
	$\Map(\Sigma)$ given by Dehn twists $d_1,\dots,d_m$ about non-separating simple closed curves which are part of corresponding markings $\mu_1,\dots , \mu_n$
	by the fact that $d_i$ acts by
	\begin{align}
		\label{eqndi}	{d_i}_* = \psi_{i} \circ \gamma_2^{i} \circ {\gamma_1^{i}}^{-1} \circ \psi_{i}^{-1} \quad \text{for}\quad 1\le i\le m \ . 
	\end{align}
\end{theorem}

\begin{proof}The fact that a generating system of non-separating Dehn twists is available can be found e.g.~in \cite[Corollary 4.16]{farbmargalit}. Therefore, it is enough to prove that we can extend each of them to a marking and hence can conclude \eqref{eqndi} from Corollary~\ref{cordehnadapted}.
	
	Clearly, we have a marking $\mu$ for $\Sigma$ for which a simple closed curve $c$ that is part of $\mu$ is non-separating. This can be easily seen by writing $\Sigma$ as $4g$-gon
	 with identified edges,  plus some cut out disks
	 in the middle
	  to create the boundary circles.
	  Now choose a marking with one vertex, a circle around each cut out disk, and edges going through identified edges of the $4g$-gon. All its edges passing through identified sides are non-separating simple closed curves.
	
	Next suppose that $d_i$ is the Dehn twist about $c_i$.
	Then there is a mapping class $f_i$ mapping $c$ to $c_i$ \cite[Section~1.3]{farbmargalit}. Now $c_i$ is a part of $\mu_i:=f_i.\mu$. This proves that each curve in the system can be extended to a marking and finishes the proof.
\end{proof}

	\begin{question}\label{q4}
			Denote by $D \in \mathbb{N}_0 \cup \{\infty\}$ the order of the operator $d_*$ on $\block(A;-,-)$.
			In the finite rigid case from Section~\ref{secfiniterigid} in which $\RV$ is the ind completion of a modular category $\Vmod$,
			\begin{align} D = |\theta| \ , 
				\end{align}
			where $\theta$ is the balancing of $\Vmod$.
			Indeed, by the comparison from Corollary~\ref{corfiniterigid} we know that $D$ is the order of the balancing $\theta^{-1}\boxtimes \theta$ of $\overline{\Vmod}\boxtimes\Vmod$, which is $|\theta|$. 
			On closed surfaces, the order of the operator associated to a Dehn twist about a \emph{non-separating} simple closed curve is also $|\theta|$ \cite[Theorem 4.1]{mwdehn}. How can these results be generalized beyond the finite rigid case?
		\end{question}

\needspace{15\baselineskip}
\section{Finite rank Theorem and stability}
In this section, we prove two structural results on spaces of conformal blocks $\block$, namely a finiteness result
and also a connection to the Euler characteristic of a surface.

\subsection{Finiteness properties}
For any reasonable construction of spaces of conformal blocks, one should have a finiteness statement in the $C_2$-cofinite case, see e.g.~\cite[Proposition 5.1.1]{DGT2} for such a result in a different context.
We can prove such a statement for the spaces of conformal blocks $\block$.
The proof will need several preparations:

\begin{lemma}\label{lemmafiltered}
	Let $Y:\cat{J}\to\Cat$ be a category-valued diagram such that all $Y(i)$ are cocomplete and $Y(f):Y(i)\to Y(j)$ for all morphisms $f:i\to j$ in $\cat{J}$ preserve filtered colimits. Then all structure maps $\pi_i : \operatorname{2-lim}\, Y\to Y(i)$
preserve filtered colimits, i.e.\ filtered colimits are computed factor-wise.	\end{lemma}

\begin{proof}
	The category $\operatorname{2-lim}\, Y$ is the category of sequences $(y_i)_{i\in\cat{J}}$ with $y_i\in Y(i)$ together with coherence isomorphisms
	$\varphi_f:Y(f)y_i\to y_j$ for all morphisms $f:i\to j$
	such that \begin{itemize}
		\item  $\varphi_{\id_i}$ for all $i\in \cat{J}$ is the canonical isomorphism $Y(\id_j) x_j\cong x_j$,
		\item $\varphi_{gf} = \varphi_g \circ X(g)\varphi_f\circ \alpha_{g,f}$ for composable morphisms $i\ra{f}j\ra{g}\ell$ with the coherence isomorphism $\alpha_{g,f} : Y(gf)\cong Y(g)Y(f)$.
		\end{itemize}
	Now let $\cat{U}$ be a filtered category and $Z:\cat{U}\to \operatorname{2-lim}\, Y$ a diagram with colimit $\colim Z$. 
	The diagram $Z$ consists of a coherent family of diagrams $Z_i : \cat{U}\to Y(i)$. 
	This means that the family $(\colimsub{\cat{U}} Z_i)_{i\in\cat{J}}$ together with the isomorphisms
	\begin{align}
		Y(f) \colimsub{\cat{U}} Z_i \cong \colimsub{\cat{U}} Y(f)Z_i \cong \colimsub{\cat{U}} Z_j
		\end{align} for all morphisms $f:i\to j$ in $\cat{J}$
	is coherent. This uses that the $Y(f)$ preserve filtered colimits.
	One can verify that this family is the filtered colimit of $Z:\cat{U}\to \operatorname{2-lim}\, Y$.
	With this description of filtered colimits in $Z:\cat{U}\to \operatorname{2-lim}\, Y$, we can deduce that the structure maps $\pi_i : \operatorname{2-lim}\, Y\to Y(i)$ preserve filtered colimits.
	\end{proof}

\begin{proposition}\label{propcompact}
	Let $X:\cat{J}\to\PrL$ be a diagram in $\PrL$ such that all $X(i)$ have enough compact objects and such that $X(f):X(i)\to X(j)$ for every morphism $f:i\to j$ in $\cat{J}$
	preserves compact objects.
	Suppose now that $(\varphi_i :X(i)\to T)_{i\in\cat{I}}$ is a homotopy coherent cone under $X$ such that $T$ has enough compact objects and all $\varphi_i$ preserve compact objects.
	Then the induced map $\operatorname{2-colim}\, X \to T$ preserves compact objects.
	\end{proposition}

\begin{proof}
	Denote by $U:\PrL\to\Cat$ the forgetful functor.
	By \cite[Proposition~2.1.11]{chirjf} the underlying category of $\operatorname{2-colim}\, X$ is the 2-limit in $\Cat$ over the categories $UX(i)$ and the arrows $X(f)^\text{R} : UX(j)\to UX(i)$, where $X(f)^\text{R}$ is the right adjoint to $X(f)$. Let us denote this $\Cat$-valued diagram by $Y:\cat{J}^\op\to \Cat$.
	Since the $X(f)$ preserve compact objects and since all $X(i)$ are compactly generated, the functors $X(f)^\text{R}$ preserve filtered colimits.
	The right adjoints $\varphi^\text{R}_i:UT\to  UX(i)$ also preserve filtered colimits.	From the description of filtered colimits in $\operatorname{2-lim}\, Y$ by Lemma~\ref{lemmafiltered}, we can deduce that the resulting map $UT\to \operatorname{2-lim}\, Y$ preserves filtered colimits.
	Its left adjoint $\operatorname{2-colim}\, X \to T$ is by  \cite[Proposition~2.1.11]{chirjf}
	the map induced by the map $\varphi$ in $\PrL$.
	The left adjoint of a functor preserving filtered colimits preserves compact objects
	(this does not need that $\operatorname{2-colim}\, X$ has enough compact objects).
	\end{proof}

\begin{lemma}\label{lemmafinite}
	Let $\cat{A}$ be a framed $E_2$-algebra in $\PrL$ obtained from a framed $E_2$-algebra in $\Prc$ by inclusion $\Prc\subset \PrL$, and
	let $F: \int_\Sigma \cat{A} \to \Vect$ be a cocontinuous functor such that the restriction along any map $\cat{A}^{\boxtimes n}\to\int_\Sigma \cat{A}$ induced by an embedding of $n$ disks into $\Sigma$ sends compact objects in $\cat{A}^{\boxtimes n}$ to finite-dimensional vector spaces.
	Then $\dim\, F(X) < \infty$ for any compact object $X\in\int_\Sigma \cat{A}$.
	\end{lemma}

\begin{proof}
	By definition $\int_\Sigma \cat{A}$ is a 2-colimit in $\PrL$
	of a diagram of categories with enough compact
	objects such that all arrows of the diagram preserve compact objects (because the  framed $E_2$-algebra underlying $\cat{A}$ lives in $\Prc$). 
	
	The fact that each restriction
	$\cat{A}^{\boxtimes n}\to\int_\Sigma \cat{A}\ra{F}\Vect$ sends compact objects to finite-dimensional vector spaces means exactly that these restrictions preserve compact objects. 
	Now Proposition~\ref{propcompact} implies that $F$ preserves compact objects, which means that it sends compact objects to finite-dimensional vector spaces.
	\end{proof}

\begin{theorem}\label{thmfinite}
	If $V$ is $C_2$-cofinite, all the block functors
	\begin{align}
		\block(\Sigma;-): \RV^{\boxtimes n} \boxtimes \left(\int_{\mathbb{S}^1} \RV\right)^{\boxtimes m} \to \Vect  \label{eqntheblockfunctors}
	\end{align}  of the modular functor with singularities associated to $V$ by
	Theorem~\ref{thmmain} satisfy
	\begin{align} \dim\, \block(\Sigma;X_1,\dots,X_n,Y_1,\dots,Y_m)<\infty
	\end{align}
	if all $X_1,\dots,X_n \in \RV$ 
	and all $Y_1,\dots,Y_m \in \int_{\mathbb{S}^1} \RV$ are compact.
\end{theorem}

If the $X_1,\dots,X_n$ are not compact, it does not make sense to ask finite-dimensionality because whenever $\block(\Sigma;X_1,\dots,X_n)\neq 0$,
the space $\block\left(\Sigma;\bigoplus_{i=0}^\infty X_1,\dots,X_n\right)\cong \bigoplus_{i=0}^\infty \block\left(\Sigma; X_1,\dots,X_n\right)$ is not finite-dimensional.

\begin{proof}
	By the discussion in Section~\ref{secC2} in the $C_2$-cofinite case,
	$\RV$ is the ind completion of a ribbon Grothendieck-Verdier category $\Vmod$ in $\Rexf$.
	Hence, the underlying framed $E_2$-algebra of $\RV$ lives in $\Prc$.
	
	Next we observe that the $\PrL$-valued modular $\O$-algebra $\Vo$ from Proposition~\ref{thmmcgaction} can be obtained as the ind completion of the $\Rexf$-valued modular $\O$-algebra $\cat{A}_V$ obtained as the modular extensions of $\Vmod$
	(indeed, $\cat{A}_V$ is again $\Rexf$-valued, see \cite[Proposition 7.1]{cyclic}): \begin{align}\label{eqnVoAV} \Vo\simeq \catf{ind} \, \cat{A}_V\end{align} as  $\PrL$-valued modular $\O$-algebras, where $\catf{ind}$ is applied operation-wise.
	For the proof of \eqref{eqnVoAV}, one just restricts to disks with marked intervals and uses Theorem~\ref{thmclassification}.
	
	This proves that  the $\PrL$-valued open modular functor $\Vo$ 
	lives actually in the subcategory $\Prc \subset \PrL$, see also Remark~\ref{remarkprc}, so the same applies to $\block$, when defined on tensor powers of $\RV$.
As a consequence, the maps $
\RV^{\boxtimes n} \boxtimes \RV^{\boxtimes \ell} \to \Vect$ from \eqref{eqnmapsinducingfh}, for each embedding of $\ell$ intervals into the closed boundary,
that by construction induce \eqref{eqntheblockfunctors}
 send compact objects to finite-dimensional vector spaces.
Now the assertion follows from Lemma~\ref{lemmafinite}. 
\end{proof}

\begin{remark}
	Without the assumption of $C_2$-cofiniteness,
	Theorem~\ref{thmfinite} generally fails, see \cite[Example~3.6]{yeralopen}.
\end{remark}

	\subsection{Stability}
The blocks $\block$ have another property that, at first, is slightly surprising, but very easy to prove:

\begin{proposition}[Stability, part I]\label{propstability}
	Suppose that $\Sigma$ and $\Sigma'$ are connected  surfaces without parametrized boundary components, but a non-zero number of punctures.  
	There is a non-canonical isomorphism
	$\block(\Sigma)\cong \block(\Sigma')$ if the Euler characteristics of $\Sigma$ and $\Sigma'$ agree.
\end{proposition}

\begin{proof}
	If $\Sigma$ is connected, it admits
	a marking $\mu=(R,g)$ by a ribbon graph $R$ with only one vertex
	and $2\ell$ half edges with $\ell = 2g+p$ if $\Sigma$ has genus $g$ and $p\ge 1$ punctures; phrased differently, $\ell = 3-\chi(\Sigma)$ with the Euler characteristic $\chi(\Sigma)$ of $\Sigma$. 
	Theorem~\ref{thmmain} tells us that $\block(\Sigma)$ is a coend over $\ell$ dummy variables of a functor $F_R:\RV^{\boxtimes (2 \ell)} \to \Vect$, namely the one from Corollary~\ref{corspr}, with the cyclic ordering of the arguments prescribed by $R$.
	But since $\RV$ is braided, $F_R$ depends, up to \emph{non-canonical} isomorphism only on $\ell$, and therefore only on $\chi(\Sigma)$. 
	The same is now true for $\block(\Sigma)$.
\end{proof}

Proposition~\ref{propstability} is consistent with the idea that $\block$ describes spaces of conformal blocks \emph{after combining left and right movers}:
In this interpretation, $\block$ on a surface with genus $g$ and $p\ge 1$ punctures should be the space of \emph{chiral} conformal blocks on the double $\bar{\Sigma}\cup_{\partial \Sigma} \Sigma$, where $\Sigma$ has genus $g$ and $p$ boundary components. 
But $\bar{\Sigma}\cup_{\partial \Sigma} \Sigma$ is a closed surface of genus $2g+p+1=3-\chi(\Sigma)$, and hence
\begin{align}
	\bar{\Sigma}\cup_{\partial \Sigma} \Sigma \cong \bar{\Sigma'}\cup_{\partial \Sigma'} \Sigma' \quad \Longleftrightarrow \quad \chi(\Sigma)=\chi(\Sigma') \ . 
	\end{align}
So if we imagine $\block$ as the evaluation of a hypothetical chiral theory on the double, then $\block$ should produce non-canonically isomorphic spaces on surfaces of the same Euler characteristic.

Similarly, we prove:

\begin{corollary}[Stability, part II]\label{corstability}
	Suppose that $\Sigma$ and $\Sigma'$ are the same surface with at least one boundary component per connected component, but with possibly different configurations of $n$ marked intervals in the boundary.
	Then there is a non-canonical
	isomorphism $\block(\Sigma)\cong \block(\Sigma')$ of functors $\RV^{\boxtimes n}\to\Vect$.
	\end{corollary}
   		
\begin{proof}
	As in the proof of Proposition~\ref{propstability},
	$\block(\Sigma)$ and $\block(\Sigma')$ are both coends over a priori different functors $F, F':\RV^{\boxtimes n}\boxtimes \RV^{\boxtimes m}\to\Vect$ for some appropriately chosen $m\ge 0$.
	But these two functors are actually non-canonically isomorphic  through the braiding of $\RV$, i.e.\ $F\cong F'$. This implies $\block(\Sigma)\cong \block(\Sigma')$.
	\end{proof}

\section{Class functions, traces and the Higman map}
This section is devoted to the study of certain algebraic aspects of the construction introduced in the preceding 
sections. In particular, we will calculate the pairing on the factorization homology of a cyclic associative algebra that we obtain by evaluation on the cylinder.
The results will turn out closely related to the question how many cyclic subcategories we may actually have on a category of modules over a vertex operator algebra.
This will lead to a far-reaching finiteness statement for rational conformal field theories.

	\subsection{The annulus and bicategorical traces}
	Traces can be defined in higher-categorical contexts, see e.g.~\cite{pontoshulman}.
	For symmetric monoidal bicategories, we will use the following notion:
	
\begin{definition}
	Let $X\in\cat{S}$ be a dualizable object in a symmetric monoidal bicategory.
	Then we define the \emph{trace} of a 1-morphism $F:X\to X$ as the endomorphism $\tr F : I \to I$
	\begin{align}
		I \ra{\text{coevaluation}} X^\vee \boxtimes X \ra{\id_{X^\vee} \boxtimes F} X^\vee \boxtimes X \ra{\text{evaluation}} I
		\end{align} of the monoidal unit $I\in\cat{S}$. The trace of the identity of $X$ is also referred to as the \emph{dimension} of $X$.
	\end{definition}
We note that there is a canonical isomorphism $\tr F^\vee \cong \tr F$.
If $\cat{S}=\PrL$, the trace of a 1-morphism can be seen as a vector space.

\begin{remark}\label{remarktrace}
If $X$ has a symmetric non-degenerate pairing $\kappa : X \boxtimes X \to I$ with coevaluation $\Delta : I \to X\boxtimes X$, then $\tr F$ is canonically isomorphic to
\begin{align}
	I \ra{\Delta} X \boxtimes X \ra{\id_X \boxtimes F} X \boxtimes X \ra{\kappa } I \ , 
	\end{align}
and it does not matter whether we place $F$ on the other leg in the sense that we replace the middle morphism with $F\boxtimes \id_X$.
 In particular, for any cyclic associative algebra $\cat{A}$ in $\cat{S}$, the trace $\tr F$ can be defined for any 1-morphism $F:\cat{A}\to\cat{A}$.
 \end{remark}

\begin{remark}
	Of course, if we choose for $\cat{S}$ the symmetric monoidal category of finite-dimensional vector spaces artificially seen as symmetric monoidal bicategory, we recover the usual traces from linear algebra.
	\end{remark}
 
 \begin{proposition}\label{propannulustrace}
 	For any cyclic associative algebra $\cat{A}$, the open modular functor $\Ao$
 	satisfies on the annulus equipped with one marked interval on each boundary circle
 	\begin{align}
 		\Ao(\mathbb{S}^1 \times [0,1]; X, Y) \cong \tr (\ell_X r_Y) \cong \tr (\ell_Y r_X) 
 		\end{align}
 	for all 1-morphisms $X,Y : I \to \cat{A}$ and their multiplication maps 
 	\begin{itemize}
 		\item $\ell_X : \cat{A} \ra{X\boxtimes-} \cat{A}\boxtimes\cat{A}\ra{\otimes}\cat{A}$ from the left,
 		\item  and  $r_Y : \cat{A} \ra{-\boxtimes Y} \cat{A}\boxtimes\cat{A}\ra{\otimes}\cat{A}$ from the right.
 		\end{itemize}
 	\end{proposition}

\begin{proof}
	  With excision for $\Ao$, we obtain
	\begin{align}
		\Ao(\mathbb{S}^1 \times [0,1]; X, Y) \cong \spr{\Delta',X,\Delta'', Y} \ , \label{eqnannulusexcision}
		\end{align}
	with the total arity four operation $\spr{-,-,-,-}:\cat{A}^{\boxtimes 4}\to I$, see the description of cyclic associative algebras in \cite[Theorem 4.2]{cyclic} and the construction of $\Ao$ in \cite{envas}. 
	Thanks to the invariance of the pairing up to coherent isomorphism,
	$ \spr{\Delta',X,\Delta'', Y}\cong \kappa(\Delta',X\otimes \Delta'' \otimes Y)$.
	Now $	\Ao(\mathbb{S}^1 \times [0,1]; X, Y) \cong \tr (\ell_X r_Y)$
	follows from Remark~\ref{remarktrace}.
	
	Finally, we observe also that
	$\spr{\Delta',X,\Delta'', Y}\cong \spr{\Delta'',Y,\Delta',X}\cong \spr{\Delta',Y,\Delta'',X}\cong \tr (\ell_Y r_X) $ by the cyclicity of the pairing and the symmetry $\Delta'\boxtimes\Delta''\cong\Delta''\boxtimes\Delta'$ of the copairing.
	\end{proof}

\subsection{The Higman map}
In the rest of the section, we will always assume that $\cat{S}$ is a cocomplete symmetric monoidal bicategory whose monoidal product commutes with colimits and that has inner homs. Clearly, $\cat{S}=\PrL$ is an example.
For an algebra $\cat{A}$ in $\cat{S}$, we may then define its \emph{(Drinfeld) center}
as
\begin{align}
	Z(\cat{A}) := \END_{\cat{A}^\text{e}}(\cat{A}) \ , 
	\end{align}
where $\cat{A}^\text{e}=\cat{A}^\text{opp} \boxtimes \cat{A}$ is the enveloping algebra (its modules are $\cat{A}$-bimodules),
 and $\cat{A}$ carries its standard bimodule structure, see \cite[Proposition 3.6]{bjss} for how this recovers existing definitions. 

\begin{theorem}\label{thmhigman}
	For any cyclic associative algebra $\cat{A} \in \cat{S}$ in a cocomplete symmetric monoidal bicategory whose monoidal product commutes with colimits and that has inner homs,
	the map \begin{align}\Delta' \otimes - \otimes \Delta''  :  \cat{A} \to \cat{A} \label{eqnthemapforH}
		\end{align}
	induces  a map
	\begin{align} H: \int_{\mathbb{S}^1}\cat{A}\to Z(\cat{A})  
		\end{align} that we call the Higman map.
	\end{theorem}

We justify the name in Section~\ref{sechigman}.

\begin{proof}
	Via the modular extension, the non-degenerate symmetric pairing
	 $\kappa : \cat{A}\boxtimes\cat{A} \to I$ factors through factorization homology
	\cite[Section~5]{envas}, see also Corollary~\ref{corpairingfh},
	and gives us
	the symmetric pairing $\beta :\int_{\mathbb{S}^1} \cat{A} \boxtimes \int_{\mathbb{S}^1} \cat{A}\to \Vect$
	with
	\begin{align}
		\beta(LX,LY)\cong \Ao(\mathbb{S}^1\times [0,1];X,Y)
		\end{align} for $X,Y\in\cat{A}$.
	
	Now we adjoin one of the $\int_{\mathbb{S}^1} \cat{A}$ to the right and use
	\begin{align}
		\HOM \left(\int_{\mathbb{S}^1} \cat{A},I\right)\simeq Z(\cat{A}) \label{eqnequivZ}
		\end{align}
	from \cite[Section~3.5]{yeralopen} to turn $\beta$ into a map
	$H:\int_{\mathbb{S}^1}\cat{A}\to Z(\cat{A}) $.
	By construction this means for the forgetful functor $U:Z(\cat{A})\to\cat{A}$ and $X\in\cat{A}$
	\begin{align}\label{eqnUHL}
		UHLX\cong \Ao(\mathbb{S}^1 \times [0,1];X,\Delta')\otimes\Delta'' \stackrel{\substack{\eqref{eqnannulusexcision} \text{\ and} \\ \text{zigzag identities}}}{\cong} \Delta' \otimes X \otimes \Delta'' \ . 
		\end{align}
	Therefore, it is clear that this map $\cat{A}\to\cat{A}$ descends to factorization homology and factors through the center and hence gives rise to the map $H$.
	\end{proof}

\begin{corollary}
	The pairing $\beta :\int_{\mathbb{S}^1} \cat{A} \boxtimes \int_{\mathbb{S}^1} \cat{A}\to I$
	is non-degenerate if and only if the Higman map is an equivalence.
	Moreover,
	there are canonical isomorphisms
	\begin{align}
		\beta(LX,	LY)\cong \kappa(X,UHLY)\cong \tr(\ell_Xr_Y ) \cong \tr (\ell_Y r_X)  \label{eqnHcomparison}
	\end{align} for all 1-morphisms $X,Y : I \to \cat{A}$ and the canonical map $L:\cat{A}\to\int_{\mathbb{S}^1}\cat{A}$.
	\end{corollary}

\begin{proof}
The Higman map	$H$ is the map induced by the pairing $\beta :\int_{\mathbb{S}^1} \cat{A} \boxtimes \int_{\mathbb{S}^1} \cat{A}\to I$
	and the equivalence $\HOM \left(\int_{\mathbb{S}^1} \cat{A},I\right)\simeq Z(\cat{A}) $ from \eqref{eqnequivZ}.
Therefore, $\beta$ is non-degenerate if and only if $H$ is an equivalence.

For \eqref{eqnHcomparison}, 
note that 	\begin{align}
	\beta(LX,LY)\cong \Ao(\mathbb{S}^1\times [0,1];X,Y)\cong \tr (\ell_X r_Y) \cong \tr (\ell_Y r_X) \label{eqnbetaLtrace}
\end{align} for $X,Y\in\cat{A}$, where we have used Proposition~\ref{propannulustrace}.
Moreover, with \eqref{eqnUHL}, 
\begin{align}\kappa(X,UHLY)\cong \kappa(X,\Delta' \otimes Y \otimes \Delta'')\cong \Ao(\mathbb{S}^1 \times [0,1];X,Y) \ . 
\end{align}
\end{proof}

\begin{remark}\label{remhandleelemment}
		We refer to
	\begin{align}
		h := HLu:I\to Z(\cat{A})
	\end{align} for the unit $u:I\to \cat{A}$
	as the \emph{handle element} of $\cat{A}$.  This terminology is chosen in generalization of \cite[page~128]{kocktft}, but it is also self-explanatory:
	The value of $\Ao$ on an annulus with one marked interval (a `handle')
	 yields a map $\cat{A}\to I$ that after dualization yields $h$, see also \cite[Definition~5.4]{tracesw}.
	 Moreover, it is well-known that the evaluation of $\Ao$ on the annulus with two marked intervals on the \emph{same} boundary component, one seen as incoming and one as outgoing, yields an $\cat{A}$-bimodule map $\chi:\cat{A}\to\cat{A}$ that is given by multiplication with $Uh$:
	 \begin{align}
	 	\chi \cong Uh\otimes- \ . 
	 	\end{align}
 	Suppose that $\cat{A}$ is actually braided, then Corollary~\ref{corstability} tells us that $\chi$ is isomorphic to \eqref{eqnthemapforH}.
	\end{remark}

\begin{corollary}\label{corhigmanF}
	For $\RV$, the Higman map $H: \int_{\mathbb{S}^1}\RV\to Z(\RV)$ is induced by the map
	\begin{align} \RV \to \RV \ , \quad X \mapsto \mathbb{F} \otimes X \quad \text{with}\quad \mathbb{F} = \int^{t\in \catT} t^* \otimes t \ . \label{eqncentralmonad}
		\end{align}
	\end{corollary}

\begin{proof}
	This is a direct consequence of Remark~\ref{remhandleelemment}.
	\end{proof}

\begin{remark}
	The fact that  the Higman map is an equivalence in the finite rigid case (because $\beta$ is non-degenerate, see Remark~\ref{rembetanondeg}) can also be deduced directly from Corollary~\ref{corhigmanF} after observing that \eqref{eqncentralmonad} is the central monad and \cite[Theorem 4.4 and Theorem 5.9]{sn}.
	\end{remark}

\subsection{The 1-categorical situation}\label{sechigman}
	Theorem~\ref{thmhigman} applies in particular to the symmetric monoidal 1-category of vector spaces. In order to unpack this, let us set up some
	notation regarding finite-dimensional algebras, see e.g.\ \cite{ars,egno,grabowski} for more background on this standard topic.
	We assume that we work over an algebraically closed field of characteristic zero:
	Suppose that $A$ is a finite-dimensional algebra with non-isomorphic indecomposable projective modules  $P_1,\dots,P_n$. In other words, $A=\End_A(\bigoplus_{i=1}^n P_i)$, and the identities $\id_{P_i}$ are the primitive idempotents $e_i\in A$ with $1=e_1+\dots+e_n$.
	To the projective object $P_i$, we can associative a simple object $S_i$ arising as a quotient of $P_i$, and the $S_1,\dots,S_n$ are representatives for the isomorphism classes of simple $A$-modules. Each simple $A$-module $S_i$ gives rise to a linear form $\chi_i:A / [A,A] \to k$, the so-called \emph{character}, namely the action
	map $A \to \Hom_A(S_i,S_i	)\cong k$ which descends to the cocenter.
	Over a field of characteristic zero, the $\chi_i$ are linearly independent as linear forms on $A / [A,A]$. 
	Indeed,  
	it is well-known that $(A / [A,A])^* \to k^n$ sending $\alpha : A / [A,A]\to k$ to $(\alpha(e_1),\dots,\alpha(e_n))$ sends the spans of the $\chi_i$ to $k^n$ because $\chi_i(e_j)=\delta_{ij} \dim_k S_i$ (and $\dim_k S_i\neq 0$ in $k$ because $k$ has characteristic zero). 
	This follows since $e_i$ acts by the identity on $P_i$ and hence $S_i$, but as zero operator on all $P_j$ and hence $S_j$ for $i\neq j$. 
	From this, we deduce also that the classes $\bar{e}_i$ of the $e_i$ in $A / [A,A]$ are linearly independent.

	A cyclic associative algebra in the category of vector spaces, seen as symmetric monoidal bicategory, is a symmetric Frobenius algebra.
	In this case $\int_{\mathbb{S}^1} A = A / [A,A]$,
	and $Z(A)$ is the usual center.
	Then $H:A/[A,A] \to Z(A)$ is given by $H\bar a = \sum_{i=1}^\ell b^i ab_i$ for dual basis $(b_1,\dots,b_\ell)$ and $(b^1,\dots,b^\ell)$ satisfying $\kappa(b_i,b^j)=\delta_i^j$, see \cite[Proposition~3.13]{broue} and also \cite{lzzhigman}; its image is the so-called \emph{projective center} of $A$, also called \emph{Higman ideal} $H(A)$. If it agrees with all of $Z(A)$, the map $H$ is an isomorphism because $A/[A,A]$ and $Z(A)$ have the same dimension thanks to $(A/[A,A])^*\cong Z(A)$ (because $A$ is a symmetric Frobenius algebra).
	We then have the following standard result:
	
	\begin{proposition} \label{propsymfrob} For a symmetric Frobenius algebra $A$ over an algebraically closed field of characteristic zero with $n$ non-isomorphic simple modules, the following are equivalent:
	\begin{pnum}
		\item     The Higman map $H:A/[A,A]\to Z(A)$ is an isomorphism.    \label{itemH}
			\item		$A$ is separable.		\label{itemsep}
		\item		$A$ is semisimple.		\label{itemssi}
		\item $\dim\, A / [A,A]=n$.  \label{itemcocenter}
		\item $\dim\, Z(A) =n$.	\label{itemcenter}
		\end{pnum}
	\end{proposition}

\begin{proof}
	All of these facts are well-known; let us give the pointers to the literature:
	If $H$ is an isomorphism, $1=\sum_{i=1}^\ell b^ivb_i$ for some $v\in A$. But then $p=\sum_{i=1}^\ell b^i \otimes vb_i$ is a separability idempotent. A fast way to verify this is using the open topological field theory associated to $A$ \cite{laudapfeiffer}, see also the discussion in \cite[Section~2.4]{lpstatesum} and \cite{aguiar}. This proves  \ref{itemH} $\Rightarrow$ \ref{itemsep}, and of course \ref{itemsep} $\Rightarrow$ \ref{itemssi}.

	For \ref{itemssi} $\Rightarrow$ \ref{itemH}, we use that in the semisimple case the Cartan matrix has full rank $n$ (it is actually the identity matrix of course). By \cite[Corollary~2.7]{lzzhigman} this implies that the image of $H$ is $n$-dimensional.
But in the semisimple case $A / [A,A]$ and $Z(A)$ are $n$-dimensional, so $H$ is an isomorphism.
So far, we know \ref{itemH} $\Leftrightarrow $ \ref{itemssi} $\Leftrightarrow $ \ref{itemssi}.

	Moreover, \ref{itemcocenter} $\Leftrightarrow$ \ref{itemcenter} since $A / [A,A]$ and $Z(A)$ are dual. Furthermore, \ref{itemssi} clearly implies \ref{itemcocenter} because in the semisimple case the category of finite-dimensional $A$-modules is just $n$ copies of the category of finite-dimensional vector spaces.

	For the proof of \ref{itemcocenter} $\Rightarrow$ \ref{itemssi}, denote by $J$ the Jacobson radical of $A$ and by $\catf{soc}\, A=\{a\in A \, |\, Ja=0\}$ the socle of $A$.
	It suffices to show that the unit is in $\catf{soc}\, A$ because then $J=0$, and $A$ is semisimple. In particular, $Z(A) \subset \catf{soc}\, A$ is enough for semisimplicity.
	The fact that we have this inclusion if we assume \ref{itemcocenter} is a direct consequence of results of Kühlshammer~\cite{kuehlshammer} (some of which, according to the author, were already known): 
	Since the algebra $A/J$ is semisimple with $n$ non-isomorphic simple modules,
	the assumption that $A/[A,A]$ is $n$-dimensional implies that the surjection $A / [A,A] \to A / ([A,A]+J)$ must actually also be injective for dimension reasons because $\dim (A / ([A,A]+J)) = \dim \left( (A/J) / [A/J,A/J]\right)=n$, see the comments after \cite[after no.~(5)]{kuehlshammer}. But this is only possible if $J \subset [A,A]$.
	For the complements with respect to the Frobenius pairing, we have
	\begin{align}
		J^\perp = \catf{soc}\, A=\{a\in A \, |\, Ja=0\} \ , \qquad [A,A]^\perp = Z(A) \qquad \text{thanks to \cite[after no.~(32), also (35)]{kuehlshammer}}.
		\end{align}
	 This implies $Z(A)\subset \catf{soc}\, A$ and finishes the proof.
	\end{proof}

\subsection{An attempt at a structural understanding of cyclic categories of modules over vertex operator algebras}
Generally, the (open) modular functors that we built in this article do not come from finitely semisimple rigid categories.
In the literature, this is still the best understood case, and it can be valuable to know when we are actually in this special situation.
The next result tells us that we can detect whether the categories are  semisimple and rigid by looking at the value on the annulus and the disk.

\begin{theorem}\label{thmssirigid}
	Suppose that $\cat{A}$ is a cyclic framed $E_2$-algebra in $\PrL$ over an algebraically closed field of characteristic zero.
	Let $\cat{A}$ be finite, i.e.\ it is 
	the ind completion of cyclic framed $E_2$-algebra in finite categories.
	\begin{pnum}
		\item	Assume that $\cat{A}$ has $n$ compact simple objects.\label{pdimhh}
		Then \begin{align} \dim\,\Ao(\mathbb{S}^1\times [0,1])\ge n \ . \label{eqndimn}
		\end{align}
	In particular, the coend $\mathbb{F}=\int^{t\in \catT} t^* \otimes t$ is non-zero for $n\ge 1$.
		\item	Assume that there is a non-degenerate trace, i.e.\ a linear form 
		\begin{align} \varepsilon : \Ao(\mathbb{S}^1 \times [0,1])\to k \label{eqntracemod}
		\end{align} such that the maps
		$\cat{A}(t,u)\otimes \cat{A}(u,t)\to \Ao(\mathbb{S}^1 \times [0,1]) \ra{\varepsilon} k$ are non-degenerate.
		If the equality in \eqref{eqndimn} holds, then $\cat{A}$ is semisimple.\label{nondegtrace}
		
		\item If $\cat{A}$ has\label{pssirigid} \begin{itemize}
			\item a non-degenerate trace as in \ref{nondegtrace}, 
			\item a simple unit and $\Ao(\mathbb{D}^2)\neq 0$, 
			\item and if equality holds in \eqref{eqndimn}, 
			\end{itemize} then $\cat{A}$ is semisimple and rigid, i.e.\ it is the ind completion of a ribbon fusion category.
		
		\item In the finite rigid case, i.e.\ if $\cat{A}$ is the ind completion of a finite ribbon category, 
		a trace as in \ref{nondegtrace} exists if and only if the finite ribbon category is unimodular. 
		In that case, the trace \eqref{eqndimn} can always be chosen such that it is invariant under the $\mathbb{Z}_2$-action on $\Ao(\mathbb{S}^1 \times [0,1])$.\label{ptrace}
	\end{pnum}
\end{theorem}

\begin{proof}
	By assumption
	and Proposition~\ref{propfinite}
	 $\cat{A}$ is the ind completion of a ribbon Grothendieck-Verdier category $\cat{C}$.
	The latter is linearly equivalent to the category of finite-dimensional modules over a finite-dimensional algebra $A$ with $n$ isomorphism classes of simple objects. The discussion before Proposition~\ref{propsymfrob} immediately proves \ref{pdimhh}.
	Thanks to $\Ao(\mathbb{S}^1\times [0,1])\cong \kappa(I,\mathbb{F})$, we conclude that $\mathbb{F}$ cannot be zero.

	The assumption in \ref{nondegtrace} says that $A$ is a symmetric Frobenius algebra. Now \ref{nondegtrace} follows from Proposition~\ref{propsymfrob}.
	
	For \ref{pssirigid}, note that $\Ao(\mathbb{D}^2)\neq 0$ says $\cat{C}(K,I)\neq 0$ for the dualizing object $K$ of $\cat{C}$ \cite[eq. (4.1)]{envas}. By assumption $I$ is simple,
	 and hence so is its image $K$ under the Grothendieck-Verdier duality.
	This implies $K\cong I$ via Schur's Lemma. In other words, $\cat{C}$ is an $r$-category in the sense of \cite{bd}.
	From \ref{nondegtrace}, we conclude that $\cat{C}$ is a finitely semisimple category. 
	By \cite[Corollary 1.3]{etingofpenneys}
	this implies rigidity of the monoidal product.
	Since $K\cong I$, the Grothendieck-Verdier duality agrees with the rigid duality \cite[Proposition 2.3]{bd}.
	In other words, $\cat{C}$ is a ribbon fusion category.
	
	For \ref{ptrace}, recall that in the finite ribbon case, a trace as in \ref{nondegtrace}
	 exists if and only if 
	the finite ribbon category is unimodular, see \cite{mtrace} and the references therein and also \cite{shibatashimizu}.
	In fact, these papers tell us that we actually find a two-sided modified trace that, in particular, will have the property that a morphism $f$ and its dual $f^\vee$ will have the same trace \cite[Lemma 3 (b)]{mtrace3}. But $f \mapsto f^\vee$ is exactly the mapping class group action by $\mathbb{Z}_2$ on $\Ao(\mathbb{S}^1 \times [0,1])\cong HH_0(\cat{A})$ in the rigid case \cite[Example 3.6]{yeralopen}.
	\end{proof}

\begin{question}\label{qtrace}
	In the finite rigid case, \eqref{eqntracemod} can be chosen as a modified trace as discussed in the proof of Theorem~\ref{thmssirigid}.
	Beyond the rigid case, this requirement does not make sense because the notion of a modified trace needs rigidity.
	However, the notion of a trace invariant under the action of the mapping class group  $\mathbb{Z}_2$
	of the annulus without marked intervals still makes sense and is always fulfilled for modified traces. 
	This suggests that $\mathbb{Z}_2$-invariant traces, a notion that always 
	can be defined
	 for the compact projective objects of $\RV$, could be a reasonable replacement of modified traces beyond the finite rigid case.
	Which vertex operator algebras have such a trace?
	\end{question}

\begin{definition}
	Let $V$ be a $C_2$-cofinite vertex operator algebra and $\Vmod$ the category of strongly graded $V$-modules.
	As in \cite[Definition 1.8.13]{egno}, denote by $K_0(V)$ the free abelian group spanned by the isomorphism classes $[P_1],\dots,[P_n]$ of indecomposable projective objects of $\Vmod$.
	We define \emph{fusion data} for $V$ as the structure of a not necessarily unital associative multiplication $[P_i] \star [P_j] = \sum_{\ell=1}^n c_{ij}^\ell [P_\ell]$ with coefficients $c_{ij}^\ell \in \mathbb{N}_0$. 
	\end{definition}

\begin{theorem}\label{corfiniteness}
	Let $V$ be a $C_2$-cofinite vertex operator algebra and $\Vmod$ the category of strongly graded $V$-modules having $n$ indecomposable projective objects.
	Equip $\Vmod$ with the non-degenerate symmetric  pairing induced by taking the contragredient $X \mapsto X'$ of 
	 $V$-modules, with the symmetry arising from the isomorphism $X''\cong X$.
	For a given fusion product $\star$ with coefficients $c_{ij}^\ell $on $K_0(V)$, denote by $F_V(\star)$ the equivalence classes of open modular functors 
	 $\cat{A}$ inducing $\star$ on $K_0(V)$ via $c_{ij}^\ell = \dim\, \cat{A}(\mathbb{D}^2_3;P_i,P_j,P_\ell')$, where $\mathbb{D}_3^2$ is the disk with three marked intervals.
	In other words, $F_V(\star)$ is the fiber of $\star$ for the restriction to disks.
	\begin{pnum}
		\item In the finite rigid case, there exists fusion data $\star$ such that the fiber $F_V(\star)$ is non-empty.\label{pnonempty} 
		\item Assume that $V$ is simple and $V'\cong V$. Consider the subfiber $F_V'(\star)$ of those $\cat{A}$ such that the value of the annulus has a non-degenerate trace and has dimension $n$.
		Then $F'_V(\star)$ is finite, i.e.\ only finitely many open modular functors extend the fusion data. \label{pfinitefiber}
		\end{pnum}
	\end{theorem}

\begin{proof}
	Statement~\ref{pnonempty} follows from Corollary~\ref{corvmod}.
For \ref{pfinitefiber}, use Theorem~\ref{thmssirigid}~\ref{pssirigid} to conclude that the category obtained by evaluation of $\cat{A}\in F'(\star)$ 
is a finitely semisimple pivotal Grothendieck-Verdier category that is actually rigid, with the Grothendieck-Verdier duality coinciding with the rigid duality. In other words, it is a fusion category. But up to equivalence, only finitely many of those recover a given fusion ring \cite[Theorem 9.1.4]{egno}. On the other hand, the fusion category plus the pivotal structure (which is fixed by construction to be $X''\cong X$) determine already the open modular functor up to equivalence (Theorem~\ref{thmclassification}).
\end{proof}

There is the following immediate consequence: Let $V$ be strongly rational, i.e.\ $\Vmod$ is finitely semisimple and $V\cong V'$.
If we build an open modular functor from $\Vmod$, it has to turn $\Vmod$ into a pivotal fusion category by Theorem~\ref{thmssirigid}~\ref{pssirigid}.
This could be the structure from \cite{huang} or the geometric one from \cite{damioliniwoike} induced by the algebro-geometric blocks from \cite{frenkelbenzvi}, or another one.
From this pivotal fusion category, one can construct an open-closed modular functor without singularities \cite{sn}, in a unique way such that the boundary induces the bulk in the sense of Remark~\ref{rembulkboundary}. For $N\in\mathbb{N}$, denote by $F_V^N$ all such open-closed modular functor extensions for $\Vmod$, regardless of the structure that we put on it, but subject to the requirement that the coefficients of the underlying fusion ring all satisfy $c_{ij}^\ell \le N$. Now Corollary~\ref{corfiniteness} \ref{pfinitefiber} tells us:

\begin{corollary}\label{corfiltration}
	In the strongly rational case, $F_V^1 \subset F_V^2 \subset \dots $ is a filtration of the extensions of $\Vmod$, without a fixed monoidal structure, to an open-closed modular functor, and each term $F_V^N$ is finite.
	\end{corollary}

This is a far-reaching finiteness result for rational conformal field theories. Let us slightly rephrase: Given a strongly rational vertex operator algebra $V$, suppose that we want to build spaces of conformal blocks with boundary labels being the simple $V$-modules, and to convert incoming to outgoing labels via the contragredient $X\mapsto X'$. The construction is to be defined on \emph{all} open-closed surfaces, with or without boundary, at all genera in a coherent way such that one obtains an open-closed modular functor. Apart from that, we are free on \emph{how} to do the construction.
Then Corollary~\ref{corfiltration} says the following: Once we ask that the dimension at disks with three marked intervals labeled with arbitrary simple objects does not exceed some fixed $N$, there remain only finitely many possibilities to do the construction.

		\section{Singular correlators for the triplet $\cat{W}_{2,3}$ at $c=0$\label{sectriplet23}}
		For a cyclic subcategory of modules over a vertex operator algebra $V$,
		Theorem~\ref{thmmain} gives us a modular functor with singularities $\block$.
		After constructing this modular functor, the next task is the construction correlators, i.e.\ vectors in the spaces of conformal blocks invariant under the mapping class group actions and sewing.
		As explained in the introduction,
		this is part of the passage from chiral to full conformal field theory, see e.g.~\cite{algcften} for more background.
		
		The construction of correlators will be achieved in this section using the modular microcosm principle \cite{microcosm}
		and will be adapted to the modular functor with singularities that we have already provided.
		In the general case, the modular functors built in this article do not allow for the removal of singularities, and this will imply also that it will be natural to ask a little less from consistent systems of correlators, simply because certain types of sewings do not exist. As a consequence, the construction in this section does not recover all aspects of the correlator construction in the rational \cite{frs1,ffrsunique} and the finite rigid case \cite{correlators}.

	We argue that the notion of correlators that we consider in this section beyond the rigid case is the
	 `correct' one, and our main piece of evidence is the following:
		The guiding example will be the triplet $\cat{W}_{2,3}$ at $c=0$ 
		from the papers \cite{grw,grw-proc} that are among the first exploratory works beyond the rigid case that have, since then, inspired the development of tools for the non-rigid case
		\cite{alsw,fsswfrobenius,microcosm}.
		For this example, a notion of boundary condition is proposed in \cite{grw} based on physical considerations, hoping that it will eventually lead to a construction of a bulk theory, more precisely a full conformal field theory described through a consistent system of correlators.
		This goal is expressed without discussing in detail spaces of conformal blocks or notions of correlators since such a framework was not known for the non-rigid case at the time. 
		The main insight of this section will be that for the modular functor with singularities and the notion of correlator we will present momentarily, the 
		boundary conditions from
		\cite{grw,grw-proc}, up to a small correction in the definition,
		provide each a system of correlators.

			\subsection{Boundary conditions of the triplet $\cat{W}_{2,3}$ at $c=0$}
		The triplet $\cat{W}_{2,3}$ at $c=0$ gives rise to ribbon Grothendieck-Verdier category $\Wmod$
		in $\Rexf$ by \cite{grw,grw-proc,alsw} with right exact, but not exact
		monoidal product $\otimes$. 
		By Proposition~\ref{propc2cofinite} the ind completion $\RW$ 
		of this ribbon Grothendieck-Verdier category is a cyclic framed $E_2$-algebra in presentable categories.
		
		Actually, $\Wmod$ gives us also
		a ribbon Grothendieck-Verdier category in $\Lexf$ with the conventions from \cite[Section~4.2]{cyclic}.
		The underlying linear category is the same, but the monoidal product of is given by $X\odot Y=(Y'\otimes X')'$ for $X,Y\in \Wmod$. The dualizing object is the vertex operator algebra $\cat{W}$ itself while the unit is its contragredient $\cat{W}'$.

	                    	\begin{definition}[$\text{Boundary theory for the triplet, based on \cite{grw,grw-proc}}$]\label{grwboundarycondition} A \emph{GRW boundary theory} for $\cat{W}_{2,3}$ is an associative algebra $F\in \Wmod$ for $\otimes$ together with a non-degenerate pairing $\beta : F\otimes F\to \cat{W}'$ that is invariant for 
	                    		the       multiplication and also symmetric.
		\end{definition}
		Symmetry means that
		\begin{align}
			F\otimes F \ra{\theta_F \otimes \id_F} F\otimes F \ra{c_{F,F}^{-1}}F\otimes F \ra{\beta}\cat{W}'
		\end{align}
		agrees again with $\beta$. In other words, $\beta$ is preserved by the action of the pivotal structure on $\Hom_{\cat{W}}(F\otimes F,\cat{W}')$~\cite[Section~5.1]{cyclic}. 
		
		\begin{remark}	The symmetry condition was not considered in \cite{grw,grw-proc}, but we believe this to be an oversight. 
			After all, $\beta$ is the supposed to be the non-degenerate two-point function for the disk, and since two intervals placed on the boundary of a disk are necessarily symmetric, the symmetry of $\beta$ is a topological requirement.\end{remark}

			\begin{example}[Chiral symmetry algebra]
			Inside $\Wmod$, there is an object  $\cat{W}(0)$
			called the \emph{chiral symmetry algebra} \cite[Section~4.2 and Lemma B.1]{grw-proc} that is 
			\begin{itemize}
				\item a non-trivial quotient of the monoidal unit $\cat{W}$,
				\label{itemquot}
				\item and a braided commutative algebra with the multiplication being an isomorphism $\mu:\cat{W}(0) \otimes \cat{W}(0)\cong \cat{W}(0)$,
				\item that moreover comes with a map $\varepsilon:\cat{W}(0)\to \cat{W}'$ such that the pairing $\beta = \varepsilon \circ \mu$ is non-degenerate (and automatically invariant for the multiplication).
			\end{itemize}
			Since $\cat{W}(0)$ is a quotient of the unit, the balancing of $\cat{W}(0)$ is trivial: $\theta_{\cat{W}(0)}=\id_{\cat{W}(0)}$. In combination with the commutativity of $\mu$, this tells us 
			that $\beta$ is symmetric. 
			This means that the chiral symmetry algebra is 
			 a GRW boundary theory.
		\end{example}

	\begin{question}\label{q5}
		How can one classify all GRW boundary conditions?
	\end{question}

\begin{remark}
	As explained in \cite[Section~4.3]{grw-proc}, both $\cat{W}$ and $\cat{W}'$ are braided commutative algebras, but they are not isomorphic to their Grothendieck-Verdier dual, and hence cannot be equipped with a non-degenerate symmetric pairing.
\end{remark}
		
		\subsection{Consistent systems of singular correlators}
		Denote by $\Omega$ a $\PrL$-valued open-closed modular functor with singularities (Definition~\ref{defmfsingularities})
		with $\cat{A}\in\PrL$ for the open sector and $\cat{B}\in\PrL$ for the closed one.
		In other words, for an operation $\Sigma \in \sing(n,m)$ in the singular surface operad, we obtain a map 
		\begin{align}\Omega(\Sigma;-):\cat{A}^{\boxtimes n}\boxtimes\cat{B}^{\boxtimes m}\to\Vect \ . 
			\end{align}
		The coevaluation objects are $\Delta \in \cat{A}\boxtimes \cat{A}$ and $\Delta_\circ \in \cat{B}\boxtimes\cat{B}$, but only $\Delta$ is 
		 the copairing to a non-degenerate pairing $\kappa : \cat{A}\boxtimes \cat{A}\to\Vect$.
		By symmetry there are canonical isomorphisms $\tau \Delta \cong \Delta$ and $\tau \Delta_\circ \cong \Delta_\circ$ for the symmetric braiding $\tau$
		of $\PrL$.
		Recall that for an object $X\in\cat{A}$, we call a map $\nu: \Delta \to X\boxtimes X$ \emph{symmetric} \cite[Definition 4.1]{microcosm} if
		\begin{align}
			\Delta \cong \tau \Delta \ra{\tau(\nu)} \tau (X\boxtimes X)\cong X \boxtimes X
			\end{align} agrees with $\nu$.
		In \cite[Definition 4.5]{microcosm}, it is defined when such a symmetric map is \emph{non-degenerate} (this is done in reference to the pairing $\kappa:\cat{A}\boxtimes\cat{A}\to\Vect$).
		An object $X$ equipped with such a non-degenerate symmetric map $\Delta \to X\boxtimes X$ is then referred to as a \emph{self-dual object}.
		
		Similarly, we can define symmetric maps $\Delta_\circ \to Y\boxtimes Y$ for $Y\in\cat{B}$. Since $\Delta_\circ$ is not the copairing for a non-degenerate pairing, no self-duality for $Y$ in the strong sense can be defined. We will refer to $Y$ in this case as a \emph{symmetric object}.
		
		 For the singular surface operad $\sing : \TwoGraphs \to \Cat$ (Section~\ref{secmfsing}), denote by
		 $\int \sing$ its Grothendieck construction \cite[Section I.5]{maclanemoerdijk}, i.e.\ the symmetric monoidal category of pairs $(T,\Sigma)$ with $T\in \TwoGraphs$ and $\Sigma\in \sing(T)$. Roughly, this is the category obtained from all operations in $\sing$ with morphisms coming from the isomorphisms between operations and gluing operations. It can be seen as a moduli space of operations in $\sing$, see \cite[Remark 3.4]{microcosm}.
		
		 If $X\in \cat{A}$ and $Y\in\cat{B}$ are symmetric objects,
		 then by the generalization of \cite[Proposition 4.7]{microcosm}
		 \begin{align}
		 \label{OmegaXYeqn}	\Omega^{X,Y} : \left(    \int\sing\right)^\op \to \Vect \ , \quad (n,m,\Sigma) \mapsto \Omega(\Sigma;X^{\boxtimes n},Y^{\boxtimes m})
		\end{align}
	extends to a symmetric monoidal functor;
	the reversal of the arrows in the source category in comparison to \cite{microcosm} comes from the fact that we work with cocontinuous functors 
	 instead of continuous ones.
	
	\begin{definition}
		For an open-closed modular functor $\Omega$
		with singularities with open and closed sector $\cat{A}, \cat{B} \in\PrL$, respectively, a \emph{consistent system of singular correlators} is
		\begin{itemize}
			\item a self-dual object $F \in \cat{A}$, the so-called \emph{boundary field},
			\item a symmetric object $B\in\cat{B}$, the so-called \emph{bulk field},
			\item plus a monoidal transformation $\lambda^{F,B} : \Omega^{X,Y} \to k$ of symmetric monoidal functors 
	$\left(    \int\sing\right)^\op \to \Vect$
	whose component
	\begin{align}
		\lambda^{F,B}_\Sigma : \Omega(\Sigma;X^{\boxtimes n},Y^{\boxtimes m}) \to k\label{eqnlambdaFB}
		\end{align} 
	at a surface $\Sigma \in \sing(n,m)$
	is referred to as the \emph{correlator} at the surface $\Sigma$.
	\end{itemize}
\end{definition}

The linear forms \eqref{eqnlambdaFB}
are by definition invariant under the mapping class group of $\Sigma$ and all sewing operations allowed in the singular surface operad; both is packaged into the naturality as in \cite{jfcs,microcosm}.
The fact that the correlators are linear forms on spaces of conformal blocks rather than vectors inside them is again a consequence of the fact that we work with cocontinuous functors instead of continuous ones.
This dualization is discussed briefly for ansular correlators in \cite[Section~3.5.1]{yeralthesis}. 
	
	\subsection{The construction of singular correlators}
	We are now in a position to prove that the boundary theories in the sense of \cite{grw,grw-proc} actually produce consistent systems of correlators for the singular 
	modular functor built from the triplet $\cat{W}_{2,3}$ through the construction of this article:
	
	\begin{theorem}\label{thmW23}
		Every GRW boundary condition gives rise to a 
		consistent system of singular correlators
		for the singular modular functor built from the triplet
		$\cat{W}_{2,3}$.
		In particular, the chiral symmetry algebra gives rise to such a system.
		\end{theorem}
	
	The statement can be proved more generally. To this end, we recall the following notion:
	
	\begin{definition}[$\text{following \cite{fsswfrobenius,microcosm}}$]\label{deffrobenius}
		A \emph{symmetric Frobenius algebra} $F$ in a pivotal Grothendieck-Verdier category $\cat{A}$ in $\Rexf$
		is a unital associative algebra in $\cat{A}$ equipped with a map
		$\beta : F \otimes F \to K$ called \emph{pairing} that is \begin{itemize}
			\item \emph{non-degenerate} in the sense that the map $\psi:F\to DF$ corresponding to $\beta$ is an isomorphism,
			\item  \emph{symmetric} in the sense that $D\psi : D^2 F \to F$ agrees with $\psi$ after identifying $F\cong D^2 F$ using the pivotal structure,\item
			and \emph{invariant} for the multiplication $\mu :F \otimes F\to F$ and its unit $\eta : I \to F$ in the sense that $\beta(\eta,\mu)=\beta$.
		\end{itemize}
	\end{definition}

\begin{lemma}\label{lemmasymmetricpairing}
	If the pivotal Grothendieck-Verdier category in Definition~\ref{deffrobenius} is a ribbon Grothendieck-Verdier category, the symmetry requirement is equivalent to the requirement for 
	\begin{align}
		F\otimes F \ra{\theta_F \otimes \id_F} F\otimes F \ra{c_{F,F}^{-1}}F\otimes F \ra{\beta}K
	\end{align}
	to agree with $\beta$. In particular, if $F$ is commutative, the requirement is equivalent to $\theta_F = \id_F$. 
\end{lemma}

\begin{proof}
	This is a consequence of the connection between the balancing, the braiding and the pivotal structure, see \cite[Section~5.4]{cyclic}.
	\end{proof}

This leads us to the following description of GRW boundary conditions:
\begin{corollary}\label{corgrw}
	GRW boundary theories for $\cat{W}_{2,3}$ are exactly symmetric Frobenius algebras in the ribbon Grothendieck-Verdier category $\Wmod$ in $\Rexf$.
	\end{corollary}

Let $\cat{A}$ be a pivotal Grothendieck-Verdier category in $\Rexf$ and $F\in\cat{A}$ a self-dual object, i.e.\
an object with a symmetric non-degenerate pairing $\beta : F \otimes F\to K$ as in Definition~\ref{deffrobenius}.
	The pairing $\beta$ induces a symmetric map
	\begin{align}
		\Delta = \int_{X\in\cat{A}} DX\boxtimes X \ra{\text{structure map of the end}} DF \boxtimes F \ra{\psi\boxtimes F} \ F \boxtimes F \ . \label{eqncopairingbeta}
	\end{align}
For the open surface operad $\O:\Graphs \to \Cat$, denote by $\int \O$ its Grothendieck construction.
	By means of this map,
	 one may define a symmetric monoidal functor \begin{align}
	\Ao^F:	\left(\int \O\right)^\op \to \vect\ , \quad \Sigma \mapsto \Ao(\Sigma;F,\dots,F) \ ;
	\end{align} 
	again, this is dual to \cite[Proposition 4.7]{microcosm}
	and analogous to $\Omega^{X,Y}$ in~\eqref{OmegaXYeqn}.

\begin{theorem}[Classification of open correlators]\label{thmopenmicrocosm}
	For any symmetric Frobenius algebra $F\in\cat{A}$ in a pivotal Grothendieck-Verdier category in $\Rexf$,
	the Frobenius structure of $F$ induces linear forms
	\begin{align}
		\lambda_\Sigma^F : \Ao(\Sigma;F,\dots,F) \to k \quad \text{for} \quad \Sigma \in \O(T) 
	\end{align}
	that produce a monoidal natural transformation $\Ao(-;F,\dots,F)\to k$ of symmetric monoidal functors $\left(\int\O\right)^\op\to \vect$, i.e.\ a consistent system of open correlators for the open conformal field theory with monodromy data $\cat{A}$.
	All such systems arise from symmetric Frobenius algebras in $\cat{A}$.  
\end{theorem}

This statement is the dual version of~\cite[Theorem~8.3]{microcosm} that is proved using the modular microcosm principle and its compatibility with modular extensions. On a disk $\mathbb{D}^2_\ell$ with $\ell\ge 0$ marked intervals,
$\lambda_{\mathbb{D}^2_\ell}:\Ao(\mathbb{D}^2_\ell;F,\dots,F)=\cat{A}(F^{\otimes \ell},K)^* \to k$ is a vector in
$\cat{A}(F^{\otimes \ell},K)$, namely the total arity $\ell$ operation 
	\begin{align}
	\nu_\ell : F^{\otimes \ell} \to V' \label{eqnnuell} \end{align} of $F$; this is dual to \cite[eq.~(2.5)]{correlators}.

Thanks to Corollary~\ref{corgrw},
Theorem~\ref{thmW23} above is now a special case of the following main result of this section:

\begin{theorem}\label{thmcor}
	If $V$ is $C_2$-cofinite and if we choose for the cyclic subcategory of $V$-modules the projective generalized grading-restricted $V$-modules, then any compact symmetric Frobenius algebra $F\in \RV$
	gives rise to a consistent system of correlators for the modular functor $\block$ with singularities.
	The singular bulk object is $LF\in\int_{\mathbb{S}^1} \RV$. Explicitly, $F$ gives us
	mapping class group invariant linear forms
	\begin{align} \label{eqnlambdaF}\lambda^F_\Sigma: \block(\Sigma ; F^{\boxtimes n},(LF)^{\boxtimes m})\to k
	\end{align}
	for all $\Sigma \in \sing(n,m)$ compatible with all sewing operations in the singular surface operad.
\end{theorem}

\begin{proof}
	The compact objects
	 $\Vmod$ of $\RV$ form a ribbon Grothendieck-Verdier category in $\Rexf$ (Proposition~\ref{propfinite}), and $F$, by virtue of being compact, is a symmetric Frobenius algebra in $\Vmod$.
	For $\Sigma \in \sing (n,m)$, denote by $\Sigma_*$ the surface obtained by placing an interval around the point on the $n$ boundary circles corresponding to a fixed base point on $\mathbb{S}^1$ under the parametrization.
	Then $\Map(\Sigma)=\Map(\Sigma_*)$ and 
	$\block(\Sigma ; F^{\boxtimes n},(LF)^{\boxtimes m})$ can canonically be identified with $\Vo(\Sigma;F^{\boxtimes (n+m)})$. 
	Now Theorem~\ref{thmopenmicrocosm} gives us the 
	needed mapping class group invariant linear forms~\eqref{eqnlambdaF}, and also
	the gluing compatibility for gluing along marked intervals.
	Moreover, the gluing maps for the gluing along boundary components (recall that in this case a puncture has to be added) are defined such that the gluing constraints are satisfied with respect to the symmetric map
	\begin{align}
		\Delta_\circ = (L\boxtimes L)\Delta \to LF \boxtimes LF
		\end{align} obtained by applying $L:\RV \to \int_{\mathbb{S}^1}\RV$ to 
	$\Delta \to F\boxtimes F$.
\end{proof}

	\begin{example}\label{exampletorus}
		Let $\mathbb{T}^2_r$ be a torus with $r\ge 1$ punctures.
		Then
		\begin{align}
			\block (  \mathbb{T}^2_r  )\cong \Hom_V(    \mathbb{A}^{\otimes (r+1)}, V'     )^*\label{eqnblockontorus}
		\end{align} with
		$\mathbb{A}= \otimes \left( \int_{X\in\Vmod} X' \boxtimes X\right)\neq 0$. 
		In particular, this vector space comes with an action of the torus braid group generalizing the one from \cite{brochierjordane}.
		By applying $\otimes$ to the structure map $\int_{X\in\Vmod} X' \boxtimes X\to F'\boxtimes  F$ for a 
		symmetric Frobenius algebra
		 $F$, we obtain a map
		\begin{align}
			\pi_F: \mathbb{A} \to F' \otimes F 
		\end{align}
		and by Theorem~\ref{thmcor}
		 a mapping class group invariant map
		\begin{align}
			\lambda^{F}_{  \mathbb{T}^2_r }:	\mathbb{A}^{\otimes (r+1)}\to V' 
		\end{align} 
	(because we obtain a linear form on \eqref{eqnblockontorus})
	that in \cite{correlators} (treating  the finite rigid case) is referred to as a \emph{generalized class function}.
		For $r=1$, it is explicitly given by 
		\begin{align}
			\mathbb{A} \otimes \mathbb{A} \ra{\pi_F \otimes \pi_F} F'\otimes F \otimes F'\otimes F \ra{\text{self-duality}} F^{\otimes 4} \ra{\id \otimes  c_{F,F}^{-1}\otimes \id } F^{\otimes 4} \ra{\mu\otimes \mu} F^{\otimes 2} \ra{\beta} V' \ . 
		\end{align}
	Note that $\beta \circ ( \mu\otimes \mu)=\nu_4 : F^{\otimes 4} \to V'$
	is the total arity four operation of $F$ from \eqref{eqnnuell}.
	\end{example}
	
We can push the considerations of Example~\ref{exampletorus} significantly further: For $\ell\ge 0$ and a symmetric Frobenius algebra $F$ in a pivotal Grothendieck-Verdier category $\cat{A}$ in $\Rexf$, define the map
\begin{align}
	\vartheta_\ell^F : \mathbb{A}^{\otimes \ell} \ra{\pi_F^{\otimes \ell}} DF \otimes F \ra{\text{self-duality}} F^{\otimes 2\ell} \ra{\nu_{2\ell}} K \ . 
	\end{align}
By means of these maps, that use only the structure maps of $F$, we can give a closed formula for correlators of arbitrary genus if $F$ is commutative:

\begin{corollary}\label{corcalculate}
	In addition to the assumptions in Theorem~\ref{thmcor}, let the compact symmetric Frobenius algebra $F\in \RV$ be commutative.
	For a surface $\Sigma_{g,p}$ of genus $g$ with $p\ge 1$ punctures, denote by $B: \block(\Sigma _{g,p}) \ra{\cong} \block(\mathbb{D}^{2}_{2g+p})$ an isomorphism to the space of conformal blocks of a disk with $2g+p$ punctures induced by the braiding via Proposition~\ref{propstability}.
	Then the correlator for $\Sigma_{g,p}$ satisfies
	\begin{align}
		\lambda_\Sigma^F = 	\vartheta_{2g+p}^F \circ B \ .
		\end{align}
	\end{corollary}

\begin{proof}
	By the comments after Theorem~\ref{thmopenmicrocosm}
	\begin{align}
			\vartheta_{2g+p}^F = \lambda_{\mathbb{D}^{2}_{2g+p}}^F
		\end{align}
	as linear forms $\block(\mathbb{D}^{2}_{2g+p}) \to k$, i.e.\ maps $\mathbb{A}^{\otimes (2g+p)} \to V'$. 
	The commutativity of $F$ implies that the braiding operations in Proposition~\ref{propstability} preserve the correlators: 
	\begin{align} 	\lambda_\Sigma^F = \lambda_{\mathbb{D}^{2}_{2g+p}}^F \circ B \ . 
		\end{align}
	The combination of these two facts yields the assertion.
	\end{proof}

This provides in the case of commutative $F$ (in particular for $\cat{W}(0)$) a closed expression for higher genus correlators in terms of the braiding of $\RV$, the structure maps of the end and the structure maps of $F$. Even in the finite rigid case, this seems to be new.

	\vspace*{0.3cm} \noindent  \textsc{Ludwig-Maximilians-Universität München, Department für Physik,	Theresienstrasse 37,	D-80333 München,	Germany} \\ \textsc{Université Bourgogne Europe, CNRS, IMB UMR 5584, F-21000 Dijon, France}


\small	

\begin{thebibliography}{MSWY26}

\bibitem[AF15]{AF}
D.~Ayala and J.~Francis.
\newblock {Factorization homology of topological manifolds}.
\newblock {\em J. Top.}, 8(4):1045--1084, 2015.

\bibitem[AGS96]{agsmoduli}
A.~Alekseev, H.~Grosse, and V.~Schomerus.
\newblock {Combinatorial quantization of the {H}amiltonian {C}hern-{S}imons
  theory {II}}.
\newblock {\em Comm. Math. Phys.}, 174(3):561--604, 1996.

\bibitem[Agu00]{aguiar}
M.~Aguiar.
\newblock {A note on strongly separable algebras}.
\newblock {\em Bol. Acad. Nac. Cienc.}, 65:51--60, 2000.

\bibitem[Ale94]{alekseevmoduli}
A.~Alekseev.
\newblock {Integrability on the {H}amiltonian {C}hern-{S}imons theory}.
\newblock {\em St. Petersburg Math. J.}, 6(2):241--253, 1994.

\bibitem[ALSW25]{alsw}
R.~Allen, S.~Lentner, C.~Schweigert, and S.~Wood.
\newblock {Duality structures for module categories of vertex operator algebras
  and the {F}eigin {F}uchs boson}.
\newblock {\em Selecta Math. New Ser.}, 31(36), 2025.

\bibitem[ARS95]{ars}
M.~Auslander, I.~Reiten, and S.~O. Smal\o.
\newblock {\em Representation Theory of Artin Algebras}, volume~36 of {\em
  Cambridge Studies Adv. Math.}
\newblock Cambridge University Press, 1995.

\bibitem[AS96]{asmoduli}
A.~Alekseev and V.~Schomerus.
\newblock {Representation theory of {C}hern-{S}imons observables}.
\newblock {\em Duke Math. J.}, 85(2):447--510, 1996.

\bibitem[AU07]{andersenuenomf}
J.~E. Andersen and K.~Ueno.
\newblock {Geometric construction of modular functors from Conformal Field
  Theory}.
\newblock {\em J. Knot Theory Ramif.}, 16:127--202, 2007.

\bibitem[AU15]{andersenuenort}
J.~E. Andersen and K.~Ueno.
\newblock {Construction of the Witten-Reshetikhin-Turaev TQFT from conformal
  field theory}.
\newblock {\em Invent. Math.}, 2:519--559, 2015.

\bibitem[Bar79]{barr}
M.~Barr.
\newblock {\em $\star$-autonomous categories}, volume 572 of {\em Lecture Notes
  in Math.}
\newblock Springer, 1979.

\bibitem[BCJF15]{bcjf}
M.~Brandenburg, A.~Chirvasitu, and T.~Johnson-Freyd.
\newblock Reflexivity and dualizability in categorified linear algebra.
\newblock {\em Theory App. Cat.}, 30(23):808--835, 2015.

\bibitem[BD13]{bd}
M.~Boyarchenko and V.~Drinfeld.
\newblock A duality formalism in the spirit of {G}rothendieck and {V}erdier.
\newblock {\em Quantum Top.}, 4(4):447--489, 2013.

\bibitem[BFM]{BFM}
A.~Beilinson, B.~Feigin, and B.~Mazur.
\newblock Notes on conformal field theory. 1991.
\newblock Available at
  \href{https://www.math.stonybrook.edu/~kirillov/manuscripts/bfmn.pdf}{https://www.math.stonybrook.edu/~kirillov/manuscripts/bfmn.pdf}.

\bibitem[BJ17]{brochierjordane}
A.~Brochier and D.~Jordan.
\newblock {Fourier transform for quantum $D$-modules via the punctured torus
  mapping class group}.
\newblock {\em Quantum Top.}, 8:361--379, 2017.

\bibitem[BJSS21]{bjss}
A.~Brochier, D.~Jordan, P.~Safronov, and N.~Snyder.
\newblock {Invertible braided tensor categories}.
\newblock {\em Alg. Geom. Top.}, 21(4):2107--2140, 2021.

\bibitem[BK00]{bakifm}
B.~Bakalov and A.~Kirillov.
\newblock {On the {L}ego-{T}eichmüller game}.
\newblock {\em Transf. Groups}, 6:207--244, 2000.

\bibitem[BK01]{baki}
B.~Bakalov and A.~Kirillov, Jr.
\newblock {\em Lectures on tensor categories and modular functors}, volume~21
  of {\em University Lecture Series}.
\newblock American Mathematical Society, Providence, RI, 2001.

\bibitem[Bro09]{broue}
M.~Broue.
\newblock Higman's criterion revisited.
\newblock {\em Michigan Math. J}, 58, 2009.

\bibitem[BSZ25]{bsz}
S.~Barkan, J.~Steinebrunner, and A.~Y. Zhang.
\newblock Open {2D TFT}s admit initial open-closed extensions.
\newblock arXiv:2509.02553 [math.AT], 2025.

\bibitem[BW25]{becerrawoike}
J.~Becerra and L.~Woike.
\newblock {Pivotal Module Categories, Factorization Homology and Modular
  Invariant Modified Traces}.
\newblock Accepted for publication in \emph{Math. Z.}, arXiv:2512.19669
  [math.QA], 2025.

\bibitem[BW26]{brochierwoike}
A.~Brochier and L.~Woike.
\newblock {A Classification of Modular Functors via Factorization Homology}.
\newblock {\em Geom. Top.}, 30(6):2157--2216, 2026.

\bibitem[BZBJ18a]{bzbj}
D.~Ben-Zvi, A.~Brochier, and D.~Jordan.
\newblock {Integrating quantum groups over surfaces}.
\newblock {\em J. Top.}, 11(4):874--917, 2018.

\bibitem[BZBJ18b]{bzbj2}
D.~Ben-Zvi, A.~Brochier, and D.~Jordan.
\newblock {Quantum character varieties and braided module categories}.
\newblock {\em Selecta Math.}, 24:4711--4748, 2018.

\bibitem[CD11]{coeckeduncan}
B.~Coecke and R.~Duncan.
\newblock Interacting quantum observables: categorical algebra and
  diagrammatics.
\newblock {\em New J. Phys.}, 13:043016, 2011.

\bibitem[CGPM23]{asm}
F.~Costantino, N.~Geer, and B.~Patureau-Mirand.
\newblock Admissible skein modules.
\newblock arXiv:2302.04493 [math.GT], 2023.

\bibitem[CJF13]{chirjf}
A.~Chirvasitu and T.~Johnson-Freyd.
\newblock {The fundamental pro-groupoid of an affine 2-scheme}.
\newblock {\em Appl. Categor. Struct.}, 21:469--22, 2013.

\bibitem[CMSY25]{crsy}
T.~Creutzig, R.~McRae, K.~Shimizu, and H.~Yadav.
\newblock Commutative algebras in {G}rothendieck-{V}erdier categories,
  rigidity, and vertex operator algebras.
\newblock {\em Comm. Contp. Math.}, 2025.

\bibitem[Cos04]{costello}
K.~Costello.
\newblock The {A}-infinity operad and the moduli space of curves.
\newblock arXiv:math/0402015 [math.AG], 2004.

\bibitem[Cos07a]{costellographs}
K.~Costello.
\newblock A dual version of the ribbon graph decomposition of moduli space.
\newblock {\em Geom. Top.}, 11(3):1637--1652, 2007.

\bibitem[Cos07b]{costellotcft}
K.~Costello.
\newblock {Topological conformal field theories and Calabi-Yau categories}.
\newblock {\em Adv. Math.}, 210(1):165--214, 2007.

\bibitem[Day70]{daypromonoidal}
B.~Day.
\newblock {\em On closed categories of functors}, volume 137 of {\em Lecture
  Notes Math.}
\newblock Springer, 1970.

\bibitem[DGT21]{DGT1}
C.~Damiolini, A.~Gibney, and N.~Tarasca.
\newblock Conformal blocks from vertex algebras and their connections on
  {$\overline{\mathcal M}_{g, n}$}.
\newblock {\em Geom. Topol.}, 25(5):2235--2286, 2021.

\bibitem[DGT24]{DGT2}
C.~Damiolini, A.~Gibney, and N.~Tarasca.
\newblock On factorization and vector bundles of conformal blocks from vertex
  algebras.
\newblock {\em Ann. Sci. \'Ec. Norm. Sup\'er. (4)}, 57(1):241--292, 2024.

\bibitem[DW25]{damioliniwoike}
C.~Damiolini and L.~Woike.
\newblock Modular functors from conformal blocks of rational vertex operator
  algebras.
\newblock arXiv:2507.05845 [math.QA], 2025.

\bibitem[EGNO15]{egno}
P.~Etingof, S.~Gelaki, D.~Nikshych, and V.~Ostrik.
\newblock {\em Tensor categories}, volume 205 of {\em Math. Surveys Monogr.}
\newblock Am. Math. Soc., 2015.

\bibitem[ENO04]{eno-d}
P.~Etingof, D.~Nikshych, and V.~Ostrik.
\newblock {An analogue of Radford's $S^4$ formula for finite tensor
  categories}.
\newblock {\em Int. Math. Res. Not.}, 2004(54):2915--2933, 2004.

\bibitem[EO04]{etingofostrik}
P.~Etingof and V.~Ostrik.
\newblock {Finite tensor categories}.
\newblock {\em Mosc. Math. J.}, 4(3):627--654, 2004.

\bibitem[EP26]{etingofpenneys}
P.~Etingof and D.~Penneys.
\newblock Rigidity of non-negligible objects of moderate growth in braided
  categories.
\newblock {\em Forum Math. Pi}, 14(7):1--18, 2026.

\bibitem[ESK24]{egaskupers}
D.~Egas~Santander and S.~Kupers.
\newblock {Comparing combinatorial models of moduli space and their
  compactifications}.
\newblock {\em Alg. Geom. Top.}, 24(2):595--654, 2024.

\bibitem[FBZ04]{frenkelbenzvi}
E.~Frenkel and D.~Ben-Zvi.
\newblock {\em Vertex algebras and algebraic curves}, volume~88 of {\em
  Mathematical Surveys and Monographs}.
\newblock American Mathematical Society, Providence, RI, second edition, 2004.

\bibitem[FFRS08]{ffrsunique}
J.~Fjelstad, J.~Fuchs, I.~Runkel, and C.~Schweigert.
\newblock {Uniqueness of open/closed rational CFT with given algebra of open
  states}.
\newblock {\em Adv. Theor. Math. Phys.}, 12:1283--1375, 2008.

\bibitem[FHL93]{fhl}
I.~B. Frenkel, Y.-Z. Huang, and L.~Lepowsky.
\newblock On axiomatic approaches to vertex operator algebras and modules.
\newblock {\em Mem. Amer. Math. Soc.}, 104(494):viii+64, 1993.

\bibitem[FM12]{farbmargalit}
B.~Farb and D.~Margalit.
\newblock {\em A Primer on Mapping Class Groups}, volume~49 of {\em Princeton
  Math. Series}.
\newblock Princeton University Press, 2012.

\bibitem[Fra13]{Franco2013}
Ignacio~L{\'o}pez Franco.
\newblock Tensor products of finitely cocomplete and abelian categories.
\newblock {\em J. Alg.}, 396:207--219, 2013.

\bibitem[FRS02]{frs1}
J.~Fuchs, I.~Runkel, and C.~Schweigert.
\newblock {TFT construction of RCFT correlators. I: Partition functions}.
\newblock {\em Nucl. Phys. B}, 646(3):353--497, 2002.

\bibitem[FRS10]{csrcft}
J.~Fuchs, I.~Runkel, and C.~Schweigert.
\newblock {Twenty five years of two-dimensional rational conformal field
  theory}.
\newblock {\em J. Math. Phys.}, 51(015210), 2010.

\bibitem[FS17a]{coendFS}
J.~Fuchs and C.~Schweigert.
\newblock Coends in conformal field theory.
\newblock In K.~Barron, E.~Jurisich, A.~Milas, and K.~Misra, editors, {\em Lie
  Algebras, Vertex Operator Algebras, and Related Topics}, pages 65--81. Amer.
  Math. Soc., 2017.

\bibitem[FS17b]{jfcs}
J.~Fuchs and C.~Schweigert.
\newblock {Consistent systems of correlators in non-semisimple conformal field
  theory}.
\newblock {\em Adv. Math.}, 307:598--639, 2017.

\bibitem[FS21]{fspivotal}
J.~Fuchs and C.~Schweigert.
\newblock Bulk from boundary in finite cft by means of pivotal module
  categories.
\newblock {\em Nucl. Phys. B}, 967:115392, 2021.

\bibitem[FSS20]{fss}
J.~Fuchs, G.~Schaumann, and C.~Schweigert.
\newblock {Eilenberg-Watts calculus for finite categories and a bimodule
  Radford $S^4$ theorem}.
\newblock {\em Trans. Amer. Math. Soc.}, 373(1):1--40, 2020.

\bibitem[FSSW25]{fsswfrobenius}
J.~Fuchs, G.~Schaumann, C.~Schweigert, and S.~Wood.
\newblock Grothendieck-{V}erdier module categories, {F}robenius algebras and
  relative {S}erre functors.
\newblock {\em Adv. Math.}, 475:110325, 2025.

\bibitem[FSV13]{fsv}
J.~Fuchs, C.~Schweigert, and A.~Valentino.
\newblock Bicategories for boundary conditions and for surface defects in 3-d
  {TFT}.
\newblock {\em Comm. Math. Phys.}, 321:543--575, 2013.

\bibitem[FSWY25]{algcften}
J.~Fuchs, C.~Schweigert, S.~Wood, and Y.~Yang.
\newblock Algebraic structures in two-dimensional conformal field theory.
\newblock In R.~Szabo and M.~Bojowald, editors, {\em Encyclopedia of
  Mathematical Physics (Second Edition, Volume 3)}, pages 604--617. Elsevier,
  2025.

\bibitem[FSY22]{rcftsn}
J.~Fuchs, C.~Schweigert, and Y.~Yang.
\newblock {\em String-Net Construction of RCFT Correlators}, volume~45 of {\em
  SpringerBriefs Math. Phys.}
\newblock Springer, 2022.

\bibitem[Gia11]{giansiracusa}
J.~Giansiracusa.
\newblock {The framed little 2-discs operad and diffeomorphisms of
  handlebodies}.
\newblock {\em J. Top.}, 4(4):919--941, 2011.

\bibitem[GJS23]{skeinfin}
S.~Gunningham, D.~Jordan, and P.~Safronov.
\newblock {The finiteness conjecture for skein modules}.
\newblock {\em Invent. Math.}, 232:301--363, 2023.

\bibitem[GK95]{gk}
E.~Getzler and M.~Kapranov.
\newblock Cyclic operads and cyclic homology.
\newblock In R.~Bott and S.-T. Yau, editors, {\em Conference proceedings and
  lecture notes in geometry and topology}, pages 167--201. Int. Press, 1995.

\bibitem[GK98]{gkmod}
E.~Getzler and M.~Kapranov.
\newblock {Modular operads}.
\newblock {\em Compositio Math.}, 110:65--125, 1998.

\bibitem[GKPM22]{mtrace}
N.~Geer, J.~Kujawa, and B.~Patureau-Mirand.
\newblock {M-traces in (non-unimodular) pivotal categories}.
\newblock {\em Algebr. Represent. Theor. (online first)}, 25:759--776, 2022.

\bibitem[GN24]{gannonnegron}
T.~Gannon and C.~Negron.
\newblock {Quantum SL(2) and logarithmic vertex operator algebras at
  $(p,1)$-central charge}.
\newblock {\em J. Eur. Math. Soc.}, (online first), 2024.

\bibitem[GPMV13]{mtrace3}
N.~Geer, B.~Patureau-Mirand, and A.~Virelizier.
\newblock {Traces on ideals in pivotal categories}.
\newblock {\em Quantum Topol.}, 4(1):91--124, 2013.

\bibitem[Gra25]{grabowski}
J.~Grabowski.
\newblock {\em Representation Theory, A Categorical Approach}.
\newblock Open Book Publishers, 2025.

\bibitem[GRW09]{grw}
M.~Gaberdiel, I.~Runkel, and S.~Wood.
\newblock {Fusion rules and boundary conditions in the $c = 0$ triplet model}.
\newblock {\em J. Phys. A}, 42(32):325--403, 2009.

\bibitem[GZ26]{guizhang}
B.~Gui and H.~Zhang.
\newblock Analytic conformal blocks of $c_2$-cofinite vertex operator algebras
  {III}: The sewing-factorization theorems.
\newblock 2026.

\bibitem[Har86]{harer86}
J.~Harer.
\newblock {The virtual cohomological dimension of the mapping class group of an
  orientable surface}.
\newblock {\em Invent. Math.}, 84:157--176, 1986.

\bibitem[Hat65]{hattori}
A.~Hattori.
\newblock Rank element of a projective module.
\newblock {\em Nagoya Math. J.}, (25):113--120, 1965.

\bibitem[HLZ11]{hlzi}
Y.-Z. Huang, J.~Lepowsky, and L.~Zhang.
\newblock Logarithmic tensor category theory, viii: Braided tensor category
  structure on categories of generalized modules for a conformal vertex
  algebra.
\newblock arXiv:1110.1931 [math.QA], 2011.

\bibitem[Hua08]{huang}
Y.-Z. Huang.
\newblock {Rigidity and modularity of vertex tensor categories}.
\newblock {\em Comm. Contemp. Math.}, 10(1):871--911, 2008.

\bibitem[Hua09]{huangfin}
Y.-Z. Huang.
\newblock Cofiniteness conditions, projective covers and the logarithmic tensor
  product theory.
\newblock 2009.

\bibitem[Idr17]{najib}
N.~Idrissi.
\newblock {Swiss-Cheese operad and Drinfeld center}.
\newblock {\em Israel J. Math.}, 221(2):941--972, 2017.

\bibitem[Iva12]{ivanov}
S.~O. Ivanov.
\newblock Nakayama functors and {E}ilenberg-{W}atts theorems.
\newblock {\em J. Math. Sci.}, 183:675--680, 2012.

\bibitem[JF21]{jfheisenberg}
T.~Johnson-Freyd.
\newblock Heisenberg-picture quantum field theory.
\newblock In A.~Alekseev, E.~Frenkel, M.~Rosso, B.~Webster, and M.~Yakimov,
  editors, {\em Progress in Math. vol 340}, pages 371--409. Birkhäuser, 2021.

\bibitem[Kel99]{keller}
B.~Keller.
\newblock {On the cyclic homology of exact categories}.
\newblock {\em J. Pure Appl. Alg.}, 136(1):1--56, 1999.

\bibitem[Kir11]{kirillovsn}
A.~Kirillov.
\newblock String-net model of {T}uraev-{V}iro invariants.
\newblock arXiv:1106.6033 [math.AT], 2011.

\bibitem[KL01]{kl}
T.~Kerler and V.~V. Lyubashenko.
\newblock {\em Non-Semisimple Topological Quantum Field Theories for
  3-Manifolds with Corners}, volume 1765 of {\em Lecture Notes in Math.}
\newblock Springer, 2001.

\bibitem[Koc03]{kocktft}
J.~Kock.
\newblock {\em Frobenius Algebras and 2D Topological Quantum Field Theories},
  volume 2003 of {\em London Math. Soc. Student Texts}.
\newblock Cambridge University Press, 2003.

\bibitem[Kon92]{kontsevichintersection}
M.~Kontsevich.
\newblock {Intersection theory on the moduli space of curves and the matrix
  Airy function}.
\newblock {\em Comm. Math. Phys.}, 147(1):1--23, 1992.

\bibitem[Kon94]{kon94}
M.~Kontsevich.
\newblock Feynman diagrams and low-dimensional topology.
\newblock In A.~Joseph, F.~Mignot, F.~Murat, B.~Prum, and R.~Rentschler,
  editors, {\em First European Congress of Mathematics Paris, July 6--10, 1992:
  Vol. II: Invited Lectures (Part 2)}, pages 97--121. Birkhäuser Basel, 1994.

\bibitem[Kü91]{kuehlshammer}
B.~Kühlshammer.
\newblock Group-theoretical descriptions of ring-theoretical invariants of
  group algebras.
\newblock {\em Prog. Math.}, 95:425--441, 1991.

\bibitem[Laz01]{Lazaroiu}
C.~I. Lazaroiu.
\newblock {On the structure of open-closed topological field theory in two
  dimensions}.
\newblock {\em Nucl. Phys. B}, 603(3):497--530, 2001.

\bibitem[LL04]{LepowskyLi2004}
J.~Lepowsky and H.~Li.
\newblock {\em Introduction to Vertex Operator Algebras and Their
  Representations}, volume 227 of {\em Progress in Mathematics}.
\newblock Birkh{\"a}user Boston, Boston, MA, 2004.

\bibitem[LP07]{lpstatesum}
A.~D. Lauda and H.~Pfeiffer.
\newblock State sum construction of two-dimensional open-closed topological
  quantum field theories.
\newblock {\em J. Knot Theory Ram.}, 16(09):1121--1163, 2007.

\bibitem[LP08]{laudapfeiffer}
A.~D. Lauda and H.~Pfeiffer.
\newblock {Open-closed strings: Two-dimensional extended TQFTs and Frobenius
  algebras}.
\newblock {\em Topology Appl.}, 155(7):623--666, 2008.

\bibitem[LW05]{levinwen}
M.~A. Levin and X.-G. Wen.
\newblock {String-net condensation: A physical mechanism for topological
  phases}.
\newblock {\em Phys. Rev. B}, 71:045110, 2005.

\bibitem[Lyu95a]{lyubacmp}
V.~V. Lyubashenko.
\newblock {Invariants of 3-manifolds and projective representations of mapping
  class groups via quantum groups at roots of unity}.
\newblock {\em Comm. Math. Phys.}, 172:467--516, 1995.

\bibitem[Lyu95b]{lyu}
V.~V. Lyubashenko.
\newblock {Modular transformations for tensor categories}.
\newblock {\em J. Pure Appl. Alg.}, 98:279--327, 1995.

\bibitem[Lyu96]{lyulex}
V.~V. Lyubashenko.
\newblock {Ribbon abelian categories as modular categories}.
\newblock {\em J. Knot Theory and its Ramif.}, 05(03):311--403, 1996.

\bibitem[LZZ12]{lzzhigman}
Y.~Liu, G.~Zhou, and A.~Zimmermann.
\newblock { Higman ideal, stable Hochschild homology and Auslander-Reiten
  conjecture}.
\newblock {\em Math. Z.}, 270:759--781, 2012.

\bibitem[McR26]{mcrae}
R.~McRae.
\newblock {On rationality for {$C_2$}-cofinite vertex operator algebras}.
\newblock {\em Camb. J. Math.}, 14(1):1--115, 2026.

\bibitem[MLM92]{maclanemoerdijk}
S.~Mac~Lane and I.~Moerdijk.
\newblock {\em Sheaves in Geometry and Logic}.
\newblock Springer Universitext. Springer, 1992.

\bibitem[MN26]{mcraenegron}
R.~McRae and C.~Negron.
\newblock Cocompletions for non-abelian vertex tensor categories.
\newblock arXiv:2606.07987 [math.QA], 2026.

\bibitem[MR95]{masbaumroberts}
G.~Masbaum and J.~Roberts.
\newblock {On central extensions of mapping class groups}.
\newblock {\em Math. Ann.}, 302:131--150, 1995.

\bibitem[MS89]{ms89}
G.~Moore and N.~Seiberg.
\newblock {Classical and Quantum Conformal Field Theory}.
\newblock {\em Comm. Math. Phys.}, 123:177--254, 1989.

\bibitem[MS06]{mooresegal}
G.~Moore and G.~Segal.
\newblock D-branes and {K}-theory in {2D} topological field theory.
\newblock arXiv:hep-th/0609042, 2006.

\bibitem[MSWY26]{sn}
L.~Müller, C.~Schweigert, L.~Woike, and Y.~Yang.
\newblock {The {L}yubashenko Modular Functor for {D}rinfeld Centers via
  Non-Semisimple String-Nets}.
\newblock {\em Adv. Math.}, 488:110770, 2026.

\bibitem[M{\"u}l26]{Muller2026GVnotes}
L.~M{\"u}ller.
\newblock An introduction to {G}rothendieck-{V}erdier duality in conformal
  field theory, quantum topology, and representation theory.
\newblock
  \url{https://drive.google.com/file/d/1fPlOjGruRym6I02F-m-R2LAJOr_V5h6H/view},
  2026.
\newblock Lecture notes based on a talk at the workshop \emph{Geometric and
  Category-Theoretic Approaches to Conformal Field Theory}, Chennai
  Mathematical Institute.

\bibitem[MW23a]{cyclic}
L.~Müller and L.~Woike.
\newblock {Cyclic framed little disks algebras, {G}rothendieck-{V}erdier
  duality and handlebody group representations}.
\newblock {\em Quart. J. Math.}, 74(1):163--245, 2023.

\bibitem[MW23b]{mwdiff}
L.~Müller and L.~Woike.
\newblock {The Diffeomorphism Group of the Genus One Handlebody and
  {H}ochschild homology}.
\newblock {\em Proc. Amer. Math. Soc.}, 151(6):2311--2324, 2023.

\bibitem[MW24]{mwansular}
L.~Müller and L.~Woike.
\newblock {Classification of {C}onsistent {S}ystems of {H}andlebody {G}roup
  {R}epresentations}.
\newblock {\em Int. Math. Res. Not.}, 2024(6):4767--4803, 2024.

\bibitem[MW25a]{envas}
L.~Müller and L.~Woike.
\newblock Categorified open topological field theories.
\newblock {\em Proc. Amer. Math. Soc.}, 153:2381--2396, 2025.

\bibitem[MW25b]{mwdehn}
L.~Müller and L.~Woike.
\newblock The {D}ehn twist action for quantum representations of mapping class
  groups.
\newblock 2025.

\bibitem[MW26]{mwskein}
L.~Müller and L.~Woike.
\newblock Admissible skein modules and ansular functors: A comparison.
\newblock {\em Ann. Fac. Sci. Toulouse Math.}, 35(2):517--542, 2026.

\bibitem[{nLa}26]{nlab:promonoidal_category}
{nLab authors}.
\newblock promonoidal category.
\newblock \url{https://ncatlab.org/nlab/show/promonoidal+category}, 2026.
\newblock
  \href{https://ncatlab.org/nlab/revision/promonoidal+category/18}{Revision
  18}.

\bibitem[Pen87]{penner}
R.~C. Penner.
\newblock {The decorated Teichmüller space of punctured surfaces}.
\newblock {\em Comm. Math. Phys.}, 113(2):299--339, 1987.

\bibitem[PS13]{pontoshulman}
K.~Ponto and M.~Shulman.
\newblock {Shadows and traces in bicategories}.
\newblock {\em J. Homotopy Relat. Struct.}, 8:151--200, 2013.

\bibitem[RGW14]{grw-proc}
I.~Runkel, M.~Gaberdiel, and S.~Wood.
\newblock Logarithmic bulk and boundary conformal field theory and the full
  centre construction.
\newblock In Chengming Bai, J{\"u}rgen Fuchs, Yi-Zhi Huang, Liang Kong, Ingo
  Runkel, and Christoph Schweigert, editors, {\em Conformal Field Theories and
  Tensor Categories}, pages 93--168, Berlin, Heidelberg, 2014. Springer Berlin
  Heidelberg.

\bibitem[Rie14]{riehl}
E.~Riehl.
\newblock {\em Categorical Homotopy Theory}, volume~24 of {\em New Math.
  Monogr.}
\newblock Cambridge University Press, 2014.

\bibitem[Rie16]{riehlcat}
E.~Riehl.
\newblock {\em Category Theory in Context}.
\newblock Dover Publications Inc., 2016.

\bibitem[Rob94]{roberts}
J.~Roberts.
\newblock {Skeins and mapping class groups}.
\newblock {\em Math. Proc. Camb. Phil. Soc.}, 115:53--77, 1994.

\bibitem[RT90]{rt1}
N.~Reshetikhin and V.~G. Turaev.
\newblock {Ribbon graphs and their invariants derived from quantum groups}.
\newblock {\em Comm. Math. Phys.}, 127:1--26, 1990.

\bibitem[RT91]{rt2}
N.~Reshetikhin and V.~Turaev.
\newblock {Invariants of 3-manifolds via link polynomials and quantum groups}.
\newblock {\em Invent. Math.}, 103:547--598, 1991.

\bibitem[RV16]{riehlverity}
E.~Riehl and D.~Verity.
\newblock Homotopy coherent adjunctions and the formal theory of monads.
\newblock {\em Adv. Math.}, 286:802--888, 2016.

\bibitem[RW22]{ghost}
Allen R. and S.~Wood.
\newblock Bosonic ghostbusting: The bosonic ghost vertex algebra admits a
  logarithmic module category with rigid fusion.
\newblock {\em Comm. Math. Phys.}, 390:959--1015, 2022.

\bibitem[Seg88]{Segal}
G.~Segal.
\newblock {Two-dimensional conformal field theories and modular functors}.
\newblock In {\em {IX International Conference on Mathematical Physics
  (IAMP)}}, 1988.

\bibitem[Shi17]{shimizucf}
K.~Shimizu.
\newblock {The monoidal center and the character algebra}.
\newblock {\em J. Pure Appl. Alg.}, 221(9):2338--2371, 2017.

\bibitem[Shi19]{shimizumodular}
K.~Shimizu.
\newblock {Non-degeneracy conditions for braided finite tensor categories}.
\newblock {\em Adv. Math.}, 355:106778, 2019.

\bibitem[SP09]{schommerpries}
C.~J. Schommer-Pries.
\newblock {\em The classification of two-dimensional extended topological field
  theories}.
\newblock PhD thesis, Berkeley, 2009.

\bibitem[SS23]{shibatashimizu}
T.~Shibata and K.~Shimizu.
\newblock {Modified Traces and the Nakayama Functor}.
\newblock {\em Alg. Rep. Theory}, 26:513--551, 2023.

\bibitem[Sta65]{stallings}
J.~Stallings.
\newblock Centerless groups --- an algebraic formulation of gottliebs theorem.
\newblock {\em Topology}, 4:129--132, 1965.

\bibitem[Ste25]{steinebrunner}
J.~Steinebrunner.
\newblock 2-dimensional {TFT}s via modular $\infty$-operads.
\newblock arXiv:2506.22104 [math.CT], 2025.

\bibitem[SW23]{tracesw}
C.~Schweigert and L.~Woike.
\newblock The trace field theory of a finite tensor category.
\newblock {\em Alg. Rep. Theory}, 26:1931--1949, 2023.

\bibitem[Til98]{tillmann}
U.~Tillmann.
\newblock {$\mathcal{S}$-Structures for $k$-Linear Categories and the
  Definition of a Modular Functor}.
\newblock {\em J. London Math. Soc.}, 58(1):208--228, 1998.

\bibitem[Tur94]{turaev}
V.~G. Turaev.
\newblock {\em Quantum Invariants of Knots and 3-Manifolds}, volume~18 of {\em
  Studies in Math.}
\newblock De Gruyter, 1994.

\bibitem[TV92]{turaevviro}
V~Turaev and O.~Viro.
\newblock State sum invariants of 3-manifolds and quantum 6j-symbols.
\newblock {\em Topology}, 31(4):865--902, 1992.

\bibitem[Woi25a]{correlators}
L.~Woike.
\newblock The construction of correlators in finite rigid logarithmic conformal
  field theory.
\newblock arXiv:2507.22841 [math.QA], 2025.

\bibitem[Woi25b]{hdr}
L.~Woike.
\newblock {\em {Modular Operads, Factorization Homology and Applications in
  Quantum Topology}}.
\newblock Habilitation {\`a} diriger des recherches, {Universit{\'e} Bourgogne
  Europe}, 2025.
\newblock hal.science/tel-05457957.

\bibitem[Woi25c]{reflection}
L.~Woike.
\newblock Reflection equivariance and the {H}eisenberg picture for spaces of
  conformal blocks.
\newblock arXiv:2507.22820 [math.QA], 2025.

\bibitem[Woi25d]{microcosm}
L.~Woike.
\newblock {The cyclic and modular microcosm principle}.
\newblock {\em Canad. J. Math. (online first)}, 2025.

\bibitem[Woi26]{woikemfo}
L.~Woike.
\newblock Spaces of conformal blocks beyond semisimplicity, rigidity and
  finiteness.
\newblock Oberwolfach Report No. 12/2026, 2026.

\bibitem[WW16]{wahlwesterland}
N.~Wahl and C.~Westerland.
\newblock {Hochschild homology of structured algebras}.
\newblock {\em Adv. Math.}, 288:240--307, 2016.

\bibitem[Yer26a]{yeralthesis}
D.~Yeral.
\newblock {\em Factorization homology, three-manifold invariants and modular
  functors}.
\newblock PhD thesis, Université Bourgogne Europe, 2026.

\bibitem[Yer26b]{yeralopen}
D.~Yeral.
\newblock Open modular functors from non-finite tensor categories.
\newblock arXiv:2602.03519 [math.QA], 2026.

\bibitem[Zha25]{zhang}
H.~Zhang.
\newblock Non-equivalence of smooth and nodal conformal block functors in
  logarithmic {CFT}.
\newblock arXiv:2509.07720 [math.QA], 2025.

\end{thebibliography}
\end{document}